\documentclass[reqno,a4paper,11pt]{amsart}
\usepackage{amsmath,amsthm,amssymb,mathrsfs}
\usepackage[
left = 2.25cm,
right = 2.25cm,
top = 2.2 cm,
bottom = 2.2 cm
]{geometry}
\usepackage{graphicx,xcolor,tikz,mathtools}
\usepackage{bm,stmaryrd}
\usepackage{cite}
\usepackage{amsaddr}
\usepackage{enumitem}
\usepackage[toc,page]{appendix}
\usepackage[
pdfencoding=auto,
colorlinks = true,
citecolor = black,
linkcolor = black,
anchorcolor = black,
urlcolor = black, 
filecolor=black,
pdfkeywords={}
]{hyperref}
\usepackage{subcaption}
\usepackage{esint} %

\newcommand{\mich}[1]{\textcolor{black}{#1}}
\newcommand{\andre}[1]{\textcolor{black}{#1}}
\newcommand{\rev}[1]{\textcolor{black}{#1}}

\theoremstyle{plain}
\newtheorem{theorem}{Theorem}[section]

\newtheorem{lemma}[theorem]{Lemma}

\newtheorem{remark}[theorem]{Remark}
\theoremstyle{definition}

\numberwithin{equation}{section}

\newcommand{\bn}{\mathbf{n}}

\newcommand{\bI}{\mathbf{I}}

\newcommand{\bP}{\mathbf{P}}

\def\vphi{\varphi}

\def\op{\overline {\mathrm p}}

\def\Oe{\Omega_\varepsilon}
\renewcommand{\d}{\mathrm{d}}
\newcommand{\dx}{\,\d x}
\newcommand{\dt}{\,\d t}

\newcommand{\ddt}{\frac{\d}{\d t}}

\def\Andrea#1{{\color{violet}#1}}

\newcommand{\ptial}[1]{ \partial_{#1} }
\newcommand{\pt}{\ptial{t}}
\newcommand{\onehalf}{\frac{1}{2}}

\newcommand{\norm}[1]{\left\Vert #1 \right \Vert}
\newcommand{\normm}[1]{\Vert #1 \Vert}
\newcommand{\normmm}[1]{\left\vert #1 \right\vert}

\def\Div{\mathrm{div}\,}

\def\R{\mathbb R}
\def\eps{\varepsilon}
\def\mM{\mathcal M}
\allowdisplaybreaks[4]

\begin{document}
    \title[Thermal Transmission in a Film-Bulk Domain
]{Thermal Transmission in a Film-Bulk Domain: Finite-Thickness Film and Thin-Film Limit}	
	
	\author[F. Auricchio, M. Chiappetta, A. Poiatti, A. Reali]{
		\small 
		Ferdinando Auricchio$^{a,b}$, Mihaela Chiappetta$^{a,*}$,  \\Andrea Poiatti$^c$, Alessandro Reali$^{a,b}$
	}

	\address{
		$^a${Department of Civil Engineering and Architecture - University of Pavia, Via Ferrata 3, 27100, Pavia, Italy}
	}
	\address{
		$^b$Istituto di Matematica Applicata e Tecnologie Informatiche “E. Magenes” - Consiglio Nazionale delle Ricerche, Via Ferrata, 5/A, 27100, Pavia, Italy
	}
	\email{ferdinando.auricchio@unipv.it}
	\email{mihaela.chiappetta@unipv.it}

	\address{
		$^c$Dipartimento di Scienze Matematiche, Fisiche e Informatiche, \\ Università degli Studi di Parma, 43124 Parma, Italy
	}
	
	\email{andrea.poiatti@unipr.it}
	 \email{alessandro.reali@unipv.it}
	\thanks{$^*$Corresponding author: Mihaela Chiappetta.}
	
	\dedicatory{\ \\ The present work is dedicated to the memory of our dear friend and esteemed colleague Prof.~Alessandro Veneziani, recalling his passionate enthusiasm and pioneering intuition in combining mathematics, engineering, and physics,\\ as well as his constant dedication to collaboration and mentoring}
	
	\date{October 6, 2026}

	\maketitle
	
    \begin{abstract}
	We investigate a thermal transmission problem in a domain composed of a film region and a bulk region coupled through an interface where temperature continuity and heat-flux balance are imposed. The analysis addresses both the finite-thickness film configuration and the corresponding thin-film limit, with anisotropic, non-homogeneous thermal conductivity tensors.
	
	In the finite-thickness film configuration, the film occupies a volumetric domain and contributes directly to the heat transfer process through its coupling with the bulk region. Letting the film thickness vanish directly leads to a degenerate problem due to the vanishing measure of the film region.
	
	To preserve the contribution of the film in the limit, we introduce suitable scaling assumptions for the normal heat fluxes, namely for the flux balance across the film-bulk interface and for the flux prescribed on the external boundary of the film region, together with the requirement that the volumetric source in the film be uniformly bounded with respect to the film thickness. Under these scaling assumptions, we derive a non-trivial thin-film limit formulation that takes the form of a coupled bulk-interface problem, where the effect of the collapsing film is represented through an additional balance relation posed on the interface.
	
	Well-posedness results are established for both the finite-thickness film and thin-film limit configurations, and a rigorous convergence analysis is carried out to justify the asymptotic limit. In particular, we show that global-in-time solutions to the finite-thickness film problem converge to the thin-film limit solution as the film thickness vanishes. The main mathematical achievement consists in finding explicit convergence rates in suitable functional norms, relying only on the weak regularity of the finite-thickness problem due to the transmission conditions at the film interface.
	To complement the study, analytical solutions are derived for a simplified film-bulk geometry and are used to compare the thermal responses of the two configurations. These solutions also provide a quantitative assessment of the convergence rates established in the theoretical analysis.
\end{abstract}

	\section{Introduction}
	\label{sec:introduction}
	
	We investigate a thermal transmission problem posed on a domain composed of a film region and a bulk region coupled through an interface. 
	The analysis considers both a finite-thickness film configuration, in which the film occupies a volumetric layer of small but finite thickness, and the corresponding thin-film limit configuration, in which the film thickness tends to zero and the film collapses onto the interface separating the two regions. 
	\mich{In the finite-thickness film configuration, the thermal fields in the film and bulk regions are coupled through standard transmission conditions. As the film thickness decreases, the mathematical description progressively transitions from a volumetric formulation to an effective interface representation.}
	
	The reduction from the finite-thickness film configuration to the thin-film limit substantially modifies the mathematical structure of the transmission problem. 
	In the finite-thickness case, the thermal contribution of the film is described through standard diffusion and source terms defined over the film region. 
	Directly letting the film thickness tend to zero leads instead to a degenerate problem, since the measure of the film region vanishes together with its volumetric contribution. Consequently, the resulting formulation fails to retain the thermal effect associated with the film region and does not provide a meaningful description of the limiting physical process.
	
	Despite this mathematical degeneracy, numerous physical systems involve thin regions whose thermal contribution remains significant even when their thickness is small. 
	Laser-based additive manufacturing processes, such as powder bed fusion of metals, provide a representative example, since thermal energy is deposited within a narrow region associated with the laser-material interaction \cite{mohsan2023advancements,sahasrabudhe2018laser,sun2017powder,singh2021powder}. 
	Although the characteristic thickness of this region is small, the associated thermal contribution strongly influences the global thermal response and therefore requires an appropriate mathematical representation \cite{meier2017thermophysical,gouge2017thermo,mirkoohi2018thermal}. 
	Similar situations arise in a variety of diffusion and transport processes involving coatings, thin layers, membranes, or localized energy sources, where the influence of a geometrically small region remains non-negligible at the macroscopic scale. \mich{Thermal barrier coatings applied to turbine and combustor components provide a first example: a ceramic top coat only a few hundred micrometers thick governs the thermal state of the underlying metallic component, and its contribution cannot be neglected in the description of the global thermal response \cite{padture2002thermal}.} A second example is offered by thin films and interface layers in microelectronic devices, whose thickness is several orders of magnitude smaller than the characteristic size of the device, but whose thermal behavior controls heat dissipation and hence the operating temperature of the whole system \cite{cahill2014nanoscale,shen2025interface}.
	
				\mich{To preserve the contribution of the film in the thin-film limit configuration, we introduce suitable scaling assumptions for the normal heat fluxes, namely for the flux balance across the film-bulk interface and for the flux prescribed on the external boundary of the film region, together with the requirement that the volumetric source in the film be uniformly bounded with respect to the film thickness.}
	The proposed scaling assumptions compensate for the vanishing measure of the film region and yield a non-trivial thin-film limit problem. 
	Under these assumptions, the limiting configuration leads to a coupled bulk-interface formulation in which the film no longer appears as a volumetric region but contributes through an additional balance relation posed on the interface and coupled with the bulk equation. 
	The resulting formulation provides a mathematically consistent description of thermal transmission in both the finite-thickness film and thin-film limit configurations. \andre{Similar scaling assumptions have already appeared in the literature in the study of other asymptotic limits (cf., e.g., \cite{Andr,DiBenedetto}).}
	
	Thin-domain problems and dimensional reduction techniques have been extensively investigated in the literature. 
	Parabolic problems posed on thin domains have been studied in \cite{Miura1,Miura3,Miura2}, where asymptotic limits and lower-dimensional effective models are derived as the domain thickness tends to zero.
	\andre{Equations on manifolds and evolving interfaces have also attracted the attention of the mathematical community, and we refer, for instance, to the studies on surface Navier-Stokes equations \cite{Poiat-Stefanelli,PS2016,PSW2021,JOR18,ORZ2022,BDGP2023}, on Cahn-Hilliard equations on evolving surfaces \cite{DE,CEGP,CEGP2}, as well as on surface Navier-Stokes-Cahn-Hilliard systems \cite{ES,AGGP1,AGPP2}. In particular, in \cite{ES} a formal thin-film procedure is adopted to derive the surface equations. Moreover, the thin-film limit problem can be interpreted as a bulk problem with dynamic boundary conditions, which have been extensively studied in the literature, especially in connection with phase-field models (cf., e.g., \cite{GK,GMS,KL,KL2,GGP}). 
	On the other hand, the finite-thickness film problem belongs to the class of so-called transmission problems. Fundamental analytical aspects, including well-posedness, regularity, and the treatment of the discontinuities arising across the interface, where suitable transmission conditions must be imposed, have been widely investigated, for instance, in \cite{Caffarelli,Gilbarg,Grisvard}; see also the recent survey \cite{urbano}.}
	
	\andre{The present work builds upon these developments and focuses on a thermal transmission problem in which volumetric and interfacial contributions coexist and interact through the thin-film limit process. 
	We formulate and analyze the thermal transmission problem in both the finite-thickness film and thin-film limit configurations. Under meaningful scaling assumptions, we first show the existence of global-in-time solutions for the finite-thickness film and thin-film problems. Then, we crucially extend the techniques for thin-film limits proposed in the works by Miura \cite{Miura1,Miura3,Miura2} to the case of a transmission problem including bulk equations and anisotropic, non-homogeneous conductivity tensors. The main tool we develop is an essential inequality connecting the corresponding bulk and (suitably averaged) film variables at the film-bulk interface (cf. Lemma~\ref{tracelemma}), together with the extension of the highly technical arguments proposed by Miura, so as to also account for the bulk and the related transmission conditions, as well as for the anisotropy of the conductivity tensors, for which some structural assumptions are in order. In this way we are able to prove that the solutions to the finite-thickness film problem converge to the solution to the thin-film problem, only relying on the limited regularity allowed for transmission problems.}
	
	\andre{We additionally show, rigorously, a convergence rate in suitable norms, which appears to be new for this type of film-bulk problem.}
    \mich{	The rate is of first order in the film thickness, up to the discrepancy of the initial datum, whenever the tangential and normal parts of the film conductivity decouple at the interface (cf. Remark~\ref{rem:structural}); without this structural assumption a quantitative estimate still holds, with a rate of the same order plus an additional term measuring the coupling in suitable norms, which must vanish as $\eps\to0$ to allow for convergence (cf. Remark~\ref{rem:noncommuting}).}
	
	\mich{In order to focus on the asymptotic behavior associated with the film region, the analysis is carried out under the assumption that the external boundary of the bulk region carries no prescribed conditions, and that the two regions share the same volumetric heat capacity; the general case is left to future investigation. }   

	\mich{We also investigate a simplified film-bulk problem under suitable geometric and modeling assumptions, derive analytical solutions in both configurations, and compare the corresponding thermal responses. Finally, these results are compared with the theoretical findings, showing that the convergence of the finite-thickness film solutions toward the thin-film solution is consistent with, and in the present simplified setting faster than, the rate established in the asymptotic analysis.}
	
	The paper is organized as follows.
	Section~\ref{sec:problem} introduces the thermal transmission problem and presents its formulation in both the finite-thickness film and thin-film limit configurations. Particular attention is devoted to the identification of the degeneracy arising in the direct thin-film limit and to the introduction of the scaled formulation that preserves the contribution of the film as its thickness tends to zero.
	Section~\ref{sec:thinfilm_analysis} develops the mathematical framework required for the asymptotic analysis. After introducing the geometric setting of the film-bulk domain and the associated family of thin regions, we present the rescaled transmission problem together with the mathematical tools needed for the analysis. The section then establishes the main theoretical results, including well-posedness of the finite-thickness and limit problems, rigorous convergence of the scaled formulation, and quantitative error estimates characterizing the asymptotic behavior of the solutions.
	Section~\ref{sec:simplified} complements the theoretical developments by considering a simplified film-bulk geometry for which analytical solutions can be derived explicitly. The resulting solutions provide additional insight into the thermal behavior of the system and allow a direct comparison between the finite-thickness film configuration and the corresponding thin-film limit. We also compare these results with the theoretical findings of Section~\ref{sec:thinfilm_analysis}. Finally, Section~\ref{sec:conclusions} summarizes the main results of the work and discusses possible directions for future developments.
	
	\section{Thermal Transmission Problem in a Film-Bulk Domain}
	\label{sec:problem}
	We begin by introducing the geometric setting and governing equations of the thermal transmission problem considered throughout the paper.
	
	Let $\Omega \subset \mathbb{R}^N$ denote a domain composed of two subregions, namely a film region, $\Omega_f$, and a bulk region, $\Omega_b$, separated by a sufficiently smooth interface denoted by $\Sigma$. A schematic representation of the domain $\Omega$ and its decomposition into the film and bulk regions is shown in Figure~\ref{fig:original_domain}.
	The geometric decomposition of the domain is therefore described by
	\[
	\Omega = \Omega_f \cup \Sigma \cup \Omega_b,
	\qquad
	\Omega_f \cap \Omega_b = \emptyset,
	\qquad
	\Sigma = \partial \Omega_f \cap \partial \Omega_b \cap \Omega.
	\]
	Here $\Omega_f$ and $\Omega_b$ are open sets, and the interface $\Sigma$ is the portion of their common boundary lying inside $\Omega$, so that the edge of the interface, which lies on the external boundary $\partial\Omega$, is excluded from $\Sigma$ and the above decomposition is consistent with $\Omega$ being an open set.
	
	\mich{The boundaries of the two regions are decomposed as
	\[
	\partial \Omega_f = \Sigma \cup \Sigma_{\mathrm p},
	\qquad
	\partial \Omega_b = \Sigma \cup \Sigma_v \cup \Sigma_{\mathrm q},
	\]
		where $\Sigma_{\mathrm p} = \partial \Omega_f \setminus \Sigma$ denotes the external boundary of the film region, on which a prescribed heat flux may be applied, while $\Sigma_v$ and $\Sigma_{\mathrm q}$ denote the Dirichlet and Neumann portions of $\partial \Omega_b$, respectively, the three portions $\Sigma$, $\Sigma_v$ and $\Sigma_{\mathrm q}$ being pairwise disjoint.}
	
	The unit normal vectors associated with these boundary portions are the following: $\mathbf n_\Sigma$ denotes the unit normal on $\Sigma$, oriented from the bulk region $\Omega_b$ toward the film region $\Omega_f$; $\mathbf n_{\mathrm p}$ denotes the outward unit normal to $\Omega_f$ on $\Sigma_{\mathrm p}$; and $\mathbf n_{\mathrm q}$ denotes the outward unit normal to $\Omega_b$ on $\Sigma_{\mathrm q}$.
	
	The temperature fields in the film and bulk regions are denoted by $u=u(\mathbf x,t)$ and $v=v(\mathbf x,t)$, respectively, while $g=g(\mathbf x,t)$ and $h=h(\mathbf x,t)$ represent volumetric heat sources in $\Omega_f$ and $\Omega_b$.
	
	The thermal behavior in the two regions is characterized by volumetric heat capacities $a=a(\mathbf x,t)$ and $c=c(\mathbf x,t)$, and by thermal conductivity tensors $\mathbf B=\mathbf B(\mathbf x,t)$ and $\mathbf D=\mathbf D(\mathbf x,t)$, defined in $\Omega_f$ and $\Omega_b$, respectively. 
	
	\begin{figure}[htbp]
		\centering
		\begin{subfigure}[b]{0.52\textwidth}\centering
			\includegraphics[scale=0.80]{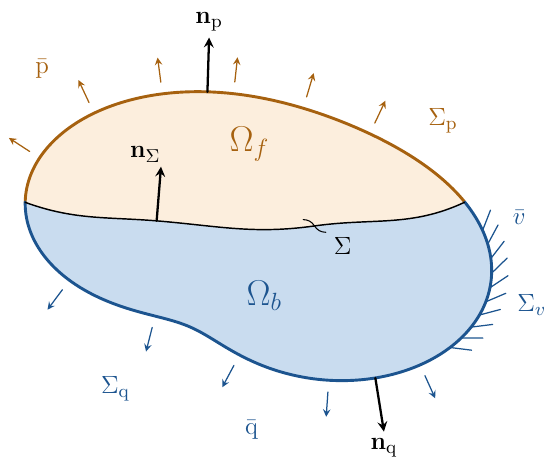}
			\caption{}
			\label{fig:domain_bc}
		\end{subfigure}\hfill
		\begin{subfigure}[b]{0.45\textwidth}\centering
			\includegraphics[scale=0.80]{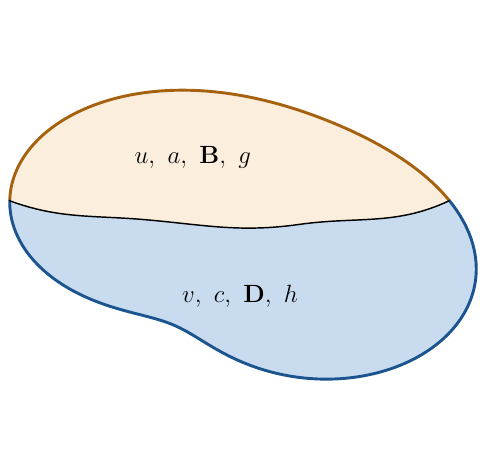}
			\caption{}
			\label{fig:domain_fields}
		\end{subfigure}
		\caption{\mich{Schematic representation of the film-bulk domain. The film region $\Omega_f$ and the bulk region $\Omega_b$ are separated by the interface $\Sigma$. Here $\Sigma_{\mathrm p}$ is the external boundary of the film region, on which the heat flux $\op$ is prescribed, whereas $\Sigma_v$ and $\Sigma_{\mathrm q}$ are the Dirichlet and Neumann portions of the bulk boundary, carrying the prescribed temperature $\bar v$ and the prescribed heat flux $\bar{\mathrm q}$; $\bn_\Sigma$, $\bn_{\mathrm p}$, $\bn_{\mathrm q}$ denote the corresponding outward unit normals. (A) Boundary portions, prescribed data and outward unit normals. (B) Fields and material properties of each region.}}
		\label{fig:original_domain}
	\end{figure}
	
	Heat conduction in the film and bulk regions is governed by parabolic balance equations \cite{PS2016}. In particular, the temperature fields $u$ and $v$ satisfy
	\[
	\begin{aligned}
		&a\,\partial_t u - \Div\!\big(\mathbf  B \nabla u\big) = g
		&& \text{in } \Omega_f,\\
		& c\,\partial_t v - \Div\!\big( \mathbf D \nabla v\big) = h
		&& \text{in } \Omega_b.
	\end{aligned}
	\]
	
	Thermal interaction between the two regions is described by transmission conditions on the interface $\Sigma$,
	\[
	\begin{aligned}
		u &= v 
		&& \text{on } \Sigma,\\
		\mathbf B \nabla u \cdot \mathbf n_\Sigma 
		&= \mathbf D \nabla v \cdot \mathbf n_\Sigma 
		&& \text{on } \Sigma,
	\end{aligned}
	\]
	which enforce temperature continuity and balance of normal heat fluxes.
	
	Finally, the external boundary conditions for the film and bulk regions are given by
	\[
	\mathbf B \nabla u \cdot \mathbf n_{\mathrm p} = \op
	\quad \text{on } \Sigma_{\mathrm p},
	\qquad
	v = \bar v
	\quad \text{on } \Sigma_v,
	\qquad
	\mathbf D \nabla v \cdot \mathbf n_{\mathrm q} = \bar{\mathrm q}
	\quad \text{on } \Sigma_{\mathrm q}.
	\]
	
	Collecting the parabolic balance equations, the transmission conditions, and the boundary conditions, the global thermal transmission problem in the finite-thickness film configuration reads
	\begin{align}
		\label{eps01}
		\begin{cases}
			a\,\partial_t u - \Div\!\big( \mathbf  B\nabla u\big) = g
			& \text{in } \Omega_f,\\[4pt]
			c\,\partial_t v - \Div\!\big( \mathbf D\nabla v\big) = h
			& \text{in } \Omega_b,\\[4pt]
			u = v
			& \text{on } \Sigma,\\[4pt]
			\mathbf  B\nabla u \cdot \mathbf n_\Sigma
			= \mathbf  D\nabla v \cdot \mathbf n_\Sigma
			& \text{on } \Sigma,\\[4pt]
			\mathbf B\nabla u \cdot \mathbf n_{\mathrm p} = \op
			& \text{on } \Sigma_{\mathrm p},\\[4pt]
			v = \bar v
			& \text{on } \Sigma_v,\\[4pt]
			\mathbf D\nabla v \cdot \mathbf n_{\mathrm q} = \bar{\mathrm q}
			& \text{on } \Sigma_{\mathrm q},
		\end{cases}
	\end{align}
	with initial conditions
	\[
	u(0) = u_0 \quad \text{in } \Omega_f,
	\qquad
	v(0) = v_0 \quad \text{in } \Omega_b,
	\qquad
	u_0 = v_0 \quad \text{on } \Sigma.
	\]
	
	We also consider the thin-film limit configuration, in which the film region collapses onto the interface $\Sigma$ \cite{Miura1,Miura3,Miura2}. The film is modeled as a family of thin domains $\{\Omega_\varepsilon\}_{\varepsilon>0}$ with characteristic thickness proportional to a small parameter $\varepsilon$, so that $|\Omega_\varepsilon| = \mathcal{O}(\varepsilon)$ (see Section~\ref{sec:thinfilm_analysis}). As $\varepsilon \to 0$, the volumetric contributions within $\Omega_\varepsilon$ vanish, and a direct passage to the limit in problem \eqref{eps01}, performed without additional scaling assumptions, leads to the formal thin-film limit configuration
	\begin{align}
		\label{thinfilm01}
		\begin{cases}
			c\,\partial_t v - \Div\!\big(\mathbf  D\nabla v\big) = h
			& \text{in } \Omega_b,\\[3pt]
			v = \bar v
			& \text{on } \Sigma_v,\\[3pt]
			\mathbf D\nabla v \cdot \mathbf n_{\mathrm q} = \bar{\mathrm q}
			& \text{on } \Sigma_{\mathrm q}.
		\end{cases}
	\end{align}
	
	\mich{Problem~\eqref{thinfilm01} reveals the degeneracy of the direct thin-film limit: all thermal contributions associated with the film region disappear as $\varepsilon \to 0$, so that the film reduces to a passive interface and the resulting formulation provides no information on the limiting physical process. Additional modeling assumptions are therefore required in order to obtain a meaningful thin-film limit that preserves the influence of the film on the overall thermal evolution.}
	
	To preserve the contribution of the film in the thin-film limit configuration, we introduce suitable asymptotic scaling assumptions, whose precise role is discussed in Section~\ref{sec:thinfilm_analysis}. \mich{Accordingly, problem \eqref{eps01} is replaced by}
	\begin{align}
		\label{eps1}
		\begin{cases}
			a\,\partial_t u - \Div\!\big( \mathbf B\nabla u\big) =  g
			& \text{in } \Omega_\varepsilon,\\[4pt]
			c\,\partial_t v - \Div\!\big( \mathbf D\nabla v\big) = h
			& \text{in } \Omega_b,\\[4pt]
			u = v
			& \text{on } \Sigma,\\[4pt]
			\mathbf  B\nabla u \cdot \mathbf n_\Sigma
			= \varepsilon\, \mathbf D\nabla v \cdot \mathbf n_\Sigma
			& \text{on } \Sigma,\\[4pt]
			\mathbf  B\nabla u \cdot \mathbf n_{\mathrm p} = \varepsilon\,\op
			& \text{on } \Sigma_{\mathrm p},\\[4pt]
			v = \bar v
			& \text{on } \Sigma_v,\\[4pt]
			\mathbf D\nabla v \cdot \mathbf n_{\mathrm q} = \bar{\mathrm q}
			& \text{on } \Sigma_{\mathrm q},
		\end{cases}
	\end{align}
	 with the same initial conditions as in problem \eqref{eps01}.
	
	As proven in Section~\ref{sec:thinfilm_analysis}, in the geometric setting introduced there, as $\varepsilon \to 0$ the thermal transmission problem \eqref{eps1} converges to the non-trivial thin-film limit configuration
	\begin{align}
		\label{thinfilm11}
		\begin{cases}
			c\,\partial_t v - \Div\!\big( \mathbf D\nabla v\big) = h
			& \text{in } \Omega_b,\\[3pt]
			v = \bar v
			& \text{on } \Sigma_v,\\[3pt]
			\mathbf D\nabla v \cdot \mathbf n_{\mathrm q} = \bar{\mathrm q}
			& \text{on } \Sigma_{\mathrm q},\\[3pt]
			u = v
			& \text{on } \Sigma,\\[3pt]
			a\,\partial_t u
			- \Div_\Sigma\!\big(\bP\mathbf  B_{|\Sigma}\bP\nabla_\Sigma u\big)
			= -\,\mathbf  D\nabla v \cdot \mathbf n_\Sigma
			+ g + \op
			& \text{on } \Sigma,
		\end{cases}
	\end{align}
	with
	\[
	u_0 = v_0 \quad \text{on } \Sigma,
	\qquad
	u(0) = u_0 \quad \text{on } \Sigma,
	\qquad
	v(0) = v_0 \quad \text{in } \Omega_b.
	\]

    	\mich{Here $\bP$ denotes the orthogonal projection onto the tangent space to $\Sigma$, $\nabla_\Sigma$ and $\Div_\Sigma$ the associated tangential gradient and divergence, and $\mathbf B_{|\Sigma}$ the restriction of $\mathbf B$ to the interface; all of them are introduced in Section~\ref{mat_sec1}.}
	Problem \eqref{thinfilm11} defines a coupled bulk-interface system in which the collapsing film region is represented through an effective interfacial balance equation. \mich{Only the tangential component of the conductivity tensor $\mathbf B$ enters the interfacial equation, as discussed in Section~\ref{sec:thinfilm_analysis}.}
	
	\mich{It is worth observing that this formulation admits two equivalent interpretations. On the one hand, problem \eqref{thinfilm11} can be regarded as a pair of coupled parabolic problems of different dimensions, the first posed in the $N$-dimensional bulk region $\Omega_b$ and the second posed on the $(N-1)$-dimensional interface $\Sigma$. The two problems interact through the continuity of the temperature across $\Sigma$ and through the normal heat flux $\mathbf D \nabla v \cdot \mathbf n_\Sigma$, which acts as a source term in the interfacial balance equation. On the other hand, since the interfacial temperature coincides with the trace of the bulk temperature on $\Sigma$, the unknown $u$ can be eliminated and problem \eqref{thinfilm11} can be equivalently rewritten as a single parabolic problem posed in $\Omega_b$, supplemented with the non-standard boundary condition
	\[
	a\,\partial_t v
	- \Div_\Sigma\!\big(\bP\mathbf B_{|\Sigma}\bP \nabla_\Sigma v\big)
	+ \mathbf D \nabla v \cdot \mathbf n_\Sigma
	= g + \op
	\qquad \text{on } \Sigma,
	\]
	where $v$ on $\Sigma$ is understood as the trace of the bulk temperature. This condition involves both the time derivative and the tangential diffusion of the trace of $v$, and therefore belongs to the class of dynamic boundary conditions (cf. the Introduction for some references).}
	\section{Mathematical Analysis of the Film-Bulk Thermal Transmission Problem}
	\label{sec:thinfilm_analysis}
	We develop the mathematical analysis of the thermal transmission problem introduced in Section~\ref{sec:problem}. The focus lies on the asymptotic behavior of the finite-thickness film configuration as the film thickness tends to zero and on the derivation of the corresponding thin-film limit configuration.
	
		\mich{We first introduce, in Section~\ref{mat_sec1}, the geometric setting of the film-bulk domain and the family of thin domains used to describe the collapsing film. Compared to Figure~\ref{fig:original_domain}, the analysis is carried out in a shell configuration, in which $\Omega_b$ is the internal region and is surrounded by the thin-film region $\Omega_\varepsilon$ (cf. Figure~\ref{bilayer}).}
	\mich{Section~\ref{mat_sec2} collects the assumptions on the coefficients of the problem and introduces the averaging operator along the film direction, on which the whole asymptotic analysis rests.}
	\mich{Section~\ref{mat_sec3} develops the preliminary results and the mathematical tools required for the analysis, in particular the estimates relating volumetric quantities defined on $\Omega_\varepsilon$ to their surface counterparts on $\Sigma$.}
	
		\mich{Section~\ref{mat_sec4} then establishes the main theoretical results: it first states the rescaled film-bulk problem and its thin-film limit explicitly, and then addresses the well-posedness of the film-bulk thermal transmission problem in both configurations, the rigorous derivation of the thin-film limit, and the convergence analysis of the rescaled problem. The latter includes explicit rates, in suitable norms, at which the solutions corresponding to $\varepsilon>0$ approach the thin-film solution.}

	\subsection{Geometric Setting of the Film-Bulk Domain}
	\label{mat_sec1}
	The thermal transmission problem introduced in Section~\ref{sec:problem} is now analyzed within a geometric setting tailored to the asymptotic analysis from the finite-thickness film configuration to the corresponding thin-film limit.
	The general \mich{film-bulk} domain previously introduced is therefore specialized to a family of uniformly thin tubular neighborhoods surrounding a closed interface without boundary.
	This geometric setting provides global normal coordinates and a natural description of the collapsing film region as the film thickness tends to zero.
	
	\mich{Let $\Sigma\subset\mathbb R^N$, with $N>1$, be a compact surface of class $C^3$ without boundary, and let $\bn_\Sigma$ denote its outward unit normal vector field. Consistently with Section~\ref{sec:problem}, $\bn_\Sigma$ is therefore oriented from the bulk region toward the film region.
	Within the asymptotic setting considered in the present section, the film region is represented by a family of curved thin domains $\Omega_\varepsilon\subset\mathbb R^N$, parameterized by the film thickness $\varepsilon>0$ and defined by}
	\begin{align}
		\Omega_\varepsilon = \{ \mathbf x\in\mathbb R^N:\ \mathbf x=y+r\bn_\Sigma(y),\quad y\in\Sigma,\quad r\in(0,\varepsilon) \}.
		\label{thindomain}
	\end{align}
	\mich{Note that, as in \cite{Miura3,Miura2}, the analysis can be generalized to the case of a nonuniform film thickness, provided that the thickness, as a function of $y\in\Sigma$, is bounded from below by $C_1\varepsilon$ and from above by $C_2\varepsilon$ for some constants $0<C_1<C_2$ independent of $\varepsilon$.
	The bulk region $\Omega_b\subset\mathbb R^N$ is the open set enclosed by $\Sigma$, so that $\partial\Omega_b=\Sigma$; in particular, in this configuration the bulk region has no external boundary.
	The complete domain, which depends on $\varepsilon$ and is denoted by $\Omega$ for brevity, is therefore given by
	\[
	\Omega = \Omega_b\cup\Sigma\cup\Omega_\varepsilon .
	\]}%
	\mich{The external boundary of the film region, denoted by $\Sigma_{\mathrm p}$ in Section~\ref{sec:problem}, is here the parallel surface at distance $\varepsilon$ from $\Sigma$, namely
	\begin{align}
		\Sigma_{\mathrm p} = \{ \mathbf x\in\mathbb R^N:\ \mathbf x = y+\varepsilon\bn_\Sigma(y),\quad y\in\Sigma \},
		\label{sigmap}
	\end{align}
	and $\bn_{\mathrm p}=\bn_\Sigma\circ\pi$ is its outward unit normal.
	A schematic representation of the film-bulk domain is shown in Figure~\ref{bilayer}.}
	
	\begin{figure}[htbp]
		\centering
		\includegraphics[width=0.5\textwidth]{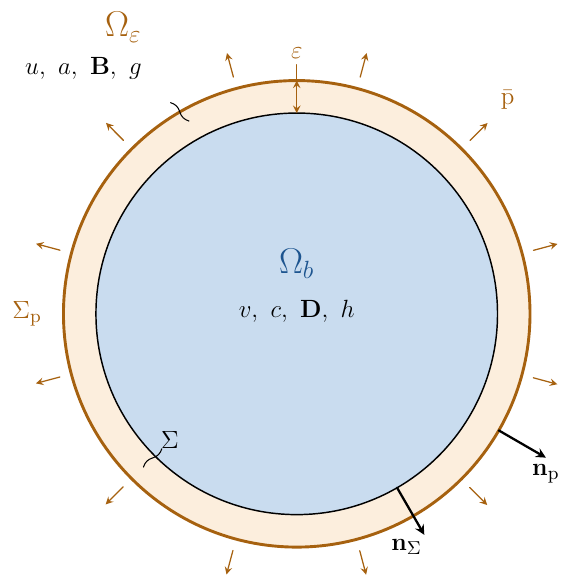}
		\caption{\mich{Schematic representation of the film-bulk domain specialized to a family of uniformly thin tubular neighborhoods surrounding a closed interface without boundary. The bulk region $\Omega_b$ is enclosed by the interface $\Sigma$ and is surrounded by the film region $\Omega_\varepsilon$ of thickness $\varepsilon$, whose external boundary is denoted by $\Sigma_{\mathrm p}$. Each region carries its temperature field, volumetric heat capacity, thermal conductivity tensor and volumetric heat source, and $\bn_\Sigma$, $\bn_{\mathrm p}$ denote the corresponding outward unit normals.}}
		\label{bilayer}
	\end{figure}
	
	\mich{Let now $d=d(\mathbf x)$ denote the signed distance function from $\Sigma$, increasing along the direction of $\bn_\Sigma$, and let $k_1,\dots,k_{N-1}$ denote the principal curvatures of $\Sigma$, taken with the sign convention fixed in \eqref{weingarten} below.
	Since $\Sigma$ is compact and of class $C^3$, the normal vector field $\bn_\Sigma$ is of class $C^2$ and the principal curvatures belong to $C_b^1(\Sigma)$; in particular, they are bounded.}
	
	Consequently, there exists a tubular neighborhood
	\(
	\mathcal N_\delta = \{ \mathbf x\in\mathbb R^N:\ |d(\mathbf x)|<\delta \}
	\)
	of $\Sigma$, of radius $\delta>0$, such that every point $\mathbf x\in\mathcal N_\delta$ admits a unique closest-point projection $\pi(\mathbf x)\in\Sigma$, satisfying
	\[
	\mathbf x = \pi(\mathbf x)+d(\mathbf x)\bn_\Sigma(\pi(\mathbf x)), \qquad \nabla d(\mathbf x) = \bn_\Sigma(\pi(\mathbf x)).
	\]
	
	The signed-distance function $d$ and the projection map $\pi$ belong to
	$C^2(\mathcal N_\delta)$ and $C^1(\mathcal N_\delta)$, respectively (see, e.g., \cite[Section 14.6]{Gilbarg}).
	\mich{Moreover, choosing $\delta>0$ sufficiently small, there exists a constant $C_\Sigma\geq1$ such that
	\begin{align}
		\frac1{C_\Sigma} \leq 1-rk_i(y) \leq C_\Sigma, \qquad \forall y\in\Sigma,\quad r\in[-\delta,\delta],\quad i=1,\dots,N-1 .
		\label{controlcurv}
	\end{align}}
        
	The orthogonal projection onto the tangent plane of $\Sigma$ is denoted by
	\(
	\bP = \bI-\bn_\Sigma\otimes\bn_\Sigma .
	\)
	
	For $\eta\in C_b^1(\Sigma)$, the tangential gradient is defined as
	\(
	\nabla_\Sigma\eta = \bP(\nabla\widetilde\eta)_{|\Sigma},
	\)
	where $\widetilde\eta$ denotes an extension of $\eta$ to $\mathcal N_\delta$.
	In particular, the constant extension of $\eta$ along the normal direction is given by
	\(
	\widetilde\eta = \eta\circ\pi,
	\)
	which satisfies
	\[
	\nabla\widetilde\eta(\mathbf x) = \nabla_\Sigma\eta(\mathbf x), \qquad \mathbf x\in\Sigma .
	\]

    	\mich{Correspondingly, for a vector field $\mathbf F\in C^1_b(\Sigma;\R^N)$ the tangential divergence is defined as
	\(
	\Div_\Sigma\mathbf F = \bP:\nabla\widetilde{\mathbf F}_{|\Sigma} ,
	\)
	where $\widetilde{\mathbf F}$ denotes an extension of $\mathbf F$ to $\mathcal N_\delta$. Both definitions are independent of the chosen extension, since only tangential derivatives are involved.}
	
	\mich{The Weingarten map of $\Sigma$ is then given by
	\begin{align}
		W := -\nabla_\Sigma\bn_\Sigma ,
		\label{weingarten}
	\end{align}
	its eigenvalues being the principal curvatures $k_1,\dots,k_{N-1}$ introduced above. This fixes the sign convention adopted throughout, and is the one under which the Jacobian of the change of variables in the film region takes the form \eqref{J}. Accordingly, we set
	\(
	H_\Sigma := -\Div_\Sigma\bn_\Sigma = \operatorname{tr}W =k_1+\dots+k_{N-1}
	\)
	for the mean curvature, so that
	\(
	\Div_\Sigma(\varphi\,\bn_\Sigma) = -\varphi\,H_\Sigma
	\)
	for every scalar function $\varphi$ on $\Sigma$: a vector field with a non-vanishing normal component contributes as usual to its own tangential divergence through a curvature term.}

	Finally, for a function
	\(
	\varphi:\mathcal N_\delta\to\mathbb R^m,
	\)
	with $m\geq1$, the corresponding representation in film coordinates is defined by
	\begin{align}
		\varphi^\sharp(y,r) = \varphi(y+r\bn_\Sigma(y)), \qquad y\in\Sigma,\quad r\in[-\delta,\delta].
		\label{sharp}
	\end{align}
	Throughout the remainder of the section we assume
	\(
	\delta\in(0,1),  \varepsilon\in(0,\delta),
	\)
	so that $\Omega_\varepsilon\subset\mathcal N_\delta$ and \eqref{sharp} is well defined for every $r\in[0,\varepsilon]$.

	\subsection{\mich{Assumptions on the Coefficients and the Averaging Operator}}
	\label{mat_sec2}
		
		\mich{We now introduce the mathematical assumptions underlying the rescaled thermal transmission problem associated with the film-bulk domain described in Section~\ref{mat_sec1}. 
		In the shell configuration adopted there, the bulk region is enclosed by $\Sigma$, so that $\partial\Omega_b=\Sigma$ and the portions $\Sigma_v$ and $\Sigma_{\mathrm q}$ of problem~\eqref{eps1} are empty (cf. Figure~\ref{bilayer}). 
		This allows us, in this first study, to focus exclusively on the asymptotic behavior associated with the film region, and avoids the technical complications arising from the external boundary conditions on the bulk region.}
		
		The coefficients $a,c,\mathbf B,\mathbf D$ are assumed to satisfy
		\(
		a,c \in C_b^1(\mathbb R^N\times[0,\infty)),
		\mathbf B,\mathbf D\in C_b^1(\mathbb R^N\times[0,\infty);\mathbb R_{sym}^{N\times N}),
		\)
		together with the conditions
		\mich{\[
		0<c_0\leq a(\mathbf x,t),\ c(\mathbf x,t),
		\qquad
		c_0|\xi|^2 \leq \xi\cdot\mathbf B(\mathbf x,t)\,\xi,
		\qquad
		c_0|\xi|^2 \leq \xi\cdot\mathbf D(\mathbf x,t)\,\xi,
		\]
				for some constant $c_0>0$, for every $\xi\in\mathbb R^N$ and every $(\mathbf x,t)\in\mathbb R^N\times[0,\infty)$. \mich{Starting from Section~\ref{mat_sec4} we shall further assume $c=a$, as discussed there.}}
		
		The mathematical analysis relies on the tangential differential operators $\nabla_\Sigma$ and $\Div_\Sigma$ introduced in Section~\ref{mat_sec1}. In addition, a fundamental role is played by an averaging operator acting along the film direction. 
		\mich{In particular, given a function $\varphi:\Omega_\varepsilon\to\mathbb R^m$, with $m\geq1$, the corresponding average is defined componentwise by}
		\begin{align}
			\mM_\varepsilon(\varphi)(y) = \frac1\varepsilon \int_0^\varepsilon \mich{\varphi^\sharp(y,r)}\,J(y,r)\,\d r, \qquad y\in\Sigma,
			\label{average}
		\end{align}
		\mich{where $J$ is the Jacobian of the change of variables in the film region $\Omega_\varepsilon$, defined in \eqref{J} below.}
		
		Averaging operators of the form \eqref{average} have been extensively employed in the analysis of thin-domain problems and thin-film limits; see, for example, \cite{Miura1,Miura3,Miura2}. The same approach is adopted here to investigate the thermal transmission problem in the film-bulk domain, \mich{and the properties of $\mM_\varepsilon$ needed in the sequel are collected in Section~\ref{mat_sec3}.}
		
		\mich{The operator $\mM_\varepsilon$ also makes the degeneracy of the unscaled thin-film limit explicit. Indeed, the change of variables in the thin domain yields}
		\[
		\int_{\Omega_\varepsilon}f\dx = \varepsilon \int_\Sigma \mM_\varepsilon(f)\, \d\mathcal H^{N-1},
		\qquad f\in L^1(\Omega_\varepsilon),
		\]
		\mich{where $\mathcal H^{N-1}$ denotes the $(N-1)$-dimensional Hausdorff measure on $\Sigma$. Hence, for data bounded uniformly with respect to $\varepsilon$, every volumetric contribution from $\Omega_\varepsilon$ is of order $\varepsilon$ and disappears as $\varepsilon\to0$, so that the film no longer participates in the overall thermal balance. This is precisely the degeneracy displayed by problem~\eqref{thinfilm01}, and the reason why the scaling assumptions leading to problem~\eqref{eps1} are introduced. The rescaled system analyzed in the sequel is written explicitly in \eqref{eps1b}, and the corresponding thin-film limit problem in \eqref{thinfilm11b}.}		

        \subsection{Preliminary Results and Mathematical Tools}
		\label{mat_sec3}
		
		We introduce the preliminary results and the mathematical tools required for the asymptotic analysis of the rescaled thermal transmission problem. The relevant function spaces are first recalled, together with several geometric properties of the film region $\Omega_\eps$.
		
		Subsequently, the averaging operator along the film direction is investigated and its main properties are established. 
		Particular attention is devoted to the estimates relating volumetric quantities defined on $\Omega_\eps$ to their surface counterparts on $\Sigma$, since such results constitute a fundamental ingredient in the derivation of the thermal transmission problem in the thin-film limit configuration (cf. in particular Lemma~\ref{tracelemma}).
		
		The section concludes with additional estimates involving the spatial and temporal derivatives of the averaging operator, which will be employed repeatedly in the well-posedness, convergence and error analyses developed in Section~\ref{mat_sec4}.
		\subsubsection{Main Functional Spaces}
		\label{add1}
		For any normed space $X$, its norm is denoted by $\norm{\cdot}_X$, while its dual space is denoted by $X'$.
		
		Let $\Omega_A$ denote generically either $\Omega$, $\Omega_\eps$ or $\Omega_b$.
		For $1\le p\le\infty$ and $k\in\mathbb N$, we denote by
		$L^p(\Omega_A)$ and $W^{k,p}(\Omega_A)$
		the classical Lebesgue and Sobolev spaces defined on $\Omega_A$,\mich{ and we set $H^k(\Omega_A)=W^{k,2}(\Omega_A)$.
		Since $\Sigma$ is of class $C^3$, the spaces $L^p(\Sigma)$, $p\in[1,\infty]$, can also be introduced, together with the spaces $H^k(\Sigma)$, obtained by replacing classical derivatives with tangential derivatives in their definition. The trace operator on $\Sigma$ is denoted by $\mathrm{Tr}_\Sigma$; whenever a function is defined on $\Omega_\eps$ only, the trace is understood to be taken from within $\Omega_\eps$.}
		
		Furthermore, for any interval $I\subset\R$ and any Banach space $X$, $C^0(I;X)$ (respectively, $C^0_b(I;X)$) denotes the space of continuous functions (respectively, bounded continuous functions) mapping from $I$ to $X$. Analogously, the spaces $C^k(I;X)$ and $C^k_b(I;X)$ are defined for arbitrary integers $k\geq1$.
		
		The notation $C^{0,\gamma}(I;X)$, with $\gamma\in(0,1]$, denotes the space of $\gamma$-H\"{o}lder continuous functions with values in $X$. In the particular case $\gamma=1$, the corresponding space coincides with the space of Lipschitz continuous functions with values in $X$.
		
		\subsubsection{Properties of the Film Region}
		\label{add2}
		For $y\in \Sigma$ and $r\in[-\delta,\delta]$ we set
		\begin{align}
			J(y,r)=\prod_{i=1}^{N-1}\bigl(1-rk_i(y)\bigr)=\mich{\det\bigl(\bI-rW(y)\bigr).}
			\label{J}
		\end{align}
		The function $J$ represents the Jacobian associated with the change-of-variables formula
		\begin{align}
			\int_{\Omega_\eps}\varphi(\mathbf x)\dx =\int_\Sigma \int_0^{\eps}\varphi^\sharp(y,r)J(y,r)\d r \d \mathcal H^{N-1}(y)
			\label{change}
		\end{align}
		for any function $\varphi\in L^1(\Omega_\eps)$, \mich{where $\mathcal H^{N-1}$ denotes the $(N-1)$-dimensional Hausdorff measure on $\Sigma$. 
		The same Jacobian governs the corresponding change of variables on the parallel surfaces: since the area element of the surface at distance $r$ from $\Sigma$ is $J(y,r)\d\mathcal H^{N-1}(y)$, recalling \eqref{sigmap} we have
		\begin{align}
			\int_{\Sigma_{\mathrm p}}\varphi\,\d \mathcal H^{N-1} =\int_\Sigma \varphi^\sharp(y,\eps)J(y,\eps)\d \mathcal H^{N-1}(y)
			\label{changesurf}
		\end{align}
		for any function $\varphi\in L^1(\Sigma_{\mathrm p})$. 
		Since $J$ is the product of the $N-1$ factors bounded in \eqref{controlcurv}, and since the principal curvatures are bounded, there exists a constant $C_J\geq1$ such that}
		\begin{align}
			\frac1 {C_J}\leq J(y,r)\leq C_J,\quad \normmm{\partial_ rJ(y,r)}\leq C_J,\quad \forall y\in \Sigma,\ r\in[-\delta,\delta].
			\label{controlJ}
		\end{align}
		Furthermore, owing to the regularity of the principal curvatures,
		\begin{align}
			\normmm{{J(y,r)}-1}\leq C\eps,\quad \normmm{\nabla_\Sigma J(y,r)}\leq C\eps,\quad \forall y\in \Sigma,\ r\in[0,\eps],
			\label{controlJ2}
		\end{align}
		holds for some constant $C>0$, where $\nabla_\Sigma J$ denotes the tangential gradient of $J$ with respect to the variable $y$. \mich{Using \eqref{change} together with the estimates above, the following equivalence is immediately obtained: for any $p\in[1,\infty)$,
		\begin{align}
			\frac 1C\norm{\varphi}_{L^p(\Omega_\eps)}^p\leq \int_\Sigma \int_0^\eps \normmm{\varphi^\sharp(y,r)}^p\d r\d \mathcal H^{N-1}(y)\leq C\norm{\varphi}_{L^p(\Omega_\eps)}^p,\quad\forall \varphi\in L^p(\Omega_\eps),
			\label{Lp}
		\end{align}
		which entails, for the constant extension along the normal direction,
		\begin{align}
			\frac {\eps^\frac1p}C\norm{\eta}_{L^p(\Sigma)}\leq \norm{\widetilde\eta}_{L^p(\Omega_\eps)}\leq C\eps^\frac1p\norm{\eta}_{L^p(\Sigma)},\quad \forall \eta\in L^p(\Sigma),\ p\in[1,\infty).
			\label{constext}
		\end{align}
		Moreover, it holds
		\begin{align}
			\norm{\nabla \widetilde \eta}_{L^2(\Omega_\eps)}\leq C\eps^\frac12\norm{\nabla_\Sigma\eta}_{L^2(\Sigma)},\quad\forall \eta\in H^1(\Sigma),
			\label{constextgrad}
		\end{align}
		and, recalling that $d$ is the signed distance function from $\Sigma$, that $\normmm{d}\leq C\eps$ in $\Omega_\eps$ and that $\nabla\widetilde\eta=(\bI-dW)^{-1}\widetilde{\nabla_\Sigma\eta}$ in $\mathcal N_\delta$, so that $\nabla\widetilde\eta-\widetilde{\nabla_\Sigma\eta}=\bigl((\bI-dW)^{-1}-\bI\bigr)\widetilde{\nabla_\Sigma\eta}$ is of order $\eps$ pointwise, \eqref{constext} with $p=2$ yields (cf. also \cite[Lemma 2.1]{Miura2})
		\begin{align}
			\norm{\nabla\widetilde\eta-\widetilde{\nabla_\Sigma\eta}}_{L^2(\Oe)}\leq C\eps^\frac32\norm{\nabla_\Sigma\eta}_{L^2(\Sigma)},\quad\forall \eta\in H^1(\Sigma).
			\label{controleta}
		\end{align}}
		
		We also introduce the notation
		\[
		\nabla_{\bn_{\Sigma}} \varphi:=\nabla \varphi\cdot \widetilde{\bn_{\Sigma}}
		\qquad\text{in }\Omega_\eps,
		\qquad
		\mich{\partial_{\bn_{\Sigma}}\varphi:=\mathrm{Tr}_\Sigma\bigl(\nabla_{\bn_{\Sigma}}\varphi\bigr)
		\qquad\text{on }\Sigma,}
		\]
		where $\bn_{\Sigma}$ denotes the outward unit normal vector field to $\Sigma$. 
		The previous definitions imply the chain rules
		\begin{align}
			\partial_r \varphi^\sharp(y,r)=(\nabla_{\bn_{\Sigma}}\varphi)^\sharp(y,r),\quad (\widetilde\eta)^\sharp(y,r)=\eta(y),
			\label{chainrules}
		\end{align}
		for sufficiently smooth functions $\varphi$ defined on $\Omega_\eps$ and for any function $\eta$ defined on $\Sigma$.
		
		\subsubsection{\mich{Properties of the Averaging Operator}}
		\label{add3}
		\mich{Throughout this subsection $\varepsilon\in(0,\delta)$, and $\mM_\eps$ denotes the averaging operator \eqref{average} introduced in Section~\ref{mat_sec2}. The motivation for that definition is provided by the following basic formula, an immediate consequence of \eqref{change} and of the second identity in \eqref{chainrules},
		\begin{align}
			\int_{\Omega_\eps} \varphi(\mathbf x)\widetilde\eta(\mathbf x)\dx=\eps\int_\Sigma \mM_\eps (\varphi)(y)\,\eta(y)\d \mathcal H^{N-1}(y),\label{basic}
		\end{align}
		which is exploited throughout the subsequent analysis.} Following and extending the presentation in \cite{Miura2}, some properties of the averaging operator are now presented.
		
		First, some basic properties are summarized. 
		For the proofs, reference is made to \cite[Lemma 5.1, Lemma 5.2]{Miura2} and \cite[Lemma 4.2]{Miura3}.
		
		\begin{lemma}
			\label{basiclemma}
			\mich{Let $p\in[1,\infty]$. Then, for any $\varphi\in L^p(\Omega_\eps)$, it holds $\mM_\eps(\varphi)\in L^p(\Sigma)$ and}
			\begin{align}
				\norm{\mM_\eps(\varphi)}_{L^p(\Sigma)}\leq C\eps^{-\frac1p}\norm{\varphi}_{L^p(\Omega_\eps)},
				\label{ctrl1}
			\end{align}
			for some $C>0$. \mich{If moreover $\varphi\in W^{1,p}(\Omega_\eps)$, then}
			\begin{align}
				\norm{\varphi-\widetilde{\mM_\eps(\varphi)}}_{L^p(\Omega_\eps)}\leq C\eps\norm{\varphi}_{W^{1,p}(\Omega_\eps)},
				\label{ctrl2}
			\end{align}
			for some $C>0$.
		\end{lemma}
		A fundamental lemma, allowing one to compare the trace of a function over $\Sigma$ with its average on $\Omega_\eps$ along the film direction, and thus allowing the bulk and the film region to communicate, is the following. Its proof shares some similarities with \cite[Proof of Lemma 5.2]{Miura2}, but the application to the present context appears to be new. 
		
		\begin{lemma}
        \label{tracelemma}
			\mich{			Let $\varphi\in H^1(\Omega_\eps)$. Then, for any $\eps\in(0,\delta)$, it holds}
			\begin{align}
				\norm{\mathrm{Tr}_\Sigma\varphi-\mM_\eps(\varphi)}_{L^2(\Sigma)}\leq C\eps^\frac12\norm{\varphi}_{H^1(\Omega_\eps)},
				\label{trace}
			\end{align}
			where $\mathrm{Tr}_\Sigma$ is the trace operator over $\Sigma$, for some $C>0$ depending only on $\Sigma$.
		\end{lemma}
		\begin{proof}
			First, we have, for almost any $y\in \Sigma$,
			\begin{align}
				\mathrm{Tr}_\Sigma\varphi(y)-\mM_\eps(\varphi)(y)=\mathrm{Tr}_\Sigma\varphi(y)-\frac1\varepsilon\int_0^\varepsilon \varphi^\sharp(y,r)J(y,r)\d r=K_1+K_2,
				\label{first}
			\end{align}
			where 
			\begin{align*}
				K_1=\frac1\eps\int_0^\eps \bigl(\mathrm{Tr}_\Sigma\varphi(y)-\varphi^\sharp(y,r)\bigr)\d r,\quad K_2=\frac1\eps \int_0^\eps \varphi^\sharp(y,r)\bigl(1-J(y,r)\bigr)\d r.
			\end{align*}
			Now, concerning $K_1$, we have, by \eqref{chainrules},
			\begin{align*}
				\normmm{K_1}\leq \frac1\eps \int_0^\eps \int_0^r \normmm{\partial_r\varphi^\sharp(y,s)}\d s\d r\leq \int_0^\eps \normmm{\partial_r\varphi^\sharp(y,s)}\d s=\int_0^\eps \normmm{(\nabla_{\bn_{\Sigma}}\varphi)^\sharp (y,s)}\d s,
			\end{align*}
			whereas, using \eqref{controlJ2}, it holds
			\begin{align*}
				\normmm{K_2}\leq \frac1\eps\int_0^\eps \normmm{\varphi^\sharp(y,r)}\normmm{1-J(y,r)}\d r\leq C \int_0^\eps \normmm{\varphi^\sharp(y,r)}\d r.
			\end{align*}
			As a consequence, from \eqref{first}, we get, by the Cauchy-Schwarz inequality,
			\begin{align*}
				& \normmm{\mathrm{Tr}_\Sigma\varphi(y)-\mM_\eps(\varphi)(y)}\\&\leq C\left(\int_0^\eps \Bigl(\normmm{(\nabla_{\bn_{\Sigma}}\varphi)^\sharp (y,r)}+  \normmm{\varphi^\sharp(y,r)}\Bigr) \d r\right)\\&\leq 
				C\eps^\frac12\left(\int_0^\eps \Bigl(\normmm{(\nabla_{{\bn_{\Sigma}}}\varphi)^\sharp (y,r)}^2+  \normmm{\varphi^\sharp(y,r)}^2\Bigr) \d r\right)^\frac12,
			\end{align*}
			so that, by taking the square and integrating over $y\in \Sigma$, we infer 
			\begin{align*}
				& \int_\Sigma\normmm{\mathrm{Tr}_\Sigma\varphi(y)-\mM_\eps(\varphi)(y)}^2\d\mathcal H^{N-1}(y)\\&\leq 
				C\eps\int_\Sigma\left(\int_0^\eps \Bigl(\normmm{(\nabla_{{\bn_{\Sigma}}}\varphi)^\sharp (y,r)}^2+  \normmm{\varphi^\sharp(y,r)}^2\Bigr) \d r\right)\d\mathcal H^{N-1}(y),
			\end{align*}
			from which, recalling \eqref{Lp}, we conclude that \eqref{trace} holds.
		\end{proof}
		Additionally, we report here some useful controls on the gradients of the averaged function, whose proof can be found, e.g., in \cite[Lemma 5.7]{Miura2} and \cite[Eq. (4.8)]{Miura3}.
		
		\begin{lemma}
			\label{gradlemma}
			Let $\varphi\in W^{1,p}(\Omega_\eps)$, $p\in[1,\infty]$. Then $\mM_\eps(\varphi)\in W^{1,p}(\Sigma)$ and
			\begin{align}
				\label{P1}  &\norm{\nabla_\Sigma \mM_\eps(\varphi)}_{L^p(\Sigma)}\leq C\eps^{-\frac1p}\norm{\varphi}_{W^{1,p}(\Omega_\eps)},\\&
				\norm{\nabla_\Sigma\mM_\eps(\varphi)-\mM_\eps(\widetilde \bP\nabla\varphi)}_{L^p(\Sigma)}\leq C\eps^{1-\frac1p}\norm{\varphi}_{W^{1,p}(\Omega_\eps)},
				\label{P2}
			\end{align}
			for some $C>0$.
		\end{lemma}
		\begin{proof}
			\mich{The proof in \cite[Lemma 5.7]{Miura2} is actually given for the case $p=2$, but one can see from it that}
			\begin{align*}
				\normmm{\nabla_\Sigma \mM_\eps(\vphi)}\leq C\bigl(\mM_\eps(\normmm{\varphi})+\mM_\eps(\normmm{\nabla\varphi})\bigr),
			\end{align*}
			and thus, by \eqref{ctrl1}, we immediately infer \eqref{P1} for any $p\in[1,\infty]$.
			Similarly, we point out that, from \cite[Eq. (4.8)]{Miura3}, we have
			\begin{align*}
				\normmm{\nabla_\Sigma\mM_\eps(\varphi)-\mM_\eps(\widetilde \bP\nabla\varphi)}\leq C\eps\,\mM_\eps(\normmm{\varphi}+\normmm{\nabla \varphi}),\quad\text{ on }\Sigma,
			\end{align*}
			so that \eqref{P2} follows after an application of \eqref{ctrl1}.
		\end{proof}
		\mich{Making use of Lemma \ref{tracelemma}, we also obtain the following control of the traces on $\Sigma$ and on $\Sigma_{\mathrm p}$.}
		
		\mich{\begin{lemma}
			\label{ctrltraces}
			Let $f\in H^1(\Oe)$. Then it holds
			\begin{align}
				&\norm{\mathrm{Tr}_\Sigma f}_{L^2(\Sigma)}\leq C\eps^{\onehalf}\norm{f}_{H^1(\Oe)}+C\eps^{-\onehalf}\norm{f}_{L^2(\Oe)},\label{ctrltrace1}\\&
				\norm{\mathrm{Tr}_{\Sigma_{\mathrm p}} f}_{L^2(\Sigma_{\mathrm p})}\leq C\eps^{\onehalf}\norm{f}_{H^1(\Oe)}+C\eps^{-\onehalf}\norm{f}_{L^2(\Oe)}.\label{ctrltrace2}
			\end{align}
			In particular, since $\eps\in(0,1)$, each of the two estimates entails
			\begin{align}
				\norm{\mathrm{Tr} f}^2_{L^2(\Sigma)}+\norm{\mathrm{Tr} f}^2_{L^2(\Sigma_{\mathrm p})}\leq C\bigl(\eps^{-1}\norm{f}^2_{L^2(\Oe)}+\eps\norm{\nabla f}^2_{L^2(\Oe)}\bigr)\leq C\eps^{-1}\norm{f}^2_{H^1(\Oe)},
				\label{ctrltrace3}
			\end{align}
			so that the trace of a function of $H^1(\Oe)$ is controlled in $L^2$ on the boundary of the film region with a weight $\eps^{-\onehalf}$ which cannot be dispensed with, as the constant function $f\equiv1$ shows.
            \label{tracelemma1}
		\end{lemma}
		\begin{proof}
			For the sake of simplicity, in this proof we do not distinguish between a function and its trace. Estimate \eqref{ctrltrace1} comes immediately from \eqref{ctrl1} and \eqref{trace}: indeed, it is enough to write
			\begin{align*}
				\norm{f}_{L^2(\Sigma)}\leq \norm{f-\mM_\eps(f)}_{L^2(\Sigma)}+\norm{\mM_\eps(f)}_{L^2(\Sigma)}\leq C\eps^{\onehalf}\norm{f}_{H^1(\Oe)}+C\eps^{-\onehalf}\norm{f}_{L^2(\Oe)}.
			\end{align*}
			To prove \eqref{ctrltrace2}, let us now observe that, using \eqref{changesurf}, \eqref{Lp} and \eqref{chainrules}, together with \eqref{controlJ} and \eqref{ctrltrace1}, we have
			\begin{align*}
				&\int_{\Sigma_{\mathrm p}}\normmm{f}^2\d\mathcal H^{N-1}=\int_\Sigma \normmm{f^\sharp(y,\eps)}^2J(y,\eps)\d\mathcal H^{N-1}\\&\leq 2\int_\Sigma \normmm{f^\sharp(y,\eps)-f(y)}^2J(y,\eps)\d\mathcal H^{N-1}+2\int_\Sigma \normmm{f}^2J(y,\eps)\d\mathcal H^{N-1}\\&
				\leq 
				C\eps\int_\Sigma \int_0^\eps \normmm{\partial_r  f^\sharp(y,s)}^2\d s\d\mathcal H^{N-1}+C\eps\norm{f}_{H^1(\Oe)}^2+C\eps^{-1}\norm{f}_{L^2(\Oe)}^2\\&
				\leq  
				C\eps\norm{f}_{H^1(\Oe)}^2+C\eps^{-1}\norm{f}_{L^2(\Oe)}^2,
			\end{align*}
			which gives \eqref{ctrltrace2} after taking the square root. Finally, squaring \eqref{ctrltrace1} and \eqref{ctrltrace2} and recalling that $\eps\leq\eps^{-1}$ for $\eps\in(0,1)$ yields \eqref{ctrltrace3}.
		\end{proof}}	
		As a consequence, we can prove a crucial generalization of \cite[Lemma 5.8]{Miura2}, including anisotropies, as follows.
		
		\begin{lemma}
			\label{anisolemma}
			Let $L\in C^{0,1}(\R^N;\R^{N\times N}_{sym})$. Then for any  $\vphi\in H^2(\Oe)$ and $\eta\in H^1(\Sigma)$ it holds 
			\begin{align}
				&\nonumber\normmm{\int_{\Oe} L\nabla \vphi\cdot \nabla \widetilde\eta\dx-\eps\int_{\Sigma}{L_{|\Sigma}}\nabla_\Sigma \mM_\eps(\vphi)\cdot \nabla_\Sigma \eta\d \mathcal H^{N-1}}\\&\nonumber\leq C\eps^\frac32\norm{L}_{C^{0,1}(\R^N)}\norm{\varphi}_{H^1(\Omega_\eps)}\norm{\nabla_\Sigma \eta}_{L^2(\Sigma)}\\&
				\quad +\nonumber C\eps^\frac32\norm{L}_{C^0_b(\R^N)}\norm{\varphi}_{H^2(\Omega_\eps)}\norm{\nabla_\Sigma \eta}_{L^2(\Sigma)}\\&\quad+C\eps\norm{L}_{C^0_b(\R^N)}\norm{\partial_{\bn_{\Sigma}} \vphi}_{L^2(\Sigma)}\norm{\nabla_\Sigma \eta}_{L^2(\Sigma)},  
				\label{controlb}
			\end{align}
			for some $C>0$. Additionally, in case $L(\mathbf x) \bP(\mathbf x)=\bP(\mathbf x)L(\mathbf x)$ for any $\mathbf x\in \Sigma,$ then 
			\begin{align}
				&\normmm{\int_{\Oe} L\nabla \vphi\cdot \nabla \widetilde\eta\dx-\eps\int_{\Sigma}L_{|\Sigma}\nabla_\Sigma \mM_\eps(\vphi)\cdot \nabla_\Sigma \eta\d \mathcal H^{N-1}}\leq C\eps^\frac32\norm{L}_{C^{0,1}(\R^N)}\norm{\varphi}_{H^1(\Omega_\eps)}\norm{\nabla_\Sigma \eta}_{L^2(\Sigma)}  
				\label{controlb2}.
			\end{align}

			More generally, if only $L\in C^{1}_b(\R^N;\R^{N\times N}_{sym})$, we deduce, for any $\eta\in H^2(\Sigma)$,
			{\begin{align}
				&\nonumber\normmm{\int_{\Oe} L\nabla \vphi\cdot \nabla \widetilde\eta\dx-\eps\int_{\Sigma}L_{|\Sigma}\nabla_\Sigma \mM_\eps(\vphi)\cdot \nabla_\Sigma \eta\d \mathcal H^{N-1}}\\&\nonumber\leq  C(\norm{L}_{C^{0,1}(\R^N)})\eps^\frac32\norm{\vphi}_{H^1(\Omega_\eps)}\norm{\eta}_{H^1(\Sigma)}\\&\quad +\, C\eps^{\onehalf}(\norm{\bP L_{|\Sigma}\bn_\Sigma}_{L^\infty(\Sigma)}+\norm{\bn_\Sigma L_{|\Sigma}\bP}_{W^{1,\infty}(\Sigma)})\norm{\vphi}_{H^1(\Omega_\eps)}\norm{\eta}_{H^2(\Sigma)}.
				\label{final}
			\end{align}}
		\end{lemma}
		
		\begin{proof}
			We split the difference into four parts as follows:
			\begin{align*}
				\int_{\Oe}L\nabla \vphi\cdot \nabla \widetilde\eta\dx-\eps\int_{\Sigma}L_{|\Sigma}\nabla_\Sigma \mM_\eps(\vphi)\cdot \nabla_\Sigma \eta\d \mathcal H^{N-1}=\sum_{i=1}^4K_i,    
			\end{align*}
			with 
			\begin{align*}
				&K_1:=\int_{\Oe}(L-\widetilde{L_{|\Sigma}})\nabla \varphi\cdot\nabla \widetilde \eta\dx,\\&
				K_2:=\int_{\Oe} \widetilde{L_{|\Sigma}}\nabla \varphi\cdot(\nabla\widetilde \eta-\widetilde{\nabla_\Sigma \eta})\dx,\\&
				K_3:=\int_{\Oe} \widetilde{L_{|\Sigma}}\widetilde{\bP}\nabla\vphi\cdot\widetilde{\nabla_\Sigma\eta}\dx-\eps\int_\Sigma {L_{|\Sigma}}\nabla_\Sigma\mM_\eps(\vphi)\cdot \nabla_\Sigma \eta\d \mathcal H^{N-1},\\&
				K_4:=\int_{\Oe} \widetilde{L_{|\Sigma}}(\bI-\widetilde{\bP})\nabla\vphi\cdot\widetilde{\nabla_\Sigma\eta}\dx=\eps\int_\Sigma \bP L_{|\Sigma} (\bI-{\bP})\mM_\eps(\nabla\vphi)\cdot{\nabla_\Sigma\eta}\d \mathcal H^{N-1}
				.
			\end{align*}
			
			Now, concerning $K_1$, we immediately infer, recalling the regularity $L\in C^{0,1}(\R^N)$ and that $\normmm{d}\leq C\eps$ in $\Omega_\eps$, with $d$ the signed distance function from $\Sigma$, 
			\begin{align*}
				&\normmm{K_1}\leq \norm{L-\widetilde{L_{|\Sigma} }}_{L^\infty(\Omega_\eps)}\norm{\nabla \varphi}_{L^2(\Omega_\eps)}\norm{\nabla \widetilde{\eta}}_{L^2(\Omega_\eps)}\leq C\eps\norm{L}_{W^{1,\infty}(\Oe)}\norm{\nabla \varphi}_{L^2(\Omega_\eps)}\norm{\nabla \widetilde{\eta}}_{L^2(\Omega_\eps)}\\&\leq C\eps^\frac32 \norm{L}_{C^{0,1}(\R^N)}\norm{\nabla \varphi}_{L^2(\Omega_\eps)}\norm{\nabla_\Sigma{\eta}}_{L^2(\Sigma)},
			\end{align*}
			where in the last inequality we made use of \eqref{constextgrad}.
			Then, about $K_2$, we have by \eqref{controlJ} and \eqref{controleta},
			\begin{align*} 
				\normmm{K_2}\leq \norm{L}_{C^0_b(\R^N)}\norm{\nabla \varphi}_{L^2(\Oe)}\norm{\nabla \widetilde\eta-\widetilde{\nabla_\Sigma \eta}}_{L^2(\Oe)}\leq C\norm{L}_{C^0_b(\R^N)}\eps^\frac32\norm{\nabla\vphi}_{L^2(\Oe)}\norm{\nabla_\Sigma\eta}_{L^2(\Sigma)}.
			\end{align*}
			Furthermore, observing that, by \eqref{basic} and since $L\widetilde \bP\nabla \varphi\cdot\widetilde{\nabla_\Sigma\eta}=\widetilde{\bP}L\widetilde \bP\nabla \varphi\cdot \widetilde{\nabla_\Sigma \eta}$ due to ${\bn_{\Sigma}}\cdot\nabla_\Sigma\eta=0$ on $\Sigma$, we have, by \eqref{controlJ} and \eqref{P2},
			\begin{align*}
				&\normmm{K_3}=\normmm{\eps\int_{\Sigma} {( \bP L_{|\Sigma})}\bigl(\mathcal M_\eps(\widetilde \bP\nabla\vphi)-\nabla_\Sigma\mathcal M_\eps(\vphi)\bigr)\cdot \nabla_\Sigma \eta\d \mathcal H^{N-1}}\\&\leq 
				\eps \norm{\bP L}_{C^0_b(\Sigma)}\norm{\mathcal M_\eps(\widetilde \bP\nabla\vphi)-\nabla_\Sigma\mathcal M_\eps(\vphi)}_{L^2(\Sigma)}\norm{\nabla_\Sigma \eta}_{L^2(\Sigma)}\\&\leq C\norm{L}_{C^0_b(\R^N)}\eps^\frac32\norm{\vphi}_{H^1(\Omega_\eps)}\norm{\nabla_\Sigma \eta}_{L^2(\Sigma)}.
			\end{align*}
			In conclusion, concerning $K_4$, we have, recalling the regularity of $L$ and using \eqref{trace},
			\begin{align*}
				& \normmm{K_4}=\normmm{\eps\int_\Sigma\mM_\eps(\nabla\vphi\cdot\widetilde{{\bn_{\Sigma}}})\bP L_{|\Sigma}{\bn_{\Sigma}}\cdot\nabla_\Sigma \eta\d\mathcal H^{N-1}}\\&
				=\normmm{\eps\int_\Sigma\bigl(\mM_\eps(\nabla\vphi\cdot\widetilde{{\bn_{\Sigma}}})-\partial_{\bn_{\Sigma}}\varphi\bigr)\bP L_{|\Sigma}{\bn_{\Sigma}}\cdot\nabla_\Sigma \eta\d\mathcal H^{N-1}+\eps\int_\Sigma\partial_{\bn_{\Sigma}}\varphi\,\bP L_{|\Sigma}{\bn_{\Sigma}}\cdot\nabla_\Sigma \eta\d\mathcal H^{N-1}}\\&
				\leq \eps\norm{\mM_\eps(\nabla\vphi\cdot\widetilde{{\bn_{\Sigma}}})-\partial_{\bn_{\Sigma}}\varphi}_{L^2(\Sigma)}\norm{\bP L_{|\Sigma}{\bn_{\Sigma}}}_{L^\infty(\Sigma)}\norm{\nabla_\Sigma\eta}_{L^2(\Sigma)}\\&\quad +\eps\norm{\partial_{\bn_{\Sigma}}\varphi}_{L^2(\Sigma)}\norm{L_{|\Sigma}{\bn_{\Sigma}}}_{L^\infty(\Sigma)}\norm{\nabla_\Sigma\eta}_{L^2(\Sigma)}\\&
				\leq C\eps^\frac32\norm{L}_{C^0_b(\R^N)}\norm{\nabla\varphi\cdot\widetilde{\bn_{\Sigma}}}_{H^1(\Omega_\eps)}\norm{\nabla_\Sigma\eta}_{L^2(\Sigma)}+C\eps\norm{L}_{C^0_b(\R^N)}\norm{\partial_{\bn_{\Sigma}}\vphi}_{L^2(\Sigma)}\norm{\nabla_\Sigma \eta}_{L^2(\Sigma)},
			\end{align*}
			concluding the proof of \eqref{controlb} by putting together the estimates for $K_1,K_2,K_3,K_4$. To prove \eqref{controlb2}, it is enough to observe that, if $L(\mathbf x) \bP(\mathbf x)=\bP(\mathbf x)L(\mathbf x)$ for any $\mathbf x\in \Sigma,$ then $K_4=0$, since $\bP L\bn_{\Sigma}=L\bP\mathbf n_\Sigma=\mathbf 0$.

			\rev{Finally, to prove \eqref{final}, we come back to the definition of $K_4$ and integrate by parts once more, to obtain
			\begin{align*}
				K_4&=-\int_{\Oe}\vphi \Div\bigl((\bI-\widetilde\bP)\widetilde{L_{|\Sigma}}\widetilde\bP\widetilde{\nabla_\Sigma\eta}\bigr)\dx+\int_{\Sigma_{\mathrm p}}\vphi(\bI-\widetilde\bP)\widetilde{L_{|\Sigma}}\widetilde{\nabla_\Sigma\eta}\cdot \widetilde \bn_\Sigma\d\mathcal H^{N-1}\\&\quad -\int_{\Sigma}\vphi(\bI-\bP){L_{|\Sigma}}{\nabla_\Sigma\eta}\cdot \bn_\Sigma\d\mathcal H^{N-1}\\&
				=-\int_{\Oe}\vphi \Div\bigl((\bI-\widetilde\bP)\widetilde{L_{|\Sigma}}\widetilde\bP\widetilde{\nabla_\Sigma\eta}\bigr)\dx+\int_{\Sigma}\bigl(\vphi^\sharp(y,\eps)-\vphi(y)\bigr)(\bI-\bP){L_{|\Sigma}}{\nabla_\Sigma\eta}\cdot \bn_\Sigma J(y,\eps)\d\mathcal H^{N-1}\\&\quad +\int_{\Sigma}\vphi(y)(\bI-\bP){L_{|\Sigma}}{\nabla_\Sigma\eta}\cdot \bn_\Sigma (J(y,\eps)-1)\d\mathcal H^{N-1}\\&
				=-\int_{\Oe}\vphi \Div\bigl((\bI-\widetilde\bP)\widetilde{L_{|\Sigma}}\widetilde\bP\widetilde{\nabla_\Sigma\eta}\bigr)\dx+\int_{\Sigma}\left(\int_0^\eps\partial_r\vphi^\sharp(y,r)\d r\right)(\bI-\bP){L_{|\Sigma}}{\nabla_\Sigma\eta}\cdot \bn_\Sigma J(y,\eps)\d\mathcal H^{N-1}\\&\quad +\int_{\Sigma}\vphi(y)(\bI-\bP){L_{|\Sigma}}{\nabla_\Sigma\eta}\cdot \bn_\Sigma (J(y,\eps)-1)\d\mathcal H^{N-1}.
			\end{align*}
			Now, recalling that, thanks to the regularity of $\Sigma$, $\nabla \widetilde f=(\bI-d\widetilde W)^{-1}\widetilde{\nabla_\Sigma f}$ in the tubular neighborhood, where $W$ is the Weingarten map \eqref{weingarten}, so that
			\begin{align*}
				\normmm{\nabla\widetilde f}\leq C\normmm{\widetilde {\nabla_\Sigma f}},
			\end{align*}
			we get, after an application of \eqref{basic}, the Cauchy--Schwarz inequality, \eqref{ctrl1}, and recalling that $ L\in C^1_b(\R^N;\R^{N\times N})$,
			\begin{align*}
				&\normmm{\int_{\Oe}\vphi \Div\bigl((\bI-\widetilde\bP)\widetilde{L_{|\Sigma}}\widetilde\bP\widetilde{\nabla_\Sigma\eta}\bigr)\dx}\\&\leq \normmm{\int_{\Oe}{\vphi}(\bI-d\widetilde W)^{-1}:\widetilde{\nabla_\Sigma}\bigl((\bI-\bP){L_{|\Sigma}}\widetilde\bP\bigr)\cdot{\widetilde{\nabla_\Sigma\eta}}\dx}+\normmm{\int_{\Oe}\vphi (\bI-\widetilde\bP)\widetilde{L_{|\Sigma}}\widetilde\bP:(\bI-d \widetilde W)^{-1}\widetilde{\nabla_\Sigma}({\nabla_\Sigma \eta})\dx}\\&
				=\eps\normmm{\int_{\Sigma}\mM_\eps({\vphi})(\bI-d W)^{-1}:\nabla_\Sigma\bigl((\bI-\bP){L_{|\Sigma}\bP}\bigr)\cdot{{\nabla_\Sigma\eta}}\d\mathcal H^{N-1}}\\&\quad +\eps\normmm{\int_{\Sigma}\mM_\eps(\vphi) (\bI-\bP){L_{|\Sigma}}\bP:(\bI-d  W)^{-1}\nabla_\Sigma({\nabla_\Sigma \eta})\d\mathcal H^{N-1}}\\&
				\leq C\eps\norm{(\bI-\bP)L_{|\Sigma}\bP}_{W^{1,\infty}(\Sigma)}\norm{\mM_\eps(\vphi)}_{L^2(\Sigma)}\norm{\eta}_{H^2(\Sigma)}\\&\leq C\eps^{\frac12}\norm{\bn_\Sigma L_{|\Sigma}\bP}_{W^{1,\infty}(\Sigma)}\norm{\vphi}_{H^1(\Oe)}\norm{\eta}_{H^2(\Sigma)},
			\end{align*}
			where in the last inequality we used Lemma \ref{tracelemma1} and the trace embedding $H^1(\Oe)\hookrightarrow L^{2}(\Sigma)$. Concerning the other terms, by the Cauchy--Schwarz inequality we get
			\begin{align*}
				&\normmm{\int_{\Sigma}\left(\int_0^\eps\partial_r\vphi^\sharp(y,r)\d r\right)(\bI-\bP){L_{|\Sigma}}{\nabla_\Sigma\eta}\cdot \bn_\Sigma J(y,\eps)\d\mathcal H^{N-1}}\\&
				\leq C\eps^\frac12\left(\int_\Sigma\int_0^\eps\normmm{\partial_r\vphi^\sharp(y,r)}^2\d r \right)^\onehalf\norm{\bP L_{|\Sigma}(\bI-\bP)\bn_\Sigma}_{L^\infty(\Sigma)}\norm{\eta}_{H^1(\Sigma)}\\&\leq C\eps^\onehalf\norm{\vphi}_{H^1(\Oe)}\norm{\bP L_{|\Sigma}\bn_\Sigma}_{L^{\infty}(\Sigma)}\norm{\eta}_{H^1(\Sigma)},
			\end{align*}
			where we also used \eqref{Lp} thanks to $\partial_r \vphi^\sharp(y,s)=\bigl((\nabla_{\bn_\Sigma} \vphi)^\sharp\bigr)(y,s)$. Finally, using \eqref{controlJ2} we deduce
			\begin{align*}
				&\normmm{\int_{\Sigma}\vphi(y)(\bI-\bP){L_{|\Sigma}}{\nabla_\Sigma\eta}\cdot \bn_\Sigma (J(y,\eps)-1)\d\mathcal H^{N-1}}\\&\leq C\eps\norm{\mathrm{Tr}_\Sigma\vphi}_{L^2(\Sigma)}\norm{\eta}_{H^1(\Sigma)}\\&\leq C\eps^\onehalf\norm{\bn_\Sigma L_{|\Sigma}\bP}_{L^\infty(\Sigma)}\norm{\vphi}_{H^1(\Oe)}\norm{\eta}_{H^1(\Sigma)},
			\end{align*}
			by means of Lemma \ref{tracelemma1}. This, together with the estimates above, concludes the proof of \eqref{final}.}
		\end{proof}
		
		Finally, the time derivative of the averaging operator is addressed through the following result. 
		
		\begin{lemma}
			\label{timederivative}
			Let $T>0$, $a\in C^{0,1}(\R^N\times[0,T])$, \rev{$a \geq 0$},  and assume $\vphi\in H^1(0,T;L^2(\Omega_\eps))$. Then $\mM_\eps(\vphi)\in H^1(0,T;L^2(\Sigma))$ and it holds 
            \begin{align}
				& \nonumber\normmm{\int_0^T\int_{\Oe}a\partial_t \vphi \widetilde\zeta\dx\dt-\eps\int_0^T\int_\Sigma \mM_\eps(a)\partial_t\mM_\eps(\vphi)\zeta\d\mathcal H^{N-1}\dt}\\&\leq C\norm{a}_{C^{0,1}(\R^N\times[0,T])}\eps^\frac32\norm{\varphi}_{H^1(0,T;L^2(\Oe))}\norm{\zeta}_{L^2(0,T;L^2(\Sigma))},\quad \forall \zeta\in L^2(0,T;L^2(\Sigma)),
				\label{timecontrol}
			\end{align}
			for some $C>0$.
		\end{lemma}
		
		\begin{proof}
			The proof is similar to the one of \cite[Lemma 5.9]{Miura2}. First, we note that, for any $\zeta\in C^\infty((0,T);L^2(\Sigma))$ with compact support in $(0,T)$, we have, after an integration by parts in time and using \eqref{basic},
			\begin{align}
				&\int_0^T\int_{\Oe} \widetilde{\mM_\eps(a)}\partial_t\vphi\widetilde\zeta\dx\dt =-\nonumber\int_0^T\int_{\Oe}  \vphi\,{{\partial_t(\widetilde{\mM_\eps(a)\zeta})}\dx}\dt\\&
				=-\eps\int_0^T\int_\Sigma \mM_\eps(\vphi)\partial_t({\mM_\eps(a)}\zeta)\d\mathcal H^{N-1}\dt\label{ident}.
			\end{align}
			Also, by means of \eqref{ctrl1}, we can write
			\begin{align*}
				\norm{\partial_t\mM_\eps(\vphi)}_{L^2(\Sigma)}=\norm{\mM_\eps(\partial_t\vphi)}_{L^2(\Sigma)}\leq C\eps^{-\frac12}\norm{\partial_t\vphi}_{L^2(\Omega_\eps)},
			\end{align*}
			which entails that $\partial_t\mM_\eps(\vphi)\in L^2(0,T;L^2(\Sigma))$, as well as that the identity \eqref{ident} can be written as 
			\begin{align}
				&\int_0^T\int_{\Oe} \widetilde{\mM_\eps(a)}\partial_t\vphi\widetilde\zeta\dx\dt 
				=\eps\int_0^T\int_\Sigma {\mM_\eps(a)}\partial_t\mM_\eps(\vphi)\zeta\d\mathcal H^{N-1}\dt\label{ident2}.
			\end{align}
			\mich{Since both sides of \eqref{ident2} are continuous with respect to $\zeta$ in $L^2(0,T;L^2(\Sigma))$, the identity extends by density to every $\zeta\in L^2(0,T;L^2(\Sigma))$.}
			We can now continue the proof, by splitting the terms in \eqref{timecontrol} as follows
			\begin{align}
				\int_0^T\int_{\Oe}a\partial_t \vphi \widetilde\zeta\dx\dt-\eps\int_0^T\int_\Sigma \mM_\eps(a)\partial_t\mM_\eps(\vphi)\zeta\d\mathcal H^{N-1}\dt=J_1+J_2,
				\label{timecontrol1}
			\end{align}
			where 
			\begin{align*}
				J_1:=\int_0^T\int_{\Oe}a\partial_t \vphi \widetilde\zeta\dx\dt-\int_0^T\int_{\Oe}\widetilde{\mM_\eps(a)}\partial_t \vphi \widetilde\zeta\dx\dt,
			\end{align*}
			and 
			\begin{align*}
				J_2=\int_0^T\int_{\Oe}\widetilde{\mM_\eps(a)}\partial_t \vphi \widetilde\zeta\dx\dt-\eps\int_0^T\int_\Sigma \mM_\eps(a)\partial_t\mM_\eps(\vphi)\zeta\d\mathcal H^{N-1}\dt=0,
			\end{align*}
			by \eqref{ident2}. 
			Therefore, we only need to estimate $J_1$, which can be done by the Cauchy-Schwarz inequality using \eqref{ctrl2} with $p=\infty$:
			\begin{align*}
				&\normmm{J_1}\leq \norm{a-\widetilde{\mM_\eps(a)}}_{L^\infty(0,T;L^\infty(\Oe))}\norm{\partial_t\varphi}_{L^2(0,T;L^2(\Oe))}\norm{\widetilde\zeta}_{L^2(0,T;L^2(\Oe))}\\&
				\leq C\norm{a}_{C^{0,1}(\R^N\times[0,T])}\eps^\frac32\norm{\varphi}_{H^1(0,T;L^2(\Oe))}\norm{\zeta}_{L^2(0,T;L^2(\Sigma))},
			\end{align*}
			where in the last estimate we also used \eqref{constext}. This concludes the proof of \eqref{timecontrol}.
		\end{proof}
        
        		\subsection{Main Theoretical Results and Proofs}
		\label{mat_sec4}
        
		\mich{In the present section we establish the main theoretical results concerning the rescaled film-bulk thermal transmission problem, for the finite-thickness film configuration and for the corresponding thin-film limit. In order to obtain the higher-order estimates in a simpler way, from now on we set $c=a$, that is, the two regions are assumed to share the same volumetric heat capacity; the general case $c\neq a$ can be treated along the same lines, at the price of some additional technical overhead. Accordingly, in the geometric setting of Section~\ref{mat_sec1}, in which $\Sigma_v$ and $\Sigma_{\mathrm q}$ are empty, the rescaled problem \eqref{eps1} reads}
			\begin{align}
				\label{eps1b}
				\begin{cases}
					a\,\partial_t u - \Div\!\big( \mathbf B\nabla u\big) =  g
					& \text{in } \Omega_\varepsilon,\\[4pt]
					a\,\partial_t v - \Div\!\big( \mathbf D\nabla v\big) = h
					& \text{in } \Omega_b,\\[4pt]
					u = v
					& \text{on } \Sigma,\\[4pt]
					\mathbf  B\nabla u \cdot \mathbf n_\Sigma
					= \varepsilon\, \mathbf D\nabla v \cdot \mathbf n_\Sigma
					& \text{on } \Sigma,\\[4pt]
					\mathbf  B\nabla u \cdot \mathbf n_{\mathrm p} = \varepsilon\,\op
					& \text{on } \Sigma_{\mathrm p},
				\end{cases}
			\end{align}
			\mich{with the initial conditions
			\[
			u(0) = u_0 \quad \text{in } \Omega_\varepsilon,
			\qquad
			v(0) = v_0 \quad \text{in } \Omega_b,
			\qquad
			u_0 = v_0 \quad \text{on } \Sigma.
			\]}
			As $\varepsilon \to 0$, problem \eqref{eps1b} formally reduces to
			\begin{align}
				\label{thinfilm11b}
				\begin{cases}
					a\,\partial_t v - \Div\!\big( \mathbf D\nabla v\big) = h
					& \text{in } \Omega_b,\\[3pt]
					u = v
					& \text{on } \Sigma,\\[3pt]
					a\,\partial_t u
					- \Div_\Sigma\!\big(\bP\mathbf  B_{|\Sigma}\bP\nabla_\Sigma u\big)
					= -\,\mathbf  D\nabla v \cdot \mathbf n_\Sigma
					+ g + \op
					& \text{on } \Sigma,
				\end{cases}
			\end{align}
			with
			\[
			u_0 = v_0 \quad \text{on } \Sigma,
			\qquad
			u(0) = u_0 \quad \text{on } \Sigma,
			\qquad
			v(0) = v_0 \quad \text{in } \Omega_b.
			\]
			\mich{Throughout this section, $\Omega_\delta$ denotes the film region \eqref{thindomain} corresponding to the maximal thickness $\delta$ fixed in Section~\ref{mat_sec1}, so that $\Omega_\eps\subset\Omega_\delta$ for every $\eps\in(0,\delta)$, and $\overline\Omega_\delta$ denotes its closure.}
		
		We first establish existence, uniqueness and regularity results
		for $\varepsilon>0$, together with \textit{a priori} estimates that make the
		dependence on $\varepsilon$ explicit.
		
		\begin{theorem}[Well-posedness for $\eps\in(0,\delta)$]
			\label{thmnotthin}
			Let $a\in C^1_b(\R^N\times [0,\infty))$ and $\mathbf B,\mathbf D\in C^1_b(\R^N\times[0,\infty);\R^{N\times N}_{sym})$, and assume that there exists $c_0>0$ such that
			\[
			c_0\leq a(\mathbf x,t),
			\qquad
			c_0\normmm{\xi}^2\leq \xi\cdot \mathbf B(\mathbf x,t)\,\xi,
			\qquad
			c_0\normmm{\xi}^2\leq \xi\cdot \mathbf D(\mathbf x,t)\,\xi,
			\]
			for every $\xi\in \R^N$ and every $(\mathbf x,t)\in\R^N\times[0,\infty)$. Assume also that $(u_0,v_0)\in L^2 (\Omega_\eps)\times L^2(\Omega_b)$, $g\in L^2(\Omega_\eps)$, $h\in L^2(\Omega_b)$ and $\op\in L^2(\Sigma_{\mathrm p})$. Then there exists a unique globally defined weak solution $(u,v)$ to \eqref{eps1b}, satisfying $\mathrm{Tr}_{\Sigma}u=\mathrm{Tr}_{\Sigma} v$ almost everywhere on $\Sigma$, such that, for any $T>0$, 
			\begin{align}
				&\nonumber\norm{u}_{L^\infty(0,T;L^2(\Omega_\eps))}+\sqrt\eps \norm{v}_{L^\infty(0,T;L^2(\Omega_b))}+\norm{u}_{L^2(0,T;H^1(\Omega_\eps))}+\sqrt\eps \norm{v}_{L^2(0,T;H^1(\Omega_b))}\\&
				+\norm{\partial_t u}_{L^2(0,T;H^1(\Omega_\eps)')}+\sqrt\eps \norm{\partial_t v}_{L^2(0,T;H^1(\Omega_b)')}\nonumber\\&\leq C(T)(\sqrt\eps\norm{v_0}_{L^2(\Omega_b)}+\norm{u_0}_{L^2(\Omega_\eps)})+C(T)(\norm{g}_{L^2(\Omega_\eps)}+\sqrt\eps\norm{h}_{L^2(\Omega_b)}+\sqrt\eps\norm{\op}_{L^2(\Sigma_{\mathrm p})}),\label{est1}
			\end{align}
			and such that, for almost any $t>0$,
			\begin{align}
				\nonumber&\langle\partial_t u,  a(\mathbf x,t)\zeta\rangle_{H^1(\Omega_\eps)',H^1(\Omega_\eps)}+\eps\langle\partial_t v,  a(\mathbf x,t)\zeta\rangle_{H^1(\Omega_b)',H^1(\Omega_b)}\\&\quad +\int_{\Omega_\eps} \mathbf B(\mathbf x,t)\nabla u\cdot \nabla \zeta\dx+\eps\int_{\Omega_b} \mathbf D(\mathbf x,t)\nabla v\cdot \nabla \zeta\dx\nonumber\\&=\int_{\Omega_\eps}g\zeta\dx+\eps \int_{\Omega_b}h\zeta\dx +\eps\int_{\Sigma_{\mathrm p}}\op\,\zeta\d\mathcal H^{N-1},\label{weakform}
			\end{align}
            for any $\zeta\in H^1(\Omega)$. %
			
			If additionally $(u_0,v_0)\in H^1(\Oe)\times H^1(\Omega_b)$, with $\mathrm{Tr}_{\Sigma} u_0=\mathrm{Tr}_{\Sigma}v_0$ a.e. on $\Sigma$, then the global solution satisfies 
			\begin{align}
				&\nonumber\norm{u}_{L^\infty(0,T;H^1(\Omega_\eps))}+\sqrt\eps \norm{v}_{L^\infty(0,T;H^1(\Omega_b))}+\norm{\Div(\mathbf B\nabla u)}_{L^2(0,T;L^2(\Omega_\eps))}+\sqrt\eps \norm{\Div(\mathbf D\nabla v)}_{L^2(0,T;L^2(\Omega_b))}\\&
				+\norm{\partial_t u}_{L^2(0,T;L^2(\Omega_\eps))}+\sqrt\eps \norm{\partial_tv}_{L^2(0,T;L^2(\Omega_b))}\nonumber\\&\leq C(T)(\sqrt\eps(1+\norm{\op}_{L^2(\Sigma_p)})+\sqrt\eps\norm{v_0}_{H^1(\Omega_b)}+\norm{u_0}_{H^1(\Omega_\eps)})\nonumber\\&\quad +C(T)(\norm{g}_{L^2(\Omega_\eps)}+\sqrt\eps\norm{h}_{L^2(\Omega_b)}+\sqrt\eps\norm{\op}_{L^2(\Sigma_{\mathrm p})}),\label{est2}
			\end{align}
			for any $T>0$.
		\end{theorem}
		\mich{\begin{remark}
			\label{rem0}
			We point out that, in general, we cannot infer the full-norm control $(u,v)\in L^2(0,T;H^2(\Omega_\eps)\times H^2(\Omega_b))$, since the boundary regularity depends on $\norm{\op}_{H^\frac12({\Sigma_{\mathrm p}})}$ and, more critically, on the $H^\frac12(\Sigma)$-norm of the interface flux $\mathbf B\nabla u\cdot\bn_\Sigma=\eps\, \mathbf D\nabla v\cdot\bn_\Sigma$, for which no uniform control is available. The issue is genuinely nontrivial for transmission problems, although it can be settled in some cases (see, e.g., \cite{Caffarelli}). Nevertheless, we will not make use of this additional regularity when studying the convergence rates as $\varepsilon\to0$.
		\end{remark}}
		
		An analogous well-posedness result holds for the thin-film limit problem \eqref{thinfilm11b}.
		\begin{theorem}[Well-posedness for the thin-film limit problem]
			\label{thmthin}
			Let  $a\in C^1_b(\R^N\times [0,\infty))$ and $\mathbf B,\mathbf D\in C^1_b(\R^N\times[0,\infty);\R^{N\times N}_{sym})$, and assume that there exists $c_0>0$ such that
			\(
			0<c_0\leq a({\mich{\mathbf x,t}}),
            \)
            as well as \(c_0\normmm{\xi}^2\leq \xi\cdot \mathbf B(\mathbf x,t)\,\xi,
            \) and \(c_0\normmm{\xi}^2\leq \xi\cdot \mathbf D(\mathbf x,t)\,\xi,
			\)
			for every $\xi\in \R^N$ and every $(\mathbf x,t)\in\R^N\times[0,\infty)$. \mich{Assume also that $v_0\in L^2(\Omega_b)$, $u_0\in L^2(\Sigma)$, $h\in L^2(\Omega_b)$ and $g,\op\in L^2(\Sigma)$, and set
			\[
			E_\Sigma:=\{f\in H^1(\Omega_b):\ \mathrm{Tr}_\Sigma f\in H^1(\Sigma)\}.
			\]}
			Then there exists a unique globally defined weak solution $(u,v)$ to \eqref{thinfilm11b}, \mich{with $u=\mathrm{Tr}_\Sigma v$ a.e. on $\Sigma$, such that, for any $T>0$, } 
			\begin{align}
				&\nonumber\norm{v}_{L^\infty(0,T;L^2(\Omega_b))}+ \norm{u}_{L^\infty(0,T;L^2(\Sigma))}+\norm{v}_{L^2(0,T;H^1(\Omega_b))}+ \norm{u}_{L^2(0,T;H^1(\Sigma))}\\&
				+\norm{\partial_tv}_{L^2(0,T;E_\Sigma')}\leq C(T)(\norm{v_0}_{L^2(\Omega_b)}+\norm{u_0}_{L^2(\Sigma)})+C(T)\left(\norm{\op}_{L^2(\Sigma)}+\norm{h}_{L^2(\Omega_b)}+\norm{g}_{L^2(\Sigma)}\right),\label{est3}
			\end{align}
			\mich{where $v$ is understood as an $E_\Sigma$-valued function, so that $\partial_t v$ acts on any $\zeta\in E_\Sigma$ through both duality pairings appearing in \eqref{thinf} below,}
			and such that, for almost any $t>0$,
			\begin{align}
				\nonumber&\langle\partial_t u,  a(\mathbf x,t)\zeta\rangle_{H^1(\Sigma)',H^1(\Sigma)}+\langle\partial_t v,  a(\mathbf x,t)\zeta\rangle_{H^1(\Omega_b)',H^1(\Omega_b)}\\&\quad +\int_{\Sigma}  \mathbf B_{|\Sigma}(\mathbf x,t)\nabla_\Sigma u\cdot \nabla_\Sigma \zeta\d\mathcal H^{N-1}+\int_{\Omega_b} \mathbf D(\mathbf x,t)\nabla v\cdot \nabla \zeta\dx\nonumber\\&=\int_{\Sigma} (g+\op)\,\zeta\d \mathcal H^{N-1}+\int_{\Omega_b}h\zeta\dx, \label{thinf}
			\end{align}
			\mich{for any $\zeta\in E_\Sigma$. Note that only the tangential part $\bP\mathbf B_{|\Sigma}\bP$ of the film conductivity enters \eqref{thinf}, since $\nabla_\Sigma\zeta$ is tangential: this is why the interfacial equation in \eqref{thinfilm11b} is written with $\bP\mathbf B_{|\Sigma}\bP$ (cf. Remark~\ref{rem:structural}).}
			
			If additionally $v_0\in H^1(\Omega_b)$ with $u_0=\mathrm{Tr}_\Sigma v_0\in H^1(\Sigma)$, then the global solution is also strong and it holds 
			\begin{align}
				&\nonumber\norm{u}_{L^\infty(0,T;H^1(\Sigma))}+ \norm{v}_{L^\infty(0,T;H^1(\Omega_b))}+\norm{u}_{L^2(0,T;H^2(\Sigma))}+ \norm{v}_{L^2(0,T;H^2(\Omega_b))}\\&
				+\norm{\partial_tu}_{L^2(0,T;L^2(\Sigma))}+ \norm{\partial_tv}_{L^2(0,T;L^2(\Omega_b))}\nonumber\\&\leq C(T)(\norm{v_0}_{H^1(\Omega_b)}+\norm{u_0}_{H^1(\Sigma)})+C(T)\left(\norm{\op}_{L^2(\Sigma)}+\norm{h}_{L^2(\Omega_b)}+\norm{g}_{L^2(\Sigma)}\right),\label{extracondition}
			\end{align}
			for any $T>0$.
		\end{theorem}
		\mich{\begin{remark}
			Notice that, in contrast with Remark \ref{rem0}, in the thin-film limit configuration the full $H^2(\Omega_b)$-control on $v$ is available. This will be crucial in the study of the limit as $\varepsilon\to0$.
		\end{remark}}
		Passing to the thin-film limit as $\eps\to0$, under suitable assumptions we obtain the following result.
		\begin{theorem}[Limit as $\eps\to0$]
			\label{thmconvergence} 
			Let  $a\in C^1_b(\R^N\times [0,\infty))$ and $\mathbf B,\mathbf D\in C^1_b(\R^N\times[0,\infty);\R^{N\times N}_{sym})$, and assume that there exists $c_0>0$ such that
			\(0<
						c_0\leq a(\mich{\mathbf x,t}),\) as well as \(c_0\normmm{\xi}^2\leq \xi\cdot \mathbf B(\mathbf x,t)\,\xi,\)
			and \(c_0\normmm{\xi}^2\leq \xi\cdot \mathbf D(\mathbf x,t)\,\xi,
			\)
			for every $\xi\in \R^N$ and every $(\mathbf x,t)\in\R^N\times[0,\infty)$.
			Let $v_{0,\eps}\in H^1(\Omega)$ be such that 
			\begin{align*}
				v_{0,\eps}=\begin{cases}
					u_{0,\eps},\quad&\text{ in }\Omega_\eps,\\
					v_0,\quad&\text{ in }\Omega_b,
				\end{cases}
			\end{align*}
			with $ \mathrm{Tr}_{\Sigma}u_{0,\eps}=\mathrm{Tr}_{\Sigma}v_0 $, and 
			\begin{align}\label{initial}
				\norm{u_{0,\eps}}_{H^1(\Omega_\eps)}\leq C\eps^\frac12, 
			\end{align}
			for some $C>0$ independent of $\eps$. Assume also that $g\in H^1(\Omega_\eps)$ satisfies
			\begin{align}
				\norm{g}_{L^2(\Omega_\eps)}\leq C\eps^\frac12,\quad \norm{g}_{H^1(\Omega_\eps)}\leq C\eps^{-\alpha},\label{f1cond}
			\end{align}
			{\mich{uniformly in $\eps>0$, for some $\alpha\in[0,\tfrac12)$, and that $h\in L^2(\Omega_b)$ and $\op\in C^{0,1}(\overline{\Omega}_\delta)$.
			Consider the sequence $\{(u_\eps,v_\eps)\}_{\eps>0}$ of solutions to problem \eqref{eps1b} given by Theorem \ref{thmnotthin}. Assume finally that either $\bP(\mathbf x)\mathbf B(\mathbf x,\cdot)=\mathbf B(\mathbf x,\cdot)\bP(\mathbf x)$ for any $\mathbf x\in \Sigma$, or that the following control holds:}}
			\begin{align}
				\norm{u_\eps}_{L^2(0,T;H^2(\Oe))}+\sqrt \eps\norm{v_\eps}_{L^2(0,T;H^2(\Omega_b))}\leq C(T)\sqrt\eps,\label{ctrl1a}
			\end{align}
			for some $C(T)>0$.
			
			\mich{Then there exists a globally defined $v:\Omega_b\times [0,\infty)\to \R$ such that }
			\begin{align*}
				&v_\eps \to v,\quad \text{weakly$^*$ in } L^\infty(0,T;H^1(\Omega_b)),
				\\&
				v_\eps \to v,\quad \text{strongly in } C([0,T];H^s(\Omega_b)),\quad s\in(\tfrac12,1),\\&
				\partial_t v_\eps \to \partial_t v,\quad \text{weakly in } L^2(0,T;L^2(\Omega_b)),
			\end{align*}
			for any $T>0$. Furthermore, introducing the averaged quantity 
			\[
			w_\eps:=\mM_\eps(u_\eps),
			\]
			where $\mM_\eps$ is defined in \eqref{average}, it holds 
			\begin{align*}
				&w_\eps \to u:=\mathrm{Tr}_{\Sigma}v,\quad \text{weakly$^*$ in } L^\infty(0,T;L^2(\Sigma)),\\&
				w_\eps \to u,\quad \text{weakly in } L^2(0,T;H^1(\Sigma)),
				\\&
				w_\eps \to u,\quad \text{strongly in } L^2(0,T;L^2(\Sigma)),\\&
				\partial_t w_\eps \to \partial_t u,\quad \text{weakly in } L^2(0,T;H^1(\Sigma)'),
			\end{align*}
			for any $T>0$. 
			
			Additionally, $(u,v)$ is the unique weak solution to problem \eqref{thinfilm11b} in the sense of Theorem \ref{thmthin}, with initial datum $v(0)= v_0 $ almost everywhere in $\Omega_b$ and $u(0)=u_0:=\mathrm{Tr}_\Sigma v_0$ almost everywhere on $\Sigma$.
		\end{theorem}
        
		\begin{remark}[Assumption $\bP \mathbf B=\mathbf B\bP$]
			\label{rem:structural}
			\mich{As follows from the weak formulation \eqref{thinf} of the limit problem, the effective conductivity tensor associated with the film region is $\bP\mathbf B_{|\Sigma}\bP$, so that only the tangential component of $\mathbf B$ contributes to the effective balance relation posed on the interface. A particularly relevant situation corresponds to conductivity tensors satisfying $\mathbf B\bP=\bP\mathbf B$ on $\Sigma$, as assumed, alternatively to \eqref{ctrl1a}, in Theorem \ref{thmconvergence}, and for the first convergence rate of Theorem \ref{thmrates}. This condition is equivalent to requiring that}
			\[
			(\mathbf B\bn_{\Sigma})\otimes \bn_{\Sigma}=\bn_{\Sigma}\otimes(\mathbf B\bn_{\Sigma})
			\qquad\text{on }\Sigma.
			\]
			\mich{To provide a geometric interpretation of this property, consider the case $N=3$ and let $\{\tau_1,\tau_2\}$ be an orthonormal basis of the tangent space to $\Sigma$, so that $\{\tau_1,\tau_2,\bn_\Sigma\}$ is an orthonormal basis of $\R^3$. With respect to it, a symmetric tensor $\mathbf B$ is represented by}
			\begin{align*}
				\begin{bmatrix}
					B_{\tau_1\tau_1}&B_{\tau_1\tau_2}&B_{\tau_1\bn_{\Sigma}}\\
					B_{\tau_1\tau_2}&B_{\tau_2\tau_2}&B_{\tau_2\bn_{\Sigma}}\\
					B_{\tau_1\bn_{\Sigma}}&B_{\tau_2\bn_{\Sigma}}&B_{\bn_{\Sigma}\bn_{\Sigma}}
				\end{bmatrix},
			\end{align*}
			\mich{and the previous condition reads}
			\begin{align*}
				(B_{\tau_1\bn_{\Sigma}},B_{\tau_2\bn_{\Sigma}},B_{\bn_{\Sigma}\bn_{\Sigma}})\otimes (0,0,1)=(0,0,1)\otimes (B_{\tau_1\bn_{\Sigma}},B_{\tau_2\bn_{\Sigma}},B_{\bn_{\Sigma}\bn_{\Sigma}}),
			\end{align*}
			\mich{which holds if and only if $B_{\tau_1\bn_\Sigma}=B_{\tau_2\bn_\Sigma}=0$. The tensor therefore takes the block-diagonal form}
			\begin{align*}
				\begin{bmatrix}
					B_{\tau_1\tau_1}&B_{\tau_1\tau_2}&0\\
					B_{\tau_1\tau_2}&B_{\tau_2\tau_2}&0\\
					0&0&B_{\bn_{\Sigma}\bn_{\Sigma}}
				\end{bmatrix},
			\end{align*}
			\mich{in which the tangential and the normal conductive contributions are completely decoupled.}
			
			\mich{In particular, under this condition $\bn_\Sigma$ is an eigenvector of $\mathbf B_{|\Sigma}$, so that $\mathbf B_{|\Sigma}\nabla_\Sigma u$ is a tangential field for every $u$. Since $\Div_\Sigma\mathbf F=\Div_\Sigma(\bP\mathbf F)-(\mathbf F\cdot\bn_\Sigma)H_\Sigma$ for any vector field $\mathbf F$, the curvature contribution then vanishes and the interfacial operator in \eqref{thinfilm11b} reduces to $\Div_\Sigma(\mathbf B_{|\Sigma}\nabla_\Sigma u)$. On the other hand, if the decoupling does not hold and $\bP\mathbf B_{|\Sigma}\bP$ does not coincide with $\mathbf B_{|\Sigma}$, the only effective component in the interfacial operator is still $\bP\mathbf B_{|\Sigma}\bP$ and the system would lose information during the passage to the limit $\varepsilon\to0$. Indeed, as expected, the quantities $\norm{\bP\mathbf B_{|\Sigma}\bn_\Sigma}_{L^\infty(\Sigma)}$ and $\norm{\bn_\Sigma\mathbf B_{|\Sigma}\bP}_{W^{1,\infty}(\Sigma)}$, which measure the failure of the condition, are precisely the ones entering the general convergence rate \eqref{C2} below.}
		\end{remark}
        
		\begin{remark}[Regularity of the initial datum]
			Condition \eqref{initial} on the initial datum is easily satisfied, for instance by assuming additionally that $u_{0, \eps}\in W^{1,\infty}(\Omega_\eps)$ uniformly in $\eps$, since $\normmm{\Omega_\eps}\leq C\eps$.
		\end{remark}
		\begin{remark}[Trace control]
			Observe that in general $\mathrm{Tr}_{\Sigma}v_\eps$ differs from $w_\eps$. \mich{Estimate \eqref{trace} is precisely what allows this identification to be recovered in the limit as $\eps\to0$. A similar issue arises for the initial datum, since }
			\[
			w_\eps(0)=\mM_\eps(u_{0,\eps}),
			\]
			and $\mathrm{Tr}_\Sigma v_{0,\eps}$ differs in general from $w_\eps(0)$. On the other hand, by \eqref{trace} and assumption \eqref{initial}, we obtain
			\begin{align}
				\label{control}
				\norm{\mathrm{Tr}_\Sigma  v_{0,\eps}-w_\eps(0)}_{L^2(\Sigma)}\leq C\eps^\frac12\norm{u_{0,\eps}}_{H^1(\Omega_\eps)}\leq C\eps,
			\end{align}
			\mich{which entails $w_\eps(0)\to \mathrm{Tr}_\Sigma  v_0 $ in $L^2(\Sigma)$ as $\eps\to0$, and therefore, in the limit, $u(0)=\mathrm{Tr}_{\Sigma} v_0=:u_0 $ almost everywhere on $\Sigma$.}
		\end{remark}
		\begin{remark}[Forcing term regularity]
			\mich{The first condition in \eqref{f1cond} is easily satisfied, for instance if $g\in L^\infty(\Omega_\eps)\cap H^1(\Oe)$ with $\norm{g}_{L^\infty(\Oe)}\leq C$ uniformly in $\eps$, again since $\normmm{\Omega_\eps}\leq C\eps$. The second one allows the $H^1(\Oe)$-norm of $g$ to blow up as $\eps\to0$, provided the blow-up rate is slower than $\eps^{-\frac12}$.}
		\end{remark}
		\begin{remark}[Further regularity assumption on $(u_\eps,v_\eps)$]
			\label{rem1}
			As observed in Remark \ref{rem0}, a complete $H^2$-estimate on $(u_\eps,v_\eps)$, as required in \eqref{ctrl1a}, does not appear simple nor natural to obtain. \mich{Assumption \eqref{ctrl1a} is therefore stated as an alternative to the more reasonable constitutive condition $\mathbf B\bP=\bP \mathbf B$ on $\Sigma$: the convergence result holds either under that condition on the coefficients, or under additional regularity of the solutions.}
		\end{remark}
		
		We conclude by giving an explicit rate of convergence of the solution $(u_\eps,v_\eps)$, for $\eps>0$, towards the thin-film solution. \rev{This can also be done without the assumption $\bP\mathbf B=\mathbf B\bP$ on $\Sigma$, obtaining a quantitative estimate which depends on the discrepancy between $\bP\mathbf B$ and $\mathbf B\bP$ on $\Sigma$.}
        
		\begin{theorem}[Convergence rates as $\eps\to 0$]
			\label{thmrates}
			Let  $a\in C^1_b(\R^N\times [0,\infty))$, $\mathbf D\in C^1_b(\R^N\times[0,\infty);\R^{N\times N}_{sym})$ and $\mathbf B\in C^{1,\gamma}(\R^N\times[0,\infty);\R^{N\times N}_{sym})$ for some $\gamma\in(0,1]$, and assume that there exists $c_0>0$ such that		
			$c_0\leq a(\mathbf x,t),$ as well as
			\(c_0\normmm{\xi}^2\leq \xi\cdot \mathbf B(\mathbf x,t)\,\xi,
			\)
			and \(c_0\normmm{\xi}^2\leq \xi\cdot \mathbf D(\mathbf x,t)\,\xi,
			\)
			for every $\xi\in \R^N$ \mich{and every $(\mathbf x,t)\in\R^N\times[0,\infty)$. }Assume also that, for \eqref{rate2}, $\bP(\mathbf x)\mathbf B(\mathbf x,\cdot)=\mathbf B(\mathbf x,\cdot)\bP(\mathbf x)$ for any $\mathbf x\in \Sigma$. 
			Let $v_{0,\eps}\in H^1(\Omega)$ be such that 
			\begin{align*}
				v_{0,\eps}=\begin{cases}
					u_{0,\eps},\quad&\text{ in }\Omega_\eps,\\
					v_0,\quad&\text{ in }\Omega_b,
				\end{cases}
			\end{align*}
			with $ \mathrm{Tr}_{\Sigma}u_{0,\eps}=\mathrm{Tr}_{\Sigma}v_0 $ and $ v_0 \in H^1(\Omega_b)$ with $u_0:= \mathrm{Tr}_\Sigma  v_0 \in H^1(\Sigma)$. \mich{ Let then $\op,g\in C^{0,1}(\overline\Omega_{{\delta}})$ and $h\in L^2(\Omega_b)$.}
			
			Let $\{(u_\eps,v_\eps)\}_{\eps>0}$ be the sequence of solutions to \eqref{eps1b} given by Theorem \ref{thmnotthin} and satisfying \eqref{est2}, and let $(u,v)$ be the strong solution to \eqref{thinfilm11b} satisfying \eqref{extracondition}, given by Theorem \ref{thmthin} with initial datum $(u_0,v_0)$.
			Then it holds 
			\begin{align}
				&\nonumber\eps^{-\frac12}\norm{u_\eps-\widetilde{u}}_{L^\infty(0,T;L^2(\Omega_\eps))}+\eps^{-\frac12}\norm{u_\eps-\widetilde{ u}}_{L^2(0,T;H^1(\Omega_\eps))}\\&\quad +\norm{v_\eps-{v}}_{L^\infty(0,T;L^2(\Omega_b))}+{\norm{v_\eps-v}_{L^2(0,T;H^1(\Omega_b))}}\leq C(T)\left(\eps^{-\frac12}\norm{u_{0,\eps}-\widetilde{u_0}}_{L^2(\Oe)}+\eps\right),
				\label{rate2}
			\end{align}
			for any $T>0$.
\rev{
			Finally, without the assumption $\bP\mathbf B=\mathbf B\bP$ on $\Sigma$, it holds more generally
			\begin{align}
				&\nonumber\eps^{-\frac12}\norm{u_\eps-\widetilde{u}}_{L^\infty(0,T;L^2(\Omega_\eps))}+\eps^{-\frac12}\norm{u_\eps-\widetilde{ u}}_{L^2(0,T;H^1(\Omega_\eps))}\\&\nonumber+\norm{v_\eps-{v}}_{L^\infty(0,T;L^2(\Omega_b))}+{\norm{v_\eps-v}_{L^2(0,T;H^1(\Omega_b))}}\\&\leq C(T)\left(\eps^{-\frac12}\norm{u_{0,\eps}-\widetilde{u_0}}_{L^2(\Oe)}+\norm{\bP\mathbf B_{|\Sigma}\bn_\Sigma}_{L^\infty(\Sigma)}+\norm{\bn_\Sigma\mathbf B_{|\Sigma}\bP}_{W^{1,\infty}(\Sigma)}+\eps\right),
				\label{C2}
			\end{align}
            for any $T>0$.
}
		\end{theorem}
		\begin{remark}
			As observed in \cite{Miura2}, the $L^2(\Omega_\eps)$-norms in \eqref{rate2} are multiplied by $\eps^{-\frac12}$, since such norms involve the square root of the thickness of $\Omega_\eps$, which is of order $\eps$. Notice also that, if for instance $u_{0,\eps}=\widetilde{u_0}$, then the rate of convergence in \eqref{rate2} is $\eps$.
		\end{remark}
		\begin{remark}
			\mich{Compared with the assumptions of Theorem \ref{thmconvergence}, here we do not require the additional $H^2$-control \eqref{ctrl1a} on $\{(u_{\eps},v_\eps)\}_\eps$, since the structural condition $\bP\mathbf B=\mathbf B\bP$ on $\Sigma$ is assumed for \eqref{rate2}. On the other hand, we assume here that the limit solution $(u,v)$ is more regular, namely that it satisfies \eqref{extracondition}. Also, $\mathbf B$ and $g$ are required to be more regular than in Theorem \ref{thmconvergence}, since a Lipschitz condition on them is needed in order not to spoil the convergence rate as $\eps\to0$}.
		\end{remark}
\rev
		{\begin{remark}[The case when $\bP \mathbf B\not=\mathbf B\bP$ on $\Sigma$]
			\label{rem:noncommuting}
			The more general case when $\bP \mathbf B\not=\mathbf B\bP$ is allowed in the theorem above. One sees that the rate of convergence depends on the discrepancy between these two quantities. Indeed, we have
			\[
			\norm{\bP\mathbf B_{|\Sigma}\bn_\Sigma}_{L^\infty(\Sigma)}=\norm{(\bP\mathbf B_{|\Sigma}-\mathbf B_{|\Sigma}\bP)\bn_\Sigma}_{L^\infty(\Sigma)},
			\]
            \begin{align*}
                \norm{\bn_\Sigma\mathbf B_{|\Sigma}\bP}_{W^{1,\infty}}=\norm{\bn_\Sigma(\bP\mathbf B_{|\Sigma}-\mathbf B_{|\Sigma}\bP)}_{W^{1,\infty}(\Sigma)}.
            \end{align*}
			If we assume these discrepancies to vanish with rate at least of order $\eps$, for instance choosing $\bP\mathbf B\bn_\Sigma$ depending on $\eps$ and such that $\sup_{\eps>0}\norm{\mathbf B_\varepsilon}_{C^{1,\gamma}(\R^N\times[0,\infty))}\leq C$ (and extending the results above in a straightforward way to this $\eps$-dependent case), then the rate of convergence from \eqref{C2} is preserved with respect to the case when $\bP\mathbf B=\mathbf B\bP$ on $\Sigma$. In a more physical interpretation, estimate \eqref{C2} shows that the normal heat flux due to the thin-film domain at the interface $\Sigma$ cannot have a too large tangential component, and this component has to vanish as $\eps\to0$ to guarantee the expected convergence to the limit thin-film configuration in the rather weak norms above.
		\end{remark}}
        
		\subsubsection{Proof of Theorem \ref{thmnotthin}}
		\label{sec:proofnonthin}
		\mich{Existence of solutions to \eqref{eps1b} can be obtained, for instance, by means} of a Galerkin approximation scheme or of a time discretization scheme. Nevertheless, since this is standard in the literature on parabolic PDEs, we only show the necessary formal energy estimates (we refer, for instance, to \cite{Miura2} for similar computations). The weak formulation we consider is \eqref{weakform}. In particular, we choose $\zeta$ so that $\zeta=u$ in $\Omega_\eps$ and $\zeta=v$ in $\Omega_b$, and obtain (here, for the sake of simplicity, we do not distinguish between $u$ and its traces on $\Sigma$ and ${\Sigma_{\mathrm p}}$):
		\begin{align*}
			&\frac d{dt}\left(\frac12\norm{\sqrt au}_{L^2(\Omega_\eps)}^2+\frac\eps2 \norm{\sqrt av}_{L^2(\Omega_b)}^2\right)+ \int_{\Oe}\mathbf B\nabla u\cdot\nabla u\dx+\eps  \int_{\Omega_b}\mathbf D\nabla v\cdot\nabla v\dx\\&=\frac12\int_{\Omega_\eps}\partial_ta \normmm{u}^2\dx + \frac\eps2\int_{\Omega_b}\partial_t {a}\normmm{v}^2\dx+\int_{\Omega_\eps}gu\dx+\eps\int_{\Omega_b}hv\dx+\eps\int_{{\Sigma_{\mathrm p}}}u\op \d\mathcal H^{N-1},
		\end{align*}
		and thus, since $a\in C^1_b(\R^N\times[0,\infty))$, after an application of the Cauchy-Schwarz and Young inequalities we end up with
		\begin{align*}
			&\frac d{dt}\left(\frac12\norm{\sqrt a u}_{L^2(\Omega_\eps)}^2+\frac\eps2 \norm{\sqrt av}_{L^2(\Omega_b)}^2\right)+ \frac12\int_{\Oe}\mathbf B\nabla u\cdot\nabla u\dx+\frac\eps2  \int_{\Omega_b}\mathbf D\nabla v\cdot\nabla v\dx\\&\leq C\left(\frac12\norm{\sqrt a u}_{L^2(\Omega_\eps)}^2+\frac\eps2 \norm{\sqrt av}_{L^2(\Omega_b)}^2\right)+C(\norm{g}_{L^2(\Omega_\eps)}^2+\eps\norm{h}_{L^2(\Omega_b)}^2+\eps\norm{\op}_{L^2({\Sigma_{\mathrm p}})}^2),
		\end{align*}
		\mich{where the boundary term has been estimated by means of \eqref{ctrltrace2}, which gives
		\[
		\eps\norm{u}_{L^2(\Sigma_{\mathrm p})}\norm{\op}_{L^2(\Sigma_{\mathrm p})}\leq C\eps^\frac32\norm{u}_{H^1(\Omega_\eps)}\norm{\op}_{L^2(\Sigma_{\mathrm p})}+C\eps^{\onehalf}\norm{u}_{L^2(\Omega_\eps)}\norm{\op}_{L^2(\Sigma_{\mathrm p})},
		\]
		the first contribution being absorbed, after Young's inequality, by the dissipative terms, and the second one by Gronwall's Lemma.}
		As a consequence, we immediately infer the first line of \eqref{est1} by Gronwall's Lemma. Concerning the time derivatives, we deduce by comparison the validity of the second line of \eqref{est1}, recalling that $a\in C^1_b(\R^N\times[0,\infty))$ and that it is uniformly bounded below by a positive constant.
		
		We now concentrate on obtaining the higher-order estimates, \mich{assuming $(u_0,v_0)\in H^1(\Omega_\eps)\times H^1(\Omega_b)$ with $\mathrm{Tr}_\Sigma u_0=\mathrm{Tr}_\Sigma v_0$ a.e. on $\Sigma$.}
        In particular, we set $\zeta=-\Div(\mathbf B\nabla u)$ in $\Omega_\eps$ \mich{and $\zeta=-\Div(\mathbf D\nabla v)$} in $\Omega_b$ in the weak formulation \eqref{weakform}, obtaining 
		\begin{align*}
			&\frac d{dt}\left(\frac12\int_{\Oe}a\nabla u\cdot \mathbf B\nabla u\dx+\frac\eps2\int_{\Omega_b}a\nabla v\cdot \mathbf D\nabla v\dx+\eps\int_{\Sigma_p}a\op u\d\mathcal H^{N-1}\right)\\&\quad +\int_{\Oe}\Div(\mathbf B\nabla u)^2\dx+\eps\int_{\Omega_b}\Div(\mathbf D\nabla v)^2\dx\\&
			=\frac12\int_{\Omega_\eps} \partial_t a {\nabla u}\cdot \mathbf B\nabla u\dx-\int_{\Omega_\eps}\nabla a \cdot \mathbf B\nabla u\,\partial_t u \dx +\frac\eps2\int_{\Omega_b} \partial_t a {\nabla v}\cdot \mathbf D\nabla v\dx-\eps\int_{\Omega_b}\nabla a \cdot \mathbf D\nabla v\,\partial_t v \dx \\&
			-\int_{\Omega_\eps}g\Div(\mathbf B\nabla u)\dx-\eps\int_{\Omega_b}h\Div(\mathbf D\nabla v)\dx+\eps\int_{{\Sigma_{\mathrm p}}}\pt a u\op\d\mathcal H^{N-1}\\&\quad +\frac12\int_{\Oe}a\nabla u\cdot \pt\mathbf B\nabla u\dx+\frac\eps2\int_{\Omega_b}a\nabla v\cdot \pt\mathbf D\nabla v\dx.
		\end{align*}
		Therefore, recalling the regularity of $a,\mathbf B,\mathbf D$, we infer, again after an application of the Cauchy-Schwarz and Young's inequalities,
		\begin{align}
			&\nonumber\frac d{dt}\left(\frac12\int_{\Oe}a\nabla u\cdot \mathbf B\nabla u\dx+\frac\eps2\int_{\Omega_b}a\nabla v\cdot \mathbf D\nabla v\dx+\eps\int_{{\Sigma_{\mathrm p}}} a u\op\d\mathcal H^{N-1}\right)\\&\nonumber+\int_{\Oe}\Div(\mathbf B\nabla u)^2\dx+\eps\int_{\Omega_b}\Div(\mathbf D\nabla v)^2\dx\\&\nonumber\leq C\left(\int_{\Oe}a\nabla u\cdot \mathbf B\nabla u\dx+\eps\int_{\Omega_b}a\nabla v\cdot \mathbf D\nabla v\dx\right)\\&\nonumber+\frac12\int_{\Oe}\Div(\mathbf B\nabla u)^2\dx+\frac\eps2\int_{\Omega_b}\Div(\mathbf D\nabla v)^2\dx\\&
			+ C(\norm{g}^2_{L^2(\Omega_\eps)}+\eps\norm{h}^2_{L^2(\Omega_b)}+\eps\norm{\op}_{L^2({\Sigma_{\mathrm p}})}^2)+\omega\norm{\partial_t u}^2_{L^2(\Omega_\eps)}+\omega\eps\norm{\partial_t v}^2_{L^2(\Omega_b)},
			\label{A0}
		\end{align}
		for some $\omega>0$ to be chosen later on, where we also used the inequality $\norm{u}_{L^2({\Sigma_{\mathrm p}})}\leq C\eps^{-\onehalf}\norm{u}_{H^1(\Omega_\eps)}$ from Lemma \ref{tracelemma1}. Then, choosing $\zeta$ so that $\zeta=\partial_t u$ in $\Omega_\eps$ and $\zeta=\partial_t v$ in $\Omega_b$ in \eqref{weakform}, we also get 
		\begin{align*}
			&\nonumber \frac d{dt}\left(\frac12\int_{\Oe}\nabla u\cdot \mathbf B\nabla u\dx+\frac\eps2\int_{\Omega_b}\nabla v\cdot \mathbf D\nabla v\dx+\eps\int_{{\Sigma_{\mathrm p}}} u\op\d\mathcal H^{N-1}\right)\\&
			+\norm{\sqrt a \partial_t u}_{L^2(\Omega_\eps)}^2+\eps\norm{\sqrt a \partial_t v}_{L^2(\Omega_b)}^2\\&=\frac12\int_{\Omega_\eps}\partial_t \mathbf B\nabla u\cdot\nabla u\dx +\frac\eps2\int_{\Omega_b}\partial_t \mathbf D\nabla v\cdot\nabla v\dx\\&
			+\int_{\Omega_\eps}g\partial_t u\dx+\eps\int_{\Omega_b}h\partial_t v\dx,
		\end{align*}
		which, again by standard inequalities, recalling the properties of $a,\mathbf B,\mathbf D$, gives 
		\begin{align}
			&\nonumber \frac d{dt}\left(\frac12\int_{\Oe}\nabla u\cdot \mathbf B\nabla u\dx+\frac\eps2\int_{\Omega_b}\nabla v\cdot \mathbf D\nabla v\dx+\eps\int_{{\Sigma_{\mathrm p}}} u\op\d\mathcal H^{N-1}\right)\\&\nonumber+\norm{\sqrt a \partial_t u}_{L^2(\Omega_\eps)}^2+\eps\norm{\sqrt  a \partial_t v}_{L^2(\Omega_b)}^2\\&\nonumber\leq C\left(\frac12\int_{\Oe}\nabla u\cdot \mathbf B\nabla u\dx+\frac\eps2\int_{\Omega_b}\nabla v\cdot \mathbf D\nabla v\dx\right)\\&\label{A1}
			+ C(\norm{g}^2_{L^2(\Omega_\eps)}+\eps\norm{h}^2_{L^2(\Omega_b)})+\frac12\norm{\sqrt a\partial_t u}^2_{L^2(\Omega_\eps)}+\frac\eps2\norm{\sqrt a\partial_t v}^2_{L^2(\Omega_b)},
		\end{align}
		entailing of course 
		\begin{align}
			&\nonumber \frac d{dt}\left(\frac12\int_{\Oe}\nabla u\cdot \mathbf B\nabla u\dx+\frac\eps2\int_{\Omega_b}\nabla v\cdot \mathbf D\nabla v\dx+\eps\int_{{\Sigma_{\mathrm p}}} u\op\d\mathcal H^{N-1}\right)\\&\nonumber+\frac12\norm{\sqrt a \partial_t u}_{L^2(\Omega_\eps)}^2+\frac\eps2\norm{\sqrt a \partial_t v}_{L^2(\Omega_b)}^2\\&\nonumber\leq C\left(\frac12\int_{\Oe}\nabla u\cdot \mathbf B\nabla u\dx+\frac\eps2\int_{\Omega_b}\nabla v\cdot \mathbf D\nabla v\dx\right)\\&\label{A2}
			+ C(\norm{g}^2_{L^2(\Omega_\eps)}+\eps\norm{h}^2_{L^2(\Omega_b)}),
		\end{align}
		\mich{so that, by summing \eqref{A0} with \eqref{A2} and choosing $\omega=\frac{c_0}{4}$, we finally deduce 
	\begin{align}
		&\frac d{dt}\left(\frac12\int_{\Oe}(1+a)\nabla u\cdot \mathbf B\nabla u\dx+\frac\eps2\int_{\Omega_b}(1+a)\nabla v\cdot \mathbf D\nabla v\dx+\eps\int_{{\Sigma_{\mathrm p}}} (1+a) u\op\d\mathcal H^{N-1}\right)\nonumber\\&
		+\int_{\Oe}\Div(\mathbf B\nabla u)^2\dx+\eps\int_{\Omega_b}\Div(\mathbf D\nabla v)^2\dx\nonumber+\frac{c_0}{4}\norm{\partial_t u}^2_{L^2(\Omega_\eps)}+\frac{c_0\eps}{4}\norm{\partial_t v}^2_{L^2(\Omega_b)}\\&\leq C\left(\frac12\int_{\Oe}(1+a)\nabla u\cdot \mathbf B\nabla u\dx+\frac\eps2\int_{\Omega_b}(1+a)\nabla v\cdot \mathbf D\nabla v\dx\right)\nonumber\\&
		+ C(\norm{g}^2_{L^2(\Omega_\eps)}+\eps\norm{h}^2_{L^2(\Omega_b)}+\eps\norm{\op}_{L^2({\Sigma_{\mathrm p}})}^2),\label{ffA}
\end{align}}
entailing \eqref{est2} after an application of Gronwall's Lemma. 
 Note that this is possible since, by Lemma \ref{tracelemma1},
\begin{align*}	&\normmm{\eps\int_{{\Sigma_{\mathrm p}}}(1+a)u\op\d\mathcal H^{N-1}}\leq \eps (1+\norm{a}_{C^0_b(\R^N)})\norm{u}_{L^2({\Sigma_{\mathrm p}})}\norm{\op}_{L^2({\Sigma_{\mathrm p}})}\\&\leq C\eps^\onehalf (1+\norm{a}_{C^0_b(\R^N)})(\norm{u}_{L^2(\Oe)}+\norm{\nabla u}_{L^2(\Oe)})\norm{\op}_{L^2({\Sigma_{\mathrm p}})}\\&
	\leq C(T)\eps(1+\norm{\op}^2_{L^2(\Sigma_p)})+\frac14 \int_{\Oe}(1+a)\nabla u\cdot \mathbf B\nabla u\dx,
\end{align*}
where we used that $u\in L^\infty(0,T;L^2(\Oe))$ by \eqref{est1}, so that the argument under the time derivative in \eqref{ffA} is bounded below by a (negative) constant which scales like $\eps$.
The proof of the theorem is concluded, since uniqueness is a trivial consequence of the linearity of the problem.

		\subsubsection{Proof of Theorem \ref{thmthin}}

		Again, to prove this theorem we should consider a suitable approximation scheme, but since this is a standard argument (for instance by means of a time discretization scheme) we prefer to stick to the \textit{a priori} energy estimates. In particular, by choosing $\zeta=v$ in the weak formulation \eqref{thinf} (recall that $u=\mathrm{Tr}_\Sigma v$ on $\Sigma$), we infer 
		\begin{align*}
			&\frac d{dt}\left(\frac12\norm{\sqrt a u}_{L^2(\Sigma)}^2+\frac12 \norm{\sqrt a v}_{L^2(\Omega_b)}^2\right)+ \int_{\Sigma}\mathbf B_{|\Sigma}\nabla_\Sigma u\cdot\nabla_\Sigma u\d\mathcal H^{N-1}+  \int_{\Omega_b}\mathbf D\nabla v\cdot\nabla v\dx\\&=\frac12\int_{\Sigma}\partial_ta \normmm{u}^2\d\mathcal H^{N-1} + \frac 12\int_{\Omega_b}\partial_ta\normmm{v}^2\dx+\int_{\Sigma}gu\d\mathcal H^{N-1}+\int_{\Omega_b}hv\dx+\int_{\Sigma}u\op \d\mathcal H^{N-1},
		\end{align*}
		and thus, thanks to the regularity of $a,\mathbf B,\mathbf D$, after an application of the Cauchy-Schwarz and Young inequalities we conclude
		\begin{align*}
			&\frac d{dt}\left(\frac12\norm{\sqrt a u}_{L^2(\Sigma)}^2+\frac12 \norm{\sqrt a v}_{L^2(\Omega_b)}^2\right)+ \frac12\int_{\Sigma}\mathbf B_{|\Sigma}\nabla_\Sigma u\cdot\nabla_\Sigma u\d\mathcal H^{N-1}+ \frac12 \int_{\Omega_b}\mathbf D\nabla v\cdot\nabla v\dx\\&\leq C\left(\frac12\norm{\sqrt a u}_{L^2(\Sigma)}^2+\frac12 \norm{\sqrt a v}_{L^2(\Omega_b)}^2\right)+C(\norm{g}_{L^2(\Sigma)}^2+\norm{h}_{L^2(\Omega_b)}^2+\norm{\op}_{L^2(\Sigma)}^2),
		\end{align*}
		which gives \eqref{est3} after an application of Gronwall's \mich{Lemma} and arguing by comparison for the time derivatives.
		
		\mich{As for the higher-order estimates \eqref{extracondition}, the proof is essentially the same as the one of \eqref{est2} (see Section \ref{sec:proofnonthin}).} We only point out the main differences, especially concerning the $H^2$-regularity, which is available here. By choosing $\zeta=\partial_t v$ in the weak formulation \eqref{thinf}, we get (recall again that $u=\mathrm{Tr}_\Sigma v$ on $\Sigma$)
		\begin{align*}
			&\nonumber \frac d{dt}\left(\frac12\int_{\Sigma}\nabla_\Sigma u\cdot \mathbf B_{|\Sigma}\nabla_\Sigma u\d\mathcal H^{N-1}+\frac12\int_{\Omega_b}\nabla v\cdot \mathbf D\nabla v\dx\right)\\&
			+\norm{\sqrt a \partial_t u}_{L^2(\Sigma)}^2+\norm{\sqrt a \partial_t v}_{L^2(\Omega_b)}^2\\&=\frac12\int_{\Sigma}\partial_t \mathbf B_{|\Sigma}\nabla_\Sigma u\cdot\nabla_\Sigma u\d\mathcal H^{N-1} +\frac12\int_{\Omega_b}\partial_t \mathbf D\nabla v\cdot\nabla v\dx\\&
			+\int_{\Sigma}(g+\op)\partial_t u\d\mathcal H^{N-1}+\int_{\Omega_b}h\partial_t v\dx,
		\end{align*}
		which immediately gives, by standard inequalities and using the regularity of $a,\mathbf B,\mathbf D$,
		\begin{align*}
			&\nonumber \frac d{dt}\left(\frac12\int_{\Sigma}\nabla_\Sigma u\cdot \mathbf B_{|\Sigma}\nabla_\Sigma u\d\mathcal H^{N-1}+\frac12\int_{\Omega_b}\nabla v\cdot \mathbf D\nabla v\dx\right)\\&
			+\frac12\norm{\sqrt a \partial_t u}_{L^2(\Sigma)}^2+\frac12\norm{\sqrt a \partial_t v}_{L^2(\Omega_b)}^2\\&\leq C\left(\frac12\int_{\Sigma}\nabla_\Sigma u\cdot \mathbf B_{|\Sigma}\nabla_\Sigma u\d\mathcal H^{N-1}+\frac12\int_{\Omega_b}\nabla v\cdot \mathbf D\nabla v\dx\right)
			+C\bigl(\norm{g+\op}_{L^2(\Sigma)}^2+\norm{h}^2_{L^2(\Omega_b)}\bigr),
		\end{align*}
		from which, after an application of Gronwall's Lemma, we infer
		\begin{align}
			&\nonumber\norm{u}_{L^\infty(0,T;H^1(\Sigma))}+ \norm{v}_{L^\infty(0,T;H^1(\Omega_b))}
			+\norm{\partial_t u}_{L^2(0,T;L^2(\Sigma))}+\norm{\partial_tv}_{L^2(0,T;L^2(\Omega_b))}\nonumber\\&\leq C(T)(\norm{v_0}_{H^1(\Omega_b)}+\norm{u_0}_{H^1(\Sigma)})+C(T)(\norm{g}_{L^2(\Sigma)}+\norm{h}_{L^2(\Omega_b)}+\norm{\op}_{L^2(\Sigma)}).\label{est2b}
		\end{align}
		Notice that we can now obtain the $H^2$-regularity of the solution by elliptic regularity arguments. \mich{Indeed, for almost any $t\in(0,T)$ we can restate the problem as the elliptic system
		\begin{align}
			\label{thinfilm1}
			\begin{cases}
				-\Div(\mathbf D(\mathbf x,t)\nabla v)=h_1:=h- a(\mathbf x,t)\partial_t v,\quad&\text{ in }\Omega_b,\\[3pt]
				\mathbf D(\mathbf x,t)\nabla v\cdot{\bn_{\Sigma}}-\Div_\Sigma( \bP\mathbf B_{|\Sigma}(\mathbf x,t)\bP\nabla_\Sigma u)+u= h_2:=\op+g-a_{|\Sigma}(\mathbf x,t)\partial_t u+u,\quad&\text{ on }\Sigma,
			\end{cases}
		\end{align}}
		where $u=\mathrm{Tr}_\Sigma v$ and where, by \eqref{est2b}, $h_1\in L^2(0,T;L^2(\Omega_b))$ and $h_2\in L^2(0,T;L^2(\Sigma))$. As a consequence, thanks to the regularity of $\mathbf B$ and $\mathbf D$, by well-known elliptic arguments (see, e.g., the proof of \cite[Lemma A.1]{MZ}, see also \cite{GGP}, where similar arguments are repeatedly used), we infer 
		\begin{align*}
			\norm{v}_{L^2(0,T;H^2(\Omega_b))}+\norm{u}_{L^2(0,T;H^2(\Sigma))}\leq C(T)(\norm{h_1}_{L^2(0,T;L^2(\Omega_b))}+\norm{h_2}_{L^2(0,T;L^2(\Sigma))}).
		\end{align*}
		\mich{Together with \eqref{est2b}, this concludes the proof of estimate \eqref{extracondition}. Since the system is linear, uniqueness of solutions follows.}
		
		\subsubsection{Proof of Theorem \ref{thmconvergence}}
        \label{Sec:convergence}
        
		We first prove the following lemma.
		\begin{lemma}
			\label{recastlemma}
			Under the assumptions of Theorem \ref{thmconvergence}, we define
			\begin{align*}
				w_\eps=\mM_\eps(u_\eps),\quad w_\eps(0)=\mM_\eps(u_{0,\eps})=\mM_\eps(v_{0,\eps}),\quad\text{ on }\Sigma.
			\end{align*}
			Then, for any $T>0$, it holds $w_\eps \in L^2(0,T;\mich{H^1(\Sigma)})\cap H^1(0,T;L^2(\Sigma))$ and
			\mich{\begin{align}
				\nonumber&\int_0^T\int_\Sigma \mM_\eps(a)\partial_t w_\eps \zeta\d\mathcal H^{N-1}\dt +\int_0^T\int_{\Omega_b} a(\mathbf x,t)\partial_t v_\eps  \zeta\dx\dt\\&\quad +\int_0^T\int_{\Sigma} \mathbf B_{|\Sigma}(\mathbf x,t)\nabla_\Sigma w_\eps\cdot \nabla_\Sigma \zeta\d\mathcal H^{N-1}\dt+\int_0^T\int_{\Omega_b} \mathbf D(\mathbf x,t)\nabla v_\eps\cdot \nabla \zeta\dx\dt\nonumber\\&\quad -\int_0^T\int_{\Sigma}\mM_\eps( g)\zeta\d \mathcal H^{N-1}\dt-\int_0^T\int_{\Omega_b}h\zeta\dx\dt-\int_0^T\int_{\Sigma}\op\zeta\d\mathcal H^{N-1}\dt=\mathcal R_\eps(u_\eps,w_\eps,\zeta,T),\label{recast}
			\end{align}}
			for any $\zeta\in L^2(0,T;H^1(\Omega_b))\cap L^2(0,T;H^1(\Sigma))$, where we do not distinguish between $\zeta$ and its trace over $\Sigma$. Here $\mathcal R_\eps(u_\eps,w_\eps,\zeta,T)$ is linear in $\zeta$. Moreover, if $\bP(\mathbf x)\mathbf B(\mathbf x,\cdot)=\mathbf B(\mathbf x,\cdot)\bP(\mathbf x)$ for any $\mathbf x\in \Sigma$, then it holds 
			\begin{align}
				\normmm{\mathcal R_\eps(u_\eps,w_\eps,\zeta,T)}\leq C(T)\eps\norm{\zeta}_{L^2(0,T;H^1(\Sigma))},\label{estR}
			\end{align}
			for any $\zeta\in L^2(0,T;H^1(\Omega_b))\cap L^2(0,T;H^1(\Sigma))$.  Analogously, for general matrices $\mathbf B$, if \eqref{ctrl1a} holds, then \eqref{estR} is true as well.
		\end{lemma}        
		
		\begin{proof}
			For a fixed $\eps\in(0,\delta)$ we consider $(u_\eps\mich{,v_\eps})$ as the unique solution to \eqref{weakform} given by Theorem \ref{thmnotthin}. Thanks to assumptions \eqref{initial} and \eqref{f1cond}, we immediately infer from \eqref{est2} that
			\begin{align}
				&\nonumber\norm{u_\eps}_{L^\infty(0,T;H^1(\Omega_\eps))}+\sqrt\eps \norm{v_\eps}_{L^\infty(0,T;H^1(\Omega_b))}
				+\norm{\partial_tu_\eps}_{L^2(0,T;L^2(\Omega_\eps))}+\sqrt\eps \norm{\partial_tv_\eps}_{L^2(0,T;L^2(\Omega_b))}\nonumber\\&\leq C(T)(\sqrt\eps(1+\norm{\op}_{L^2(\Sigma_p)})+\sqrt\eps\norm{ v_0 }_{H^1(\Omega_b)}+\norm{u_{0,\eps}}_{H^1(\Omega_\eps)})\nonumber\\&\quad +C(T)(\norm{g}_{L^2(\Omega_\eps)}+\sqrt\eps\norm{h}_{L^2(\Omega_b)}+\sqrt\eps\norm{\op}_{L^2({\Sigma_{\mathrm p}})})\nonumber\\&
				\leq C(T)\sqrt \eps,\label{est2f}
			\end{align}
			for any $T>0$. Indeed, note that, by the change of variables \eqref{changesurf},
			\begin{align*}
								\norm{\op}_{L^2({\Sigma_{\mathrm p}})}^2=\int_{\Sigma} \bigl(\op^\sharp(y,\eps)\bigr)^2 J(y,\eps)\d \mathcal H^{N-1}(y)\leq C_J\norm{\mich{\op^\sharp(\cdot,\eps)}}_{L^2(\Sigma)}^2\mich{\leq C},
			\end{align*}
            \mich{uniformly in $\eps$, where we used \eqref{controlJ} and the fact that, by assumption, $\op\in C^{0}(\overline\Omega_\delta)$ and $\Omega_\eps\subset \Omega_\delta$ for any $\eps\in(0,\delta)$, so that $\op^\sharp(\cdot,\eps)$ is bounded in $L^\infty(\Sigma)$ uniformly with respect to $\eps$.}
			
			As a consequence, by means of \eqref{P1} and Lemma \ref{timederivative}, we infer that $w_\eps=\mM_\eps(u_\eps)\in L^2(0,T;H^1(\Sigma))\cap H^1(0,T;L^2(\Sigma))$.
			
			Then, given $\zeta\in L^2(0,T;H^1(\Omega_b))\cap L^2(0,T;H^1(\Sigma))$, we construct the test function 
			\begin{align}
				\widehat \zeta=\begin{cases}
					\widetilde\zeta,\quad\text{ in }\Oe,\\
					\zeta,\quad \text{ in }\Omega_b.
				\end{cases}
				\label{test}
			\end{align}
			To be precise, we should write $\widetilde\zeta=\widetilde{\mathrm{Tr}_{\Sigma}\zeta}$ in $\Omega_\eps$, but as above we do not distinguish between $\zeta$ and its trace over $\Sigma$. \mich{We can now use $\widehat\zeta$ as a test function in \eqref{weakform} and, after dividing by $\eps$, recast it as in \eqref{recast}, upon setting }
			\begin{align*}
				\mathcal R_\eps(u_\eps,w_\eps,\zeta,T):=\sum_{i=1}^4K_i,
			\end{align*}
			where 
			\begin{align*}
				&K_1:=\int_0^T\int_\Sigma \mM_\eps(a)\partial_t w_\eps \zeta\mich{\d\mathcal H^{N-1}\dt}-\frac1\eps\int_0^T\int_{\Oe} a\partial_t u_\eps \widetilde\zeta\dx\dt,\\&
				\mich{K_2:=\int_0^T\int_{\Sigma} \mathbf B_{|\Sigma}\nabla_\Sigma w_\eps\cdot \nabla_\Sigma \zeta{\d\mathcal H^{N-1}}\dt-\frac1\eps\int_0^T\int_{\Oe}\mathbf B\nabla u_\eps\cdot \nabla \widetilde\zeta\dx\dt,}\\&
				K_3:=-\int_0^T\int_{\Sigma}\mM_\eps( g)\zeta\d \mathcal H^{N-1}\dt+\frac1\eps\int_0^T\int_{\Omega_\eps} g\widetilde\zeta\dx\dt,\\&
				K_4:=-\int_0^T\int_{\Sigma}\op_{|\Sigma}\zeta\d \mathcal H^{N-1}\dt+\int_0^T\int_{{\Sigma_{\mathrm p}}}\op \widetilde\zeta\d \mathcal H^{N-1}\dt.
			\end{align*}
			Now, using \eqref{timecontrol} and \eqref{est2f}, we immediately infer, by the regularity of $a$, that, for any $T>0$,
			\begin{align*}
				\normmm{K_1}\leq C\norm{a}_{C^{0,1}(\R^N\times[0,T])}\eps^\frac12\norm{u_\eps}_{H^1(0,T;L^2(\Oe))}\norm{\zeta}_{L^2(0,T;L^2(\Sigma))}\leq C(T)\eps\norm{\zeta}_{L^2(0,T;L^2(\Sigma))},
			\end{align*}
			since $\norm{a}_{C^{0,1}(\R^N\times[0,T])}\leq C\norm{a}_{C^{1}_b(\R^N\times[0,T])}$.
			Moreover, if $\bP \mathbf B=\mathbf B\bP$ on $\Sigma$, then from \eqref{controlb2} and \eqref{est2f} we have   
			\begin{align*}
				&\normmm{K_2}\leq C\eps^\frac12\norm{\mathbf B}_{C^{0,1}(\R^N\times[0,T])}\norm{u_\eps}_{L^2(0,T;H^1(\Oe))}\norm{\zeta}_{L^2(0,T;H^1(\Sigma))}\\&
				\leq C\eps\norm{\zeta}_{L^2(0,T;H^1(\Sigma))}.
			\end{align*}
			On the other hand, for a general matrix $\mathbf B$, under assumption \eqref{ctrl1a} we obtain from \eqref{controlb} that
			\begin{align*}
				\normmm{K_2}&\leq C\eps^\frac12\norm{\mathbf B}_{C^{0,1}(\R^N\times[0,T])}\norm{u_\eps}_{L^2(0,T;H^1(\Omega_\eps))}\norm{\nabla_\Sigma \zeta}_{L^2(0,T;L^2(\Sigma))}\\&
				\quad +\nonumber C\eps^\frac12\norm{\mathbf B}_{C^0_b(\R^N\times[0,T])}\norm{u_\eps}_{L^2(0,T;H^2(\Omega_\eps))}\norm{\nabla_\Sigma \zeta}_{L^2(0,T;L^2(\Sigma))}\\&\quad+C\norm{\mathbf B}_{C^0_b(\R^N\times[0,T])}\norm{\partial_{\bn_{\Sigma}} u_\eps}_{L^2(0,T;L^2(\Sigma))}\norm{\nabla_\Sigma \zeta}_{L^2(0,T;L^2(\Sigma))}\\&
				\leq C\eps\norm{\zeta}_{L^2(0,T;H^1(\Sigma))},
			\end{align*}
			\mich{where we have used the fact that $\mathbf B \nabla u_\eps \cdot {\bn_{\Sigma}}=\eps \mathbf D\nabla v_\eps\cdot {\bn_{\Sigma}}$ a.e. on $\Sigma\times(0,T)$, and thus, by \eqref{ctrl1a},}
			\begin{align*}
				\norm{\partial_{\bn_{\Sigma}} u_\eps}_{L^2(0,T;L^2(\Sigma))}\leq C\eps\norm{\partial_{\bn_{\Sigma}} v_\eps}_{L^2(0,T;L^2(\Sigma))}\leq C\eps\norm{v_\eps}_{L^2(0,T;H^2(\Omega_b))}.
			\end{align*}
			
			Then, exploiting \eqref{basic}, we infer 
			\begin{align*}
				\int_0^T\int_{\Sigma}\mM_\eps( g)\zeta\d \mathcal H^{N-1}\dt=\frac1\eps\int_0^T\int_{\Omega_\eps} g\widetilde\zeta\dx\dt,
			\end{align*}
			and thus $K_3=0$. In conclusion, concerning $K_4$, we have, by the change of variables \eqref{changesurf},
			\begin{align}
				\nonumber\normmm{K_4}\leq&\normmm{\int_0^T\int_{\Sigma}\op_{|\Sigma}\zeta\bigl(1-\mich{J(y,\eps)}\bigr)\d \mathcal H^{N-1}\dt}+\normmm{\int_0^T\int_{\Sigma}\bigl(\op_{|\Sigma}-\op^\sharp(y,\eps)\bigr)\zeta J(y,\eps)\d \mathcal H^{N-1}\dt}\\&
				\leq CT^\frac12\eps{\norm{\op}_{L^2(\Sigma)}\norm{\zeta}_{L^2(0,T;L^2(\Sigma))}+CT^\frac12\eps}\mich{\norm{\op}_{C^{0,1}(\overline{\Omega}_\delta)}}\norm{\zeta}_{L^2(0,T;L^2(\Sigma))},\label{K4}
			\end{align}
			where we have used \eqref{controlJ2} \mich{and the fact that, $\op$ being Lipschitz continuous,}
			\begin{align*}
				\norm{\op_{|\Sigma}-\op^\sharp(\cdot,\eps)}_{L^2(\Sigma)}\leq C\eps\norm{\op}_{C^{0,1}(\overline{\Omega}_\delta)}.
			\end{align*}
			Putting together the estimates for $K_1,K_2,K_3,K_4$, we can finally conclude the proof of \eqref{estR}.
		\end{proof}
		In order to continue the proof of Theorem \ref{thmconvergence}, we need uniform estimates on the sequence $\{w_\eps\}_\eps$. \mich{To this aim, the test function in \eqref{recast} has to be chosen carefully. Indeed, it is not true in general that $w_\eps=\mathrm{Tr}_\Sigma v_\eps$, }and thus we cannot simply choose $\zeta=v_\eps$ in \eqref{recast}. The idea is then to exploit the surjectivity of the trace operator over $\Sigma$ from $H^1(\Omega_b)$ onto $H^\frac12(\Sigma)$: given $w_\eps(t)\in H^1(\Sigma)\hookrightarrow H^\frac12(\Sigma)$ for almost any $t>0$, there exists $\beta_\eps(t)\in H^1(\Omega_b)$ such that $\mathrm{Tr}_\Sigma \mich{\beta_\eps}=w_\eps$ a.e. on $\Sigma$ and, without loss of generality since an infimum over such functions will be taken, $\norm{\beta_\eps(t)}_{H^1(\Omega_b)}\leq \norm{w_\eps(t)}_{H^\frac12(\Sigma)}+1$. Therefore, we can choose $\zeta=\beta_\eps\in L^2(0,T;H^1(\Omega_b))\cap L^2(0,T;H^1(\Sigma))$ in \eqref{recast}, to obtain, for any $0<t\leq T$, 
                \begin{align}\nonumber&\frac12\norm{\sqrt{\mM_\eps(a(t))}w_\eps(t)}^2_{L^2(\Sigma)}+\int_0^t\int_\Sigma \mich{\mathbf B_{|\Sigma}\nabla_\Sigma w_\eps\cdot\nabla_\Sigma w_\eps\d\mathcal H^{N-1}\d s}\\&\nonumber=\frac12\norm{\sqrt{\mM_\eps(a(0))}w_\eps(0)}^2_{L^2(\Sigma)}+\frac12\int_0^t\int_\Sigma \partial_t\mM_\eps(a)\normmm{w_\eps}^2\d\mathcal H^{N-1} \mich{\d s}\nonumber\\&\nonumber-\int_0^t\int_{\Omega_b}a\partial_t v_\eps \beta_\eps\dx\d s-\int_0^t\int_{\Omega_b} \mathbf D{\nabla v_\eps}\cdot\nabla \beta_\eps\dx\d s\\&\quad +\int_0^t\int_{\Omega_b}h\beta_\eps\dx\d s+\int_0^t\int_{\Sigma}(\mM_\eps(g)+\op)w_\eps\d \mathcal H^{N-1} \d s
			+\mathcal R_\eps(u_\eps,w_\eps,\beta_\eps,t).\label{ff}
		\end{align}
		First, concerning $w_\eps(0)$, we have, by the regularity of $a$ and assumption \eqref{initial}, making use of \eqref{ctrl1},
		\begin{align*}
			\frac12\norm{\sqrt{\mM_\eps(a(0))}w_\eps(0)}^2_{L^2(\Sigma)}\leq C \normm{w_\eps(0)}_{L^2(\Sigma)}^2=C\norm{\mM_\eps({u}_{0,\eps})}^2_{L^2(\Sigma)}\leq C\eps^{-1}\norm{{u}_{0,\eps}}_{L^2(\Omega_\eps)}^2\leq C,
		\end{align*}
		for some $C>0$ independent of $\eps>0$.
		Then, recalling again the regularity of $a$, we have
		\begin{align*}
			\normmm{\frac12\int_0^t\int_\Sigma \partial_t\mM_\eps(a)\normmm{w_\eps}^2\d\mathcal H^{N-1}\mich{\d s}}\leq \mich{C}\norm{a}_{C^1_b({\R^N\times[0,T]})}\norm{w_\eps}^2_{L^2(0,t;L^2(\Sigma))}.
		\end{align*}
		Also, by the Cauchy-Schwarz and Young inequalities and the regularity of $a,\mathbf D$, we deduce
		\begin{align*}
			&\normmm{\int_0^t\int_{\Omega_b}a\partial_t v_\eps \beta_\eps\dx\d s+\int_0^t\int_{\Omega_b} \mathbf D{\nabla v_\eps}\cdot\nabla \beta_\eps\dx\d s+\int_0^t\int_{\Omega_b}h\beta_\eps\dx\d s+\int_0^t\int_{\Sigma}(\mM_\eps(g)+\op)w_\eps\d\mathcal H^{N-1}\d s}\\&
			\leq \norm{a}_{C^0_b(\R^N\times[0,T])}\norm{\partial_t v_\eps}_{L^2(0,T;L^2(\Omega_b))}\norm{\beta_\eps}_{L^2(0,t;L^2(\Omega_b))}\\&\quad +\norm{ \mM_\eps(g)+\op}_{L^2(0,T;L^2(\Sigma))}\norm{w_\eps}_{L^2(0,t;L^2(\Sigma))}\\&\quad +\norm{\mathbf D}_{C^0_b(\R^N\times[0,T])}\norm{\nabla v_\eps}_{L^2(0,T;L^2(\Omega_b))}\norm{\nabla \beta_\eps}_{L^2(0,t;L^2(\Omega_b))}+\norm{ h}_{L^2(0,T;L^2(\Omega_b))}\norm{\beta_\eps}_{L^2(0,t;L^2(\Omega_b))}\\&
			\leq C(T)\bigl(\norm{w_\eps}_{L^2(0,t;L^2(\Sigma))}+\norm{\beta_\eps}_{L^2(0,t;H^1(\Omega_b))}\bigr)\leq C(T)+\norm{w_\eps}_{L^2(0,t;L^2(\Sigma))}^2+\norm{\beta_\eps}_{L^2(0,t;H^1(\Omega_b))}^2,
		\end{align*}
		where in the last estimate we used the uniform controls \eqref{est2f}. Here we also used \eqref{ctrl1}, so that 
		\begin{align}
			\norm{\mM_\eps (g)}_{L^2(0,T;L^2(\Sigma))}\leq \mich{C}\eps^{-\frac12}T^\frac12\norm{g}_{L^2(\Omega_\eps)}\leq CT^\frac12,\label{ct}
		\end{align}
		by assumption \eqref{f1cond}.
		\mich{Notice that this estimate shows the role of assumptions \eqref{initial} and \eqref{f1cond}, which are precisely what makes the above bounds uniform in $\eps>0$. }
		
		In conclusion, recalling \eqref{estR}, we have, by Young's inequality,
		\begin{align*}
			\normmm{\mathcal R_\eps(u_\eps,w_\eps,\beta_\eps,t)}\leq C(T)\eps\norm{w_\eps}_{L^2(0,t;H^1(\Sigma))}\leq C(T)\eps^2+\frac{c_0}{2C_J}\norm{\nabla w_\eps}^2_{L^2(0,t;L^2(\Sigma))}+\frac12\norm{w_\eps}^2_{L^2(0,t;L^2(\Sigma))}.
		\end{align*}
				Plugging all these estimates in \eqref{ff}, recalling that $\xi\cdot \mathbf B\xi\geq c_0\normmm{\xi}^2$, and $J\geq \frac1{C_J}$, by \eqref{controlJ}, we end up with 
		\begin{align}
			\nonumber&\frac12\norm{\sqrt{\mM_\eps(a(t))}w_\eps(t)}^2_{L^2(\Sigma)}+\frac{c_0}{2C_J}\int_0^t\int_\Sigma \normmm{\nabla_\Sigma w_\eps}^2\d\mathcal H^{N-1}\d s\\&
			\leq C(T)(1+\eps^2)+C(T)\norm{w_\eps}^2_{L^2(0,t;L^2(\Sigma))}+\norm{\beta_\eps}_{L^2(0,t;H^1(\Omega_b))}^2,\label{ff2}
		\end{align}
		for any $t\in(0,T]$ and for any $T>0$, but since this inequality holds for any $\beta_\eps$ such that $\mathrm{Tr}_\Sigma \beta_\eps=w_\eps$, we can take the infimum over all these functions and,  since by definition
		\begin{align}
			\label{H12}
			\norm{w_\eps}_{H^\frac12(\Sigma)}:=\inf_{\beta_\eps\in H^1(\Omega_b):\ \mathrm{Tr}_\Sigma \beta_{\eps}=w_\eps }\norm{\beta_\eps}_{H^1(\Omega_b)},
		\end{align}
		we obtain
		\begin{align}
			\nonumber&\frac12\norm{\sqrt{\mM_\eps(a(t))}w_\eps(t)}^2_{L^2(\Sigma)}+\frac{c_0}{2C_J}\int_0^t\int_\Sigma \normmm{\nabla_\Sigma w_\eps}^2\d \mathcal H^{N-1}\d s\\&
			\leq C(T)(1+\eps^2)+C(T)\norm{w_\eps}_{L^2(0,t;H^\frac12(\Sigma))}^2\nonumber\\&\leq C(T)(1+\eps^2)+C(T)\norm{w_\eps}^2_{L^2(0,t;L^2(\Sigma))}+C\norm{w_\eps}_{L^2(0,t;L^2(\Sigma))}\norm{w_\eps}_{L^2(0,t;H^1(\Sigma))}\nonumber\\&
			\leq C(T)(1+\eps^2)+C(T)\norm{w_\eps}_{L^2(0,t;L^2(\Sigma))}^2+\frac{c_0}{4C_J}\norm{\nabla w_\eps}_{L^2(0,t;L^2(\Sigma))}^2,\label{ff3}
		\end{align}
		where in the last inequality we used the well-known interpolation $\norm{\cdot}_{H^\frac12(\Sigma)}\leq C\norm{\cdot}_{L^2(\Sigma)}^\frac12\norm{\cdot}_{H^1(\Sigma)}^\frac12$, together with Young's inequality. Since this holds for any $t\in(0,T]$, by the integral Gronwall Lemma we infer that
		\begin{align}
			\norm{w_\eps}_{L^\infty(0,T;L^2(\Sigma))}+\norm{w_\eps}_{L^2(0,T;H^1(\Sigma))}\leq C(T),
			\label{regveps}
		\end{align}
		for any $T>0$, uniformly in $\eps>0$.
		
		Concerning the time derivative, given $\xi\in L^2(0,T;H^1(\Sigma))$, let $\beta\in L^2(0,T;H^1(\Omega_b))$ be such that $\mathrm{Tr}_\Sigma \beta=\xi$, and set $\zeta=\beta$ in \eqref{recast}. After standard inequalities, together with the use of \eqref{estR}, we infer
		\begin{align*}
			&\normmm{\int_0^T \int_\Sigma \mM_\eps(a)\partial_tw_\eps \xi\d\mathcal H^{N-1}\dt}\\&\leq \left(\norm{a}_{C^0_b(\R^N\times[0,T])}\norm{\partial_t v_\eps}_{L^2(0,T;L^2(\Omega_b))}+\norm{\mathbf D}_{C^0_b(\R^N\times[0,T])}\norm{\nabla v_\eps}_{L^2(0,T;L^2(\Omega_b))}+\norm{h}_{L^2(0,T;L^2(\Omega_b))}\right)\\&\quad \times\norm{\beta}_{L^2(0,T;H^1(\Omega_b))}\\&
			+\left(\norm{\mathbf B}_{L^\infty(\Sigma\times(0,T))}\norm{\nabla_\Sigma w_\eps}_{L^2(0,T;L^2(\Sigma))}+\norm{\mM_\eps (g)+\op}_{L^2(0,T;L^2(\Sigma))}+C(T)\eps\right)\norm{\xi}_{L^2(0,T;H^1(\Sigma))}.
		\end{align*}
		\mich{Recalling \eqref{ct}, so that $\norm{\mM_\eps (g)}_{L^2(0,T;L^2(\Sigma))}\leq CT^\frac12$, and using the estimates \eqref{est2f} together with \eqref{regveps},} we infer 
		\begin{align*}
			\normmm{\int_0^T \int_\Sigma \mM_\eps(a)\partial_tw_\eps \xi\d\mathcal H^{N-1}\dt}\leq C(T)\bigl(\norm{\xi}_{L^2(0,T;H^1(\Sigma))}+\norm{\beta}_{L^2(0,T;H^1(\Omega_b))}\bigr),
		\end{align*}
		and, since in this estimate $\beta$ is arbitrary, recalling \eqref{H12} we deduce 
		\begin{align*}
			\normmm{\int_0^T \int_\Sigma \mM_\eps(a)\partial_tw_\eps \xi\d\mathcal H^{N-1}\dt}\leq C(T)\bigl(\norm{\xi}_{L^2(0,T;H^1(\Sigma))}+\norm{\xi}_{L^2(0,T;H^{\frac 12}(\Sigma))}\bigr)\leq 
			C(T)\norm{\xi}_{L^2(0,T;H^1(\Sigma))},
		\end{align*}
		giving
		\begin{align}
			\norm{\mM_\eps(a)\partial_tw_\eps}_{L^2(0,T;H^1(\Sigma)')}\leq C(T),
			\label{est1a}
		\end{align}
		for any $T>0$. In conclusion, since $a\geq c_0>0$ and thus, recalling \eqref{controlJ}, $\mM_\eps(a)\geq \frac{c_0}{C_J}$, \mich{we infer, for any $\xi\in L^2(0,T;H^1(\Sigma))$,}
		\begin{align*}
			&\normmm{\int_0^T \int_\Sigma \partial_tw_\eps \xi\d\mathcal H^{N-1}\dt} \leq \normmm{\int_0^T \int_\Sigma \mM_\eps(a)\partial_tw_\eps \frac{\xi}{\mM_\eps(a)}\d\mathcal H^{N-1}\dt} \\&\leq  \norm{\mM_\eps(a)\partial_tw_\eps}_{L^2(0,T;H^1(\Sigma)')}\norm{\frac{\xi}{\mM_\eps(a)}}_{L^2(0,T;H^1(\Sigma))}
			\leq C(T)\norm{\xi}_{L^2(0,T;H^1(\Sigma))},
		\end{align*}
		\mich{		where we also used, by \eqref{P1} with $p=\infty$, that
		\begin{align*}
			\norm{\nabla_\Sigma\mM_\eps({ a})}_{L^\infty(\Sigma)}\leq C\norm{ a}_{W^{1,\infty}(\Oe)}\leq C\norm{a}_{C^1_b(\R^N\times[0,T])}.
		\end{align*}
		This means that we can infer
		\begin{align}
			\norm{\partial_tw_\eps}_{L^2(0,T;H^1(\Sigma)')}\leq C(T),
			\label{est1b}
		\end{align}
		for any $T>0$.}
		
		Exploiting the uniform estimates \eqref{est2f} for $v_\eps$ in $\Omega_b$, and \eqref{regveps}--\eqref{est1b} for $w_\eps$ on $\Sigma$, we can now deduce, by compactness and the Aubin-Lions Lemma, that there exist a (nonrelabeled) subsequence $\{(v_\eps,w_\eps)\}_{\eps}$ and a pair $(u,v)$ such that 
		\begin{align*}
			&v_\eps \to v,\quad \text{weakly$^*$ in } L^\infty(0,T;H^1(\Omega_b)),
			\\&
			v_\eps \to v,\quad \text{strongly in } C([0,T];H^s(\Omega_b)),\quad s\in(\tfrac12,1),\\&
			\partial_t v_\eps \to \partial_t v,\quad \text{weakly in } L^2(0,T;L^2(\Omega_b)),
		\end{align*}
		as well as
		\begin{align*}
			&w_\eps \to u,\quad \text{weakly$^*$ in } L^\infty(0,T;L^2(\Sigma)),\\&
			w_\eps \to u,\quad \text{weakly in } L^2(0,T;H^1(\Sigma)),
			\\&
			w_\eps \to u,\quad \text{strongly in } L^2(0,T;L^2(\Sigma)),\\&
			\partial_t w_\eps \to \partial_t u,\quad \text{weakly in } L^2(0,T;H^1(\Sigma)'),
		\end{align*}
		for any $T>0$. We can now identify $u$ with $\mathrm{Tr}_\Sigma v$ a.e. on $\Sigma$. Indeed, by means of \eqref{trace} we have, for any $\eps$ sufficiently small,
		\begin{align}
			\int_0^T\norm{\mathrm{Tr}_\Sigma u_\eps-w_\eps}_{L^2(\Sigma)}^2\dt\leq C\eps\int_0^T\norm{u_\eps}_{H^1(\Omega_\eps)}^2\dt\leq C(T)\eps^2,\label{trace2}
		\end{align}
		by means of \eqref{est2f}.
		Also, thanks to the convergences above, namely $v_\eps\to v$ \mich{strongly in $C([0,T];H^s(\Omega_b))$ for $s\in(\frac12,1)$,} by the continuity of the trace operator we get
		\begin{align*}
			\mathrm{Tr}_\Sigma u_\eps=\mathrm{Tr}_\Sigma v_\eps\to \mathrm{Tr}_\Sigma v,\quad \text{ strongly in }L^2(0,T;H^{s-\frac12}(\Sigma)).
		\end{align*}
		Since we also know that $w_\eps\to u$ strongly in $L^2(0,T;L^2(\Sigma))$ as $\eps\to0$, this means that
		\begin{align*}
			& \norm{\mathrm{Tr}_\Sigma v-u}_{L^2(0,T;L^2(\Sigma))}\\&\leq \norm{\mathrm{Tr}_\Sigma v-\mathrm{Tr}_\Sigma u_\eps}_{L^2(0,T;L^2(\Sigma))}+\norm{\mathrm{Tr}_\Sigma u_\eps-w_\eps}_{L^2(0,T;L^2(\Sigma))}+\norm{w_\eps-u}_{L^2(0,T;L^2(\Sigma))}\to 0,
		\end{align*}
		as $\eps\to0$, entailing as desired that $\mathrm{Tr}_\Sigma v=u$ a.e. on $\Sigma$.
		To conclude the proof of Theorem \ref{thmconvergence} we only need to pass to the limit in \eqref{recast}. First, let us observe that, by the $C^{1}$ regularity of $a$, we immediately have, using \eqref{trace},
		\begin{align*}
			\sup_{t\in[0,T]}\norm{a_{|\Sigma}(t)-\mM_\eps(a(t))}_{L^2(\Sigma)}\leq C\eps^\frac12\norm{a}_{L^\infty(0,T;H^1(\Omega_\eps))}\leq C\eps\norm{a}_{C^1_b(\R^N\times [0,\infty))}\to 0,
		\end{align*}
		as $\eps\to0$, for any $T>0$, and the same holds for $\mathbf B$, so that in the limit as $\eps\to 0$ we simply need to evaluate $a$ and $\mathbf B$ on $\Sigma$. On the other hand, recalling assumption \eqref{f1cond}, we have by means of \eqref{trace} that 
		\begin{align*}
			\norm{\mathrm{Tr}_\Sigma g-\mM_\eps(g)}_{L^2(\Sigma)}\leq C\eps^\frac12\norm{g}_{H^1(\Oe)}\leq C\eps^{\frac12-\alpha}\to 0,
		\end{align*}
		as $\eps\to0$, since we have assumed $\alpha\in[0,\frac12)$. Using these results and the convergences above we can finally pass to the limit as $\eps\to0$ in \eqref{recast}, making use of \eqref{estR}, so that the pair $(u,v)$ satisfies 
		\begin{align}
			\nonumber&\int_0^T\langle a\partial_t u ,\zeta\rangle_{H^1(\Sigma)',H^1(\Sigma)}\d t +\int_0^T \langle a \partial_t v,  \zeta\rangle_{H^1(\Omega_b)',H^1(\Omega_b)}\dt\\&			+\int_0^T\int_{\Sigma} \mathbf B_{|\Sigma}\nabla_\Sigma u\cdot \nabla_\Sigma \zeta\d\mathcal H^{N-1}\dt+\int_0^T\int_{\Omega_b} \mathbf D\nabla v\cdot \nabla \zeta\dx\d t\nonumber\\&\quad -\int_0^T\int_{\Sigma} (g+\op)\zeta\d \mathcal H^{N-1}\dt-\int_0^T\int_{\Omega_b}h\zeta\d x\dt=0,\label{recast1}
		\end{align}
		for any $\zeta\in L^2(0,T;H^1(\Omega_b))\cap L^2(0,T;H^1(\Sigma))$, with $u=\mathrm{Tr}_{\Sigma}v$ a.e. in $\Sigma\times(0,T)$. At last, concerning the initial datum, by \eqref{control} we see that $w_\eps(0)\to \mathrm{Tr}_\Sigma  v_0 $ strongly in $L^2(\Sigma)$ as $\eps\to0$, and thus $\mathrm{Tr}_\Sigma v(0)=\mathrm{Tr}_\Sigma  v_0 $. Since we clearly have $v(0)= v_0 $ in $\Omega_b$, this concludes the proof of Theorem \ref{thmconvergence}. Note that, by the uniqueness of the solution to \eqref{recast1}, \mich{the entire sequence $\{(v_\eps,w_\eps)\}_\eps$ converges as above to the same limit as $\eps\to0$.}
		
		\subsubsection{Proof of Theorem \ref{thmrates}}
		\mich{The proof takes some inspiration from \cite[Theorem 1.3]{Miura2}, but it is considerably extended in order to take care of the nontrivial film-bulk coupling.}
		Let us use the same notation as in Section \ref{Sec:convergence}, namely we consider the sequence $\{(u_\eps,v_\eps)\}_\eps$ and the limit solution $(u,v)$ satisfying \eqref{thinfilm11b}. We then introduce
		\begin{align*}
			& V_\eps:=\begin{cases}u_\eps-\widetilde{u},\quad\text{ in }\Oe,\\
				v_\eps-v,\quad\text{ in }\Omega_b,
			\end{cases}
		\end{align*}
		so that 
		\begin{align}
			V_\eps(0)=\begin{cases}
				u_{0,\eps}-\widetilde{u_0},\quad\text{ in }\Oe,\\
				0,\quad\text{ in }\Omega_b,
			\end{cases}
			\label{Vzero}
		\end{align}
		where we recall that $u_0=\mathrm{Tr}_\Sigma v_0$ by assumption.
		\mich{Notice that, since $\widetilde u$ is the constant extension of $u=\mathrm{Tr}_\Sigma v$ along the normal direction, the traces of $V_\eps$ on $\Sigma$ from within $\Omega_\eps$ and from within $\Omega_b$ coincide:
		\begin{align}
			\mathrm{Tr}_{\Sigma}\bigl(u_\eps-\widetilde u\bigr)=\mathrm{Tr}_\Sigma u_\eps-u=\mathrm{Tr}_\Sigma v_\eps-\mathrm{Tr}_\Sigma v,
			\label{tracejump}
		\end{align}
		so that $V_\eps$ has no jump across $\Sigma$ and $V_\eps\in L^2(0,T;H^1(\Omega))$.}
		\mich{Let us now assume that $v$ is also a strong solution according to Theorem \ref{thmthin}, namely that, in addition to $ v_0 \in H^1(\Omega_b)$, it holds $u_0=\mathrm{Tr}_\Sigma  v_0  \in H^1(\Sigma)$. }Then, by \eqref{extracondition}, we have the following further uniform estimate
		\begin{align}
			\norm{\partial_t u}_{L^2(0,T;L^2(\Sigma))}+\norm{u}_{L^2(0,T;H^2(\Sigma))}\leq C(T),\label{extrareg}
		\end{align}
		for any $T>0$ and some $C(T)>0$.
		
		\mich{Now, from \eqref{weakform}, integrated in time on $(0,t)$ with $t\in(0,T]$, satisfied by $(u_\eps,v_\eps)$, we subtract \eqref{thinf}, integrated in time on $(0,t)$ and satisfied by $(u,v)$, multiplied by $\eps>0$.} We get
		\begin{align}
			\nonumber&\int_0^t\int_{\Oe} a\partial_t u_\eps \zeta\dx\d s-\eps\int_0^t\int_\Sigma a\partial_t u \zeta\d\mathcal H^{N-1}\d s+\eps\int_0^t\int_{\Omega_b} a\partial_t V_\eps  \zeta\dx\d s\\&\quad +\int_0^t\int_{\Oe} \mathbf B\nabla u_\eps\cdot \nabla \zeta\dx\d s-\eps\int_0^t\int_{\Sigma} \mathbf B_{|\Sigma}\nabla_\Sigma u\cdot \nabla_\Sigma \zeta\d\mathcal H^{N-1}\d s+\eps\int_0^t\int_{\Omega_b} \mathbf D\nabla V_\eps\cdot \nabla \zeta\dx\d s\nonumber\\&\quad -\int_0^t\int_{\Oe} g\zeta\dx\d s+\eps\int_0^t\int_{\Sigma} g\zeta\d \mathcal H^{N-1}\d s-\eps\int_0^t\int_{{\Sigma_{\mathrm p}}} \op\zeta\d\mathcal H^{N-1}\d s+\eps\int_0^t\int_{\Sigma} \op\zeta\d \mathcal H^{N-1}\d s=0,\label{recast3}
		\end{align}
		for any $\zeta\in L^2(0,t;H^1(\Omega))$. Thanks to the trace identity \eqref{tracejump} on $\Sigma$, we can safely choose $\zeta=V_\eps\in L^2(0,t;H^1(\Omega))$ and, by adding and subtracting suitable terms, obtain
		\begin{align}
			\nonumber&\underbrace{\int_0^t\int_{\Oe} a\partial_t u_\eps V_\eps\dx\d s-\eps\int_0^t\int_\Sigma \partial_t u\, \mM_\eps(a V_\eps)\d\mathcal H^{N-1}\d s}_{L_1}+\frac{\eps}{2}\int_{\Omega_b} a(t) \normmm{V_\eps(t)}^2 \dx\\&\quad +\underbrace{\int_0^t\int_{\Oe} \mich{\mathbf B\nabla u_\eps\cdot \nabla V_\eps\dx\d s}-\eps\int_0^t\int_{\Sigma} \mathbf B_{\mich{|\Sigma}}\nabla_\Sigma u\cdot \nabla_\Sigma V_\eps\d\mathcal H^{N-1}\d s}_{L_2}+\eps\int_0^t\int_{\Omega_b} \mathbf D\nabla V_\eps\cdot\nabla V_\eps\dx\d s\nonumber\\&
			=\frac{\eps}{2}\int_{\Omega_b} a(0) \normmm{V_\eps(0)}^2 \dx+
			\underbrace{\eps\int_0^t\int_\Sigma \partial_t u \left(aV_\eps-\mM_\eps(aV_\eps)\right)\d\mathcal H^{N-1}\d s}_{L_3}\nonumber\\&
			+\underbrace{\eps\int_0^t\int_{\Sigma} g\bigl(\mM_\eps(V_\eps)-V_\eps\bigr)\d \mathcal H^{N-1}\d s}_{L_4}
			+\underbrace{\int_0^t\int_{\Oe} gV_\eps\dx\d s-\eps\int_0^t\int_{\Sigma} g\,\mM_\eps(V_\eps)\d \mathcal H^{N-1}\d s}_{L_5}\nonumber\\&
			+\underbrace{\eps\int_0^t\int_{{\Sigma_{\mathrm p}}} \op V_\eps\d\mathcal H^{N-1}\d s-\eps\int_0^t\int_{\Sigma}\op V_\eps\d \mathcal H^{N-1}\d s}_{L_6}.
			\label{recast4}
		\end{align}
        
		We now estimate $L_1$--$L_6$. First, observe that, similarly to the proof of Lemma \ref{timederivative}, it holds, using \eqref{basic},
		\begin{align*}
			\eps\int_0^t\int_\Sigma \partial_t u\, \mM_\eps(a V_\eps)\d\mathcal H^{N-1}\d s=\int_0^t\int_{\Oe} a\,\partial_t\widetilde u\, V_\eps\dx\d s, 
		\end{align*}
		so that it is immediate to infer, recalling that $V_\eps=\mich{u_\varepsilon}-\widetilde u$ in $\Omega_\varepsilon$,
		\begin{align*}
			L_1&=\int_0^t\int_{\Oe}\frac{a}2\ddt \normmm{V_\eps}^2\dx\d s\\&=\frac12\int_{\Oe}a(t)\normmm{V_\eps(t)}^2\dx-\frac12\int_{\Oe}a(0)\normmm{V_\eps(0)}^2\dx-\int_0^t\int_{\Oe}\frac{\partial_t a}2 \normmm{V_\eps}^2\dx\d s.
		\end{align*}
		
		We now decompose $L_2$ as
		\begin{align*}
			L_2&=\int_0^t\int_{\Oe}\mathbf B\nabla V_\eps\cdot\nabla V_\eps\dx\d s\mich{+}\underbrace{\int_0^t\int_{\Oe}\mathbf B\nabla \widetilde u\cdot \nabla V_\eps\dx\d s\mich{-}\eps\int_0^t\int_{\Sigma}\mathbf B_{|\Sigma}\nabla_\Sigma \mM_\eps(V_\eps)\cdot \nabla_\Sigma u\d \mathcal H^{N-1}\d s}_{L_7}\\&\quad +\underbrace{\eps\int_0^t\int_{\Sigma}\mathbf B_{\mich{|\Sigma}}\bigl(\nabla_\Sigma \mM_\eps(V_\eps)-\nabla_\Sigma V_\eps\bigr)\cdot \nabla_\Sigma u\d \mathcal H^{N-1}\d s}_{L_8}.
		\end{align*}
		Concerning $L_7$, we observe that, under the assumption $\bP(\mathbf x)\mathbf B(\mathbf x,\cdot)=\mathbf B(\mathbf x,\cdot)\bP(\mathbf x)$ for any $\mathbf \mich{x}\in \Sigma$, by means of \eqref{controlb2}, since $t\leq T$, it holds
		\begin{align*}
			\normmm{L_7}&\rev{=\normmm{\int_0^t\int_{\Oe}\mathbf B\nabla V_\eps\cdot \nabla \widetilde u\dx\d s-\eps\int_0^t\int_{\Sigma}\mathbf B_{|\Sigma}\nabla_\Sigma \mM_\eps(V_\eps)\cdot \nabla_\Sigma u\d \mathcal H^{N-1}\d s}}\\&\leq C\eps^\frac32\norm{\mathbf B}_{C^{0,1}(\R^N\times[0,T])}\norm{V_\eps}_{L^2(0,t;H^1(\Omega_\eps))}\norm{\nabla_\Sigma u}_{L^2(0,T;L^2(\Sigma))}.  
		\end{align*}

		\rev{Then, about $L_8$, we have, integrating by parts on $\Sigma$ and using \eqref{ctrl1}, \eqref{trace}, \eqref{P1},
		\begin{align*}
			\normmm{L_8}&\leq \eps\normmm{\int_0^t\int_\Sigma \Div_\Sigma (\mathbf B_{|\Sigma})\cdot \nabla_\Sigma u\,\bigl(\mM_\eps(V_\eps)-V_\eps\bigr)\d\mathcal H^{N-1}\d s}\\&
			+\eps\normmm{\int_0^t\int_\Sigma \mathbf B_{|\Sigma} :\nabla_\Sigma^2 u\,\bigl(\mM_\eps(V_\eps)-V_\eps\bigr)\d\mathcal H^{N-1}\d s}\\&
			+\eps\normmm{\int_0^t\int_\Sigma H_\Sigma \mathbf B_{|\Sigma}\bn_\Sigma \cdot\nabla_\Sigma u\,\bigl(\mM_\eps(V_\eps)-V_\eps\bigr)\d\mathcal H^{N-1}\d s}\\&
			\leq \eps \int_0^t\norm{\Div_\Sigma \mathbf B}_{L^{\infty}(\Sigma)}\norm{\nabla_\Sigma u}_{L^{2}(\Sigma)}\norm{\mM_\eps(V_\eps)-V_\eps}_{L^2(\Sigma)}\d s\\&
			+\eps\int_0^t\norm{\mathbf B}_{L^\infty(\Sigma)}\norm{\nabla_\Sigma^2 u}_{L^2(\Sigma)}\norm{\mM_\eps(V_\eps)-V_\eps}_{L^2(\Sigma)}\d s\\&
			+C\eps\int_0^t\norm{\mathbf B}_{L^\infty(\Sigma)}\norm{\nabla_\Sigma u}_{L^2(\Sigma)}\norm{\mM_\eps(V_\eps)-V_\eps}_{L^2(\Sigma)}\d s\\&
			\leq C\eps^{\frac32}\int_0^t\norm{\mathbf B}_{W^{1,\infty}(\Oe)}\norm{u}_{H^2(\Sigma)}\norm{V_\eps}_{H^1(\Oe)}\d s\\&\quad +C\eps^\frac32\int_0^t\norm{\mathbf B}_{L^{\infty}(\Oe)}\norm{u}_{H^1(\Sigma)}\norm{V_\eps}_{H^1(\Oe)}\d s\\&
			\leq C\eps^\frac32\norm{\mathbf B}_{C^{0,1}(\R^N\times[0,T])}\norm{u}_{L^2(0,T;H^2(\Sigma))}\norm{V_\eps}_{L^2(0,t;H^1(\Oe))},
		\end{align*}
		where $H_\Sigma$ is the scalar mean curvature of $\Sigma$.}
		Then, we proceed with the estimate of $L_3$. In particular, we have by \eqref{trace} that
		\begin{align*}
			\normmm{L_3}&\leq \eps\int_0^t\norm{\partial_t u}_{L^2(\Sigma)}\norm{aV_\eps-\mM_\eps(aV_\eps)}_{L^2(\Sigma)}\d s\\&\leq C\eps^\frac32 \int_0^t\norm{\partial_t u}_{L^2(\Sigma)}\norm{a V_\eps}_{H^1(\Oe)}\d s\\&\leq C\eps^{\frac32}\norm{a}_{C^1_b(\R^N\times[0,T])}\norm{\partial_t u}_{L^2(0,T;L^2(\Sigma))}\norm{V_\eps}_{L^2(0,t;H^1(\Oe))}.
		\end{align*}
		Then, concerning $L_4$, again by \eqref{trace} we infer
		\begin{align*}
			\normmm{L_4}\leq \eps\norm{g}_{L^2(\Sigma)}\norm{\mM_\eps(V_\eps)-V_\eps}_{L^1(0,t;L^2(\Sigma))}\leq \mich{C}\eps^\frac32 T^\frac12\norm{g}_{L^2(\Sigma)}\norm{V_\eps}_{L^2(0,t;H^1(\Oe))}.
		\end{align*}
		About $L_5$, observe that, by \eqref{basic},
		\begin{align*}
			\eps\int_0^t\int_{\Sigma} g\,\mM_\eps(V_\eps)\d \mathcal H^{N-1}\d s=\int_0^t \int_{\Oe} \widetilde g\,V_\eps\dx\d s,    
		\end{align*}
		where we set $\widetilde g=\widetilde{\mathrm{Tr}_\Sigma g}$, and thus we have
		\begin{align*}
			\normmm{L_5}&\leq t^\frac12\norm{g-\widetilde g}_{L^\infty(0,t;L^2(\Oe))}\norm{V_\eps}_{L^2(0,t;L^2(\Oe))}\\&
			\leq CT^\frac12\eps^\frac12\norm{g-\widetilde g}_{L^\infty(\Oe)}\norm{V_\eps}_{L^2(0,t;L^2(\Oe))}\\&\leq CT^\frac12\eps^\frac32\norm{g}_{C^{0,1}(\overline\Omega_\delta)}\norm{V_\eps}_{L^2(0,t;L^2(\Oe))},
		\end{align*}
		where we have used that $\normmm{d}\leq C\eps$ in $\Omega_\eps$, $d$ being the signed distance function from $\Sigma$. In conclusion, concerning $L_6$, we have, by the change of variables \eqref{changesurf} and arguing as for \eqref{K4},
		\mich{\begin{align*}
			\normmm{L_6}&=\eps\normmm{\int_0^t\int_{\Sigma} {\op^\sharp(y,\eps)V_\eps^\sharp(y,\eps)} J(y,\eps)\d\mathcal H^{N-1}\d s-\int_0^t\int_{\Sigma} \op\, \mathrm{Tr}_\Sigma V_\eps\d \mathcal H^{N-1}\d s}\\&
            \leq \eps \normmm{\int_0^t\int_{\Sigma} {\op^\sharp(y,\eps)(V_\eps^\sharp(y,\eps)}-V_\eps(y)) J(y,\eps)\d\mathcal H^{N-1}\d s} 
            \\&\quad +\eps\normmm{\int_0^t\int_{\Sigma} {(\op^\sharp(y,\eps)-\op(y))V_\eps(y)} J(y,\eps)\d\mathcal H^{N-1}\d s}\\&
+\eps\normmm{\int_0^t\int_{\Sigma}\op_{|\Sigma}V_\eps(y)\bigl(1-\mich{J(y,\eps)}\bigr)\d \mathcal H^{N-1}\dt}\\&
				\leq CT^\onehalf\eps^\frac32\norm{V_\eps}_{L^2(0,t;H^1(\Omega_\eps))}\norm{\op}_{L^2(\Sigma_p)}+CT^\frac12\eps^2\mich{\norm{\op}_{C^{0,1}(\overline{\Omega}_\delta)}}\norm{V_\eps}_{L^2(0,t;L^2(\Sigma))}\\&\quad +CT^\frac12\eps^2{\norm{\op}_{L^2(\Sigma)}}\norm{V_\eps}_{L^2(0,t;L^2(\Sigma))}
			\\&\leq C(T)\eps^{\frac32}\norm{V_\eps}_{L^2(0,t;H^1(\Oe))},
		\end{align*}}
   where we also used \eqref{ctrltrace3} and
   \begin{align*}
       \int_\Sigma \normmm{V^\sharp (y,\eps)-V(y)}^2\d\mathcal H^{N-1}\leq \eps\int_\Sigma \int_0^\eps\normmm{\partial_r V^\sharp(y,r)}^2\d r\d\mathcal H^{N-1}\leq C\eps\norm{V_\eps}_{H^1(\Oe)}^2,
   \end{align*}
   using \eqref{controleta}.
        		
		Summing up, using Young's inequality we can thus obtain from \eqref{recast4} that 
		\begin{align*}
			&\frac{\eps}{2}\int_{\Omega_b} a(t) \normmm{V_\eps(t)}^2 \dx+\frac12\int_{\Oe}a(t)\normmm{V_\eps(t)}^2\dx+\eps \int_0^t\int_{\Omega_b}\mathbf D\nabla V_\eps\cdot\nabla V_\eps\dx\d s+\int_0^t\int_{\Oe}\mathbf B\nabla V_\eps\cdot \nabla V_\eps\dx\d s\\&
			\leq \frac{\eps}{2}\int_{\Omega_b} a(0) \normmm{V_\eps(0)}^2 \dx+\frac12\int_{\Oe}a(0)\normmm{V_\eps(0)}^2\dx\\&\quad + C(T)\int_0^t\int_{\Oe}a\normmm{V_\eps}^2\dx\d s+\frac12\int_0^t\int_{\Omega_\eps}\mathbf B\nabla V_\eps\cdot\nabla V_\eps \dx\d s\\&\quad +C(T)\eps^3\bigl(1+\norm{u}_{L^2(0,T;H^2(\Sigma))}+\norm{\partial_t u}_{L^2(0,T;L^2(\Sigma))}\bigr)\mich{^2}\\&
			\leq C\norm{V_\eps(0)}_{L^2(\Oe)}^2+C(T)\int_0^t\int_{\Oe}a\normmm{V_\eps}^2\dx\d s+\frac12\int_0^t\int_{\Oe}\mathbf B\nabla V_\eps\cdot\nabla V_\eps\dx\d s+C(T)\eps^3,
		\end{align*}
		for any $t\in(0,T]$, where we have also used \eqref{extrareg} and \eqref{Vzero}. This gives
		\begin{align*}
			&\frac{\eps}{2}\int_{\Omega_b} a(t) \normmm{V_\eps(t)}^2 \dx+\frac12\int_{\Oe}a(t)\normmm{V_\eps(t)}^2\dx\\&\quad +\eps \int_0^t\int_{\Omega_b}\mathbf D\nabla V_\eps\cdot\nabla V_\eps\dx\d s+\frac12\int_0^t\int_{\Oe}\mathbf B\nabla V_\eps\cdot\nabla V_\eps\dx\d s\\&
			\leq C(T)\int_0^t\int_{\Oe}a\normmm{V_\eps}^2\dx\d s+C\norm{V_\eps(0)}_{L^2(\Oe)}^2+C(T)\eps^3,\quad \forall t\in(0,T].
		\end{align*}
		As a consequence, by means of the integral Gronwall Lemma we finally deduce \eqref{rate2}, after dividing by $\eps$ and taking the square root.
\rev{Finally, in the more general case in which the identity $\bP \mathbf B=\mathbf B\bP$ on $\Sigma$ is not assumed, we can change the estimate of $L_7$. Indeed, by means of \eqref{final}, assuming $\eps\in(0,1)$ we can write
		\begin{align*}
			\normmm{L_7}&\leq C(\norm{\mathbf B}_{C^{0,1}(\R^N\times[0,T])})\eps^\frac32\norm{V_\eps}_{L^2(0,t;H^1(\Oe))}\norm{u}_{L^2(0,T;H^1(\Sigma))}\\&\quad +C\eps^\onehalf(\norm{\bP\mathbf B_{|\Sigma}\bn_\Sigma}_{L^\infty(\Sigma)}+\norm{\bn_\Sigma\mathbf B_{|\Sigma}\bP}_{W^{1,\infty}(\Sigma)})\norm{V_\eps}_{L^2(0,t;H^1(\Oe))}\norm{u}_{L^2(0,T;H^2(\Sigma))},
		\end{align*}
		which then gives, together with all the other estimates above, inequality \eqref{C2}.}
		This concludes the proof. 

		\section{Thermal Transmission Problem in a Simplified Film-Bulk Domain}
		\label{sec:simplified}
		
		We now investigate the thermal transmission problem introduced in Section~\ref{sec:problem} within a simplified film-bulk domain. Whereas Section~\ref{sec:thinfilm_analysis} establishes the general mathematical framework and the corresponding convergence results, the present section considers a simplified setting in which explicit analytical solutions can be derived.
		
		The objective is threefold: to analyze both the finite-thickness film configuration and the corresponding thin-film limit in a case where both admit a closed-form solution, to compare their thermal responses as the film thickness decreases, and to compare the convergence rate observed in this particular configuration with the one established in Theorem~\ref{thmrates}.
		\subsection{Simplified Film-Bulk Domain and Problem Formulation}
		\label{simply1}
		
		We consider a simplified two-dimensional rectangular film-bulk domain. It consists of a bulk region of height $H$ and length $L$, together with a film region of uniform thickness $\varepsilon$ attached to the upper boundary of the bulk region, the two being separated by the interface $\Sigma$.
		
		As illustrated in Figure~\ref{fig:simplified_domain}, the bulk region occupies the rectangle $(0,L)\times(0,H)$ and the film region the layer $(0,L)\times(H,H+\varepsilon)$, so that $\Sigma=(0,L)\times\{H\}$. The upper boundary of the film is denoted by $\Sigma_{\mathrm p}$, the lower boundary of the bulk by $\Sigma_v$, and the lateral boundaries by $\Sigma_{\mathrm q}$.
		
		\begin{figure}
			\centering		
            \includegraphics[width=0.5\textwidth]{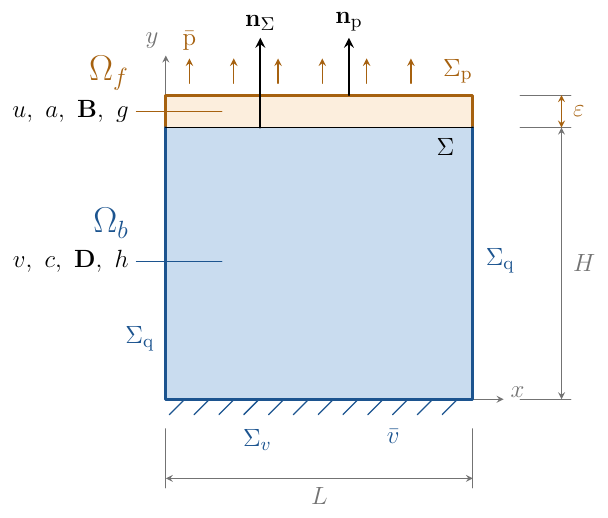}					\caption{Schematic representation of the simplified film-bulk domain used for the derivation of analytical solutions. The bulk region occupies the rectangle $(0,L)\times(0,H)$ and the film region the layer $(0,L)\times(H,H+\varepsilon)$ of uniform thickness $\varepsilon$, the two being separated by the interface $\Sigma=(0,L)\times\{H\}$. The heat flux $\op$ is prescribed on the upper boundary $\Sigma_{\mathrm p}$ of the film, a homogeneous Dirichlet condition on the lower boundary $\Sigma_v$ of the bulk, and homogeneous Neumann conditions on the lateral boundaries $\Sigma_{\mathrm q}$. Each region carries its temperature field, constant volumetric heat capacity and constant thermal conductivity: $u$, $a$, $B$ in the film and $v$, $c$, $D$ in the bulk.}
			\label{fig:simplified_domain}
		\end{figure}	
		Notice that this configuration does not fall within the geometric setting of Section~\ref{mat_sec1}, in which $\Sigma$ is a closed surface without boundary, the bulk region has no external boundary and $\Sigma_v$, $\Sigma_{\mathrm q}$ are empty: here $\Sigma$ is a segment with boundary and both $\Sigma_v$ and $\Sigma_{\mathrm q}$ are non-empty. The comparison carried out in Section~\ref{comparisons} is therefore to be understood as a consistency check in a degenerate configuration, rather than as an instance of Theorem~\ref{thmrates}.
		
		We assume that heat transfer occurs along the vertical direction only. Hence all fields depend on the vertical coordinate $y$ and on time $t$ alone, homogeneous Neumann conditions are imposed on the lateral boundaries $\Sigma_{\mathrm q}$, a homogeneous Dirichlet condition is imposed on $\Sigma_v$, and the diffusion operators reduce to second derivatives with respect to $y$. The remaining boundary and transmission conditions are obtained from those introduced in Section~\ref{sec:problem}, specialized to the present setting. 
		
		To obtain explicit analytical solutions, the volumetric heat capacities and the thermal conductivities are assumed to be constant. In particular, $a$ and $c$ denote the constant volumetric heat capacities in the film and bulk regions, respectively, while $B>0$ and $D>0$ denote the corresponding scalar thermal conductivities, so that the conductivity tensors of Section~\ref{sec:problem} reduce to $\mathbf B=B\,\mathbf I$ and $\mathbf D=D\,\mathbf I$. Volumetric source terms are neglected in both regions, in order to isolate the role of the transmission conditions, of the prescribed boundary heat flux and of the film thickness in the thermal response.
		
		Under these assumptions, the rescaled thermal transmission problem~\eqref{eps1} in the finite-thickness film configuration, with $\varepsilon>0$, reduces to
		\begin{align}
			\label{eq:finite_eps_problem_nonhom}
			\begin{cases}
				a\,\partial_t u - B\,\partial_{yy}u = 0
				& \text{in } H<y<H+\varepsilon,\\[4pt]
				c\,\partial_t v - D\,\partial_{yy}v = 0
				& \text{in } 0<y<H,\\[4pt]
				u=v
				& \text{on } y=H,\\[4pt]
				B\,\partial_y u=\varepsilon D\,\partial_y v
				& \text{on } y=H,\\[4pt]
				B\,\partial_y u =\varepsilon\,\op(t)
				& \text{on } y=H+\varepsilon,\\[4pt]
				v=0
				& \text{on } y=0,
			\end{cases}
		\end{align}
		supplemented with initial conditions $u(\cdot,0)=u_0$ in $(H,H+\varepsilon)$ and $v(\cdot,0)=v_0$ in $(0,H)$, with $u_0=v_0$ at $y=H$.
		
		The corresponding thin-film limit problem~\eqref{thinfilm11} reduces to
		\begin{align}
			\label{eq:limit_bulk_interface_nonhom}
			\begin{cases}
				c\,\partial_t v - D\,\partial_{yy}v = 0
				& \text{in } 0<y<H,\\[4pt]
				u=v
				& \text{on } y=H,\\[4pt]
				a\,\partial_t u=-D\,\partial_y v+\op(t)
				& \text{on } y=H,\\[4pt]
				v=0
				& \text{on } y=0,
			\end{cases}
		\end{align}
		supplemented with $v(\cdot,0)=v_0$ in $(0,H)$ and $u(0)=u_0:=v_0(H)$.
		In this simplified one-dimensional setting the tangential diffusion term appearing in the general bulk-interface formulation vanishes, since the interfacial temperature does not vary along the tangential direction, and the interfacial balance relation reduces to an ordinary differential equation for $u$.
		
		\subsection{Analytical Solutions in the Finite-Thickness Film Configuration and Thin-Film Limit}
		\label{analytic}
		The analysis aims to provide an explicit description of the thermal interaction between the film and the bulk regions, and to illustrate in a simplified setting the asymptotic behavior established in Section~\ref{sec:thinfilm_analysis}.
		
		Two representative boundary conditions are considered on the boundary $\Sigma_{\mathrm p}$ of the film region (see Figure~\ref{fig:simplified_domain}), namely the homogeneous Neumann condition $\op=0$ and the non-homogeneous Neumann condition $\op\neq0$. For each of them, analytical solutions are derived for both the finite-thickness film configuration and the corresponding thin-film limit.
		
		The derivation is based on separation of variables. In the homogeneous case this leads to an eigenvalue problem, whose solutions form a countable family of modes decaying exponentially in time, and the solution is a superposition of such modes. The present analysis retains the fundamental mode only, namely the one associated with the smallest eigenvalue, which governs the thermal response at large times, since the remaining modes decay more rapidly. In the non-homogeneous case no eigenvalue problem arises, and the decay rate is instead prescribed by the boundary forcing.
		
		Restricting the analysis to the fundamental mode is not a limitation for the present purposes. The aim is not to reconstruct the complete modal expansion, but to obtain closed-form expressions allowing a direct comparison between the two formulations; the fundamental mode already carries the whole interaction between the film and the bulk regions through the transmission conditions, which is precisely what the thin-film limit modifies. The comparison performed in Section~\ref{comparisons} is therefore to be understood between the fundamental modes of the two configurations.
		
		The resulting solutions are then compared in order to highlight the influence of the film thickness and the convergence of the finite-thickness film configuration toward the thin-film limit, and to compare the convergence rate observed in this setting with the one established in Theorem~\ref{thmrates}.
		
		\subsubsection{Homogeneous Neumann Condition ($\op = 0$)}
		\label{homogeneus}
		\paragraph{\emph{\textbf{Finite-Thickness Film Configuration}}}
		
		We first consider the homogeneous Neumann case, obtained from problem~\eqref{eq:finite_eps_problem_nonhom} by setting \(\op=0\). In the absence of a prescribed heat flux on the upper boundary of the film region, the thermal transmission problem reduces to an eigenvalue problem.
		
		To derive the fundamental modal response, we employ separation of variables and seek the temperature fields in the form
		\begin{align}
			u(y,t)=U(y)e^{-\lambda t},
			\qquad
			v(y,t)=V(y)e^{-\lambda t},
			\label{eq:sep_finite_hom}
		\end{align}
		where \(\lambda>0\) denotes the modal decay rate, while \(U\) and \(V\) describe the spatial distributions of temperature within the film and the bulk regions, respectively.
		
		Substituting \eqref{eq:sep_finite_hom} into the governing equations of problem~\eqref{eq:finite_eps_problem_nonhom} yields
		\[
		-a\lambda U-BU''=0,
		\qquad
		-c\lambda V-DV''=0,
		\]
		that is, introducing the wave numbers
		\[
		\alpha^2=\frac{a\lambda}{B},
		\qquad
		\beta^2=\frac{c\lambda}{D},
		\]
		\[
		U''+\alpha^2U=0,
		\qquad
		V''+\beta^2V=0,
		\]
		whose general solutions are
		\[
		U(y)=A_1\cos(\alpha y)+A_2\sin(\alpha y),
		\qquad
		V(y)=C_1\cos(\beta y)+C_2\sin(\beta y),
		\]
		with \(A_1\), \(A_2\), \(C_1\) and \(C_2\) integration constants.
		
		We begin by imposing the Dirichlet condition at the lower boundary of the bulk region. Since \(V(0)=0\), it follows that \(C_1=0\), and therefore
		\[
		V(y)=C\sin(\beta y),
		\qquad C=C_2 .
		\]
		
		Next, we impose the homogeneous Neumann condition at the upper boundary of the film region,
		\[
		U'(H+\varepsilon)=0.
		\]
		Rather than enforcing this condition on the general trigonometric representation, it is convenient to incorporate it directly into the functional form of the solution, by means of the shifted representation
		\[
		U(y)
		=
		A\cos\!\bigl(\alpha(y-H-\varepsilon)\bigr).
		\]
		This entails no loss of generality, since the functions satisfying $U''+\alpha^2U=0$ together with $U'(H+\varepsilon)=0$ form a one-dimensional space, which the above representation spans. Indeed, differentiation gives
		\[
		U'(y)
		=
		-A\alpha
		\sin\!\bigl(\alpha(y-H-\varepsilon)\bigr),
		\]
		so that \(U'(H+\varepsilon)=-A\alpha\sin(0)=0\) identically.
		
		The remaining unknown coefficients are determined from the transmission conditions at the interface \(y=H\). The continuity of temperature requires \(U(H)=V(H)\), which yields
		\[
		A\cos(\alpha\varepsilon)
		=
		C\sin(\beta H),
		\]
		whereas the rescaled heat-flux balance condition reads \(BU'(H)=\varepsilon DV'(H)\). Using
		\[
		U'(H)
		=
		A\alpha\sin(\alpha\varepsilon),
		\qquad
		V'(H)
		=
		C\beta\cos(\beta H),
		\]
		we obtain
		\[
		BA\alpha\sin(\alpha\varepsilon)
		=
		\varepsilon DC\beta\cos(\beta H).
		\]
		
		The transmission conditions therefore generate a homogeneous linear system for the amplitudes \(A\) and \(C\). A non-trivial solution exists only if the determinant of the associated coefficient matrix vanishes, which gives
		\begin{align}
			B\alpha\sin(\alpha\varepsilon)\sin(\beta H)
			-
			\varepsilon D\beta
			\cos(\alpha\varepsilon)\cos(\beta H)
			=
			0.
			\label{eq:char_finite_hom_det}
		\end{align}
		Since both $\alpha$ and $\beta$ depend on $\lambda$, equation~\eqref{eq:char_finite_hom_det} is a transcendental equation for the modal decay rate, and constitutes the characteristic equation of the transmission problem. Whenever \(\cos(\alpha\varepsilon)\sin(\beta H)\neq0\), it can be written in the equivalent form
		\begin{align}
			B\alpha\tan(\alpha\varepsilon)
			=
			\varepsilon D\beta\cot(\beta H).
			\label{eq:char_finite_hom}
		\end{align}
		
		The admissible eigenvalues correspond to the strictly positive roots of \eqref{eq:char_finite_hom_det}. Since the present analysis focuses on the dominant thermal behavior, we retain only the principal eigenvalue, namely the smallest such root, denoted by \(\lambda_1^\varepsilon\), with associated wave numbers
		\[
		\alpha_1^\varepsilon
		=
		\sqrt{\frac{a\lambda_1^\varepsilon}{B}},
		\qquad
		\beta_1^\varepsilon
		=
		\sqrt{\frac{c\lambda_1^\varepsilon}{D}}.
		\]
		The continuity condition provides the relation between the amplitudes,
		\[
		C
		=
		A
		\frac{\cos(\alpha_1^\varepsilon\varepsilon)}
		{\sin(\beta_1^\varepsilon H)},
		\]
		allowing the solution to be expressed in terms of a single amplitude coefficient, which we denote by \(A_1^\varepsilon\). The fundamental-mode solution therefore takes the form
		\begin{align}
			u_\varepsilon(y,t)
			&=
			A_1^\varepsilon
			\cos\!\Bigl(
			\alpha_1^\varepsilon
			(y-H-\varepsilon)
			\Bigr)
			e^{-\lambda_1^\varepsilon t},
			\qquad
			H<y<H+\varepsilon,
			\label{eq:u_finite_hom}
			\\[4pt]
			v_\varepsilon(y,t)
			&=
			A_1^\varepsilon
			\frac{\cos(\alpha_1^\varepsilon\varepsilon)}
			{\sin(\beta_1^\varepsilon H)}
			\sin(\beta_1^\varepsilon y)
			e^{-\lambda_1^\varepsilon t},
			\qquad
			0<y<H,
			\label{eq:v_finite_hom}
		\end{align}
		where \(A_1^\varepsilon\) is determined by the projection of the initial temperature distribution onto the principal eigenfunction. For the analytical comparisons considered in the following sections, \(A_1^\varepsilon\) may be prescribed independently, since it only affects the overall amplitude of the solution and does not influence either the modal decay rate \(\lambda_1^\varepsilon\) or the associated spatial eigenfunctions.
		
		\paragraph{\emph{\textbf{Thin-Film Limit}}}
		
		We now consider the corresponding thin-film limit, obtained from problem~\eqref{eq:limit_bulk_interface_nonhom} by setting \(\op=0\). In contrast to the finite-thickness film configuration discussed above, the film region no longer occupies a volumetric domain: it has collapsed onto the interface \(y=H\), and the unknown \(u\) therefore represents an interfacial temperature rather than a volumetric temperature field. Consistently with Section~\ref{sec:thinfilm_analysis}, the solution of the thin-film limit problem is denoted by \((u,v)\) without subscript, whereas the subscript \(\varepsilon\) is reserved for the finite-thickness film configuration.
		
		As in the previous case, we derive the fundamental modal response by separation of variables. Since the interfacial temperature depends on time only, we seek the bulk and interface temperatures in the form
		\begin{align}
			v(y,t)=V(y)e^{-\lambda t},
			\qquad
			u(t)=Ue^{-\lambda t},
			\label{eq:sep_limit_hom}
		\end{align}
		where \(\lambda>0\) denotes the modal decay rate and \(U\) is the amplitude associated with the interface mode.
		
		The bulk equation of problem~\eqref{eq:limit_bulk_interface_nonhom} is the same as in the finite-thickness configuration, so that \(V''+\beta^2V=0\) with \(\beta^2=c\lambda/D\), and the Dirichlet condition \(V(0)=0\) again yields
		\[
		V(y)=C\sin(\beta y).
		\]
		
		The continuity condition at the interface, \(u(t)=v(H,t)\), provides a direct relation between the interface and the bulk amplitudes,
		\[
		U=V(H)=C\sin(\beta H),
		\]
		while the remaining condition is provided by the interfacial balance relation
		\[
		a\,\partial_t u(t)
		=
		-D\,\partial_yv(H,t).
		\]
		Using the separated representation \eqref{eq:sep_limit_hom}, this becomes
		\[
		-a\lambda Ue^{-\lambda t}
		=
		-DV'(H)e^{-\lambda t},
		\]
		and, after cancellation of the non-vanishing exponential factor and recalling that \(V'(H)=C\beta\cos(\beta H)\),
		\[
		a\lambda C\sin(\beta H)
		=
		DC\beta\cos(\beta H).
		\]
		
		A non-trivial modal solution requires \(C\neq0\). Consequently, the admissible modal decay rates satisfy
		\begin{align}
			a\lambda\sin(\beta H)
			-
			D\beta\cos(\beta H)
			=
			0,
			\qquad
			\beta=\sqrt{\frac{c\lambda}{D}},
			\label{eq:char_limit_hom_det}
		\end{align}
		which constitutes the characteristic equation associated with the thin-film limit problem. Whenever \(\sin(\beta H)\neq0\), it can be rewritten in the equivalent form
		\begin{align}
			a\lambda
			=
			D\beta\cot(\beta H).
			\label{eq:char_limit_hom}
		\end{align}
		
		It is worth observing that \eqref{eq:char_limit_hom} is recovered from \eqref{eq:char_finite_hom} in the limit of a vanishing film thickness, and that the discrepancy between the two is of second order. Indeed, expanding \(\tan(\alpha\varepsilon)=\alpha\varepsilon+\frac13(\alpha\varepsilon)^3+O(\varepsilon^5)\) in \eqref{eq:char_finite_hom}, dividing by \(\varepsilon\) and recalling that \(B\alpha^2=a\lambda\), one obtains
		\begin{align}
			a\lambda+\frac{a^2\lambda^2}{3B}\,\varepsilon^2+O(\varepsilon^4)
			=
			D\beta\cot(\beta H).
			\label{eq:char_expansion}
		\end{align}
		The characteristic equation of the finite-thickness configuration is therefore an $O(\varepsilon^2)$ perturbation of the one associated with the thin-film limit. Since $\lambda_1^0$ is a simple root of \eqref{eq:char_limit_hom_det}, the implicit function theorem applies and gives
		\begin{align}
			\lambda_1^\varepsilon=\lambda_1^0+O(\varepsilon^2),
			\label{eq:lambda_expansion}
		\end{align}
		a fact that will be used in Section~\ref{comparisons} to explain the convergence rate observed there.
		
		The admissible eigenvalues correspond to the strictly positive roots of \eqref{eq:char_limit_hom_det}. Consistently with the analysis of the finite-thickness film configuration, we retain only the principal eigenvalue, namely the smallest such root, denoted by \(\lambda_1^0\), with associated wave number
		\[
		\beta_1^0
		=
		\sqrt{\frac{c\lambda_1^0}{D}}.
		\]
		Using the continuity condition, the bulk amplitude can be expressed in terms of the interface amplitude as \(C=U/\sin(\beta_1^0 H)\). Denoting by \(A_1^0\) the amplitude of the fundamental interface mode, the corresponding analytical solution takes the form
		\begin{align}
			v(y,t)
			&=
			A_1^0
			\frac{\sin(\beta_1^0 y)}
			{\sin(\beta_1^0 H)}
			e^{-\lambda_1^0 t},
			\qquad
			0<y<H,
			\label{eq:v_limit_hom}
			\\[4pt]
			u(t)
			&=
			A_1^0
			e^{-\lambda_1^0 t}.
			\label{eq:u_limit_hom}
		\end{align}
		The coefficient \(A_1^0\) is determined by the projection of the initial temperature distribution onto the principal eigenmode. As above, for the analytical comparisons performed in the following sections \(A_1^0\) may be prescribed independently, since it only scales the overall magnitude of the thermal response and affects neither the modal decay rate \(\lambda_1^0\) nor the associated eigenfunctions.
        \subsubsection{Non-Homogeneous Neumann Condition ($\op \neq 0$)}
		\label{nonhomogeneus}
		
		\paragraph{\emph{\textbf{Finite-Thickness Film Configuration}}}
		
		We now consider the non-homogeneous Neumann case, obtained from problem~\eqref{eq:finite_eps_problem_nonhom} by prescribing a non-vanishing heat flux on the upper boundary of the film region.
		
		In contrast to the homogeneous case discussed in Section~\ref{homogeneus}, the thermal transmission problem is no longer homogeneous, and a direct application of the eigenvalue analysis developed above is not possible. To preserve the separated representation adopted throughout the present section and to obtain an analytical solution in closed form, we assume that the prescribed heat flux has the exponentially decaying form
		\begin{align}
			\op(t)
			=
			\op_0 e^{-\lambda t},
			\label{eq:flux_nonhom}
		\end{align}
		where \(\op_0\) denotes the prescribed heat-flux amplitude and \(\lambda>0\) a prescribed decay rate. Unlike the homogeneous case, the parameter \(\lambda\) is not determined by the transmission problem itself, but is imposed through the boundary forcing.
		
		Accordingly, we seek the temperature fields in the form
		\begin{align}
			u(y,t)=U(y)e^{-\lambda t},
			\qquad
			v(y,t)=V(y)e^{-\lambda t},
			\label{eq:sep_finite_nonhom}
		\end{align}
		which produces the forced response associated with \eqref{eq:flux_nonhom}. The general solution of the problem is obtained by adding to it a superposition of the homogeneous modes of Section~\ref{homogeneus}, determined by the initial datum; such modes decay like $e^{-\lambda_n^\varepsilon t}$ and are therefore negligible at large times whenever $\lambda<\lambda_1^\varepsilon$, which is the regime considered in Section~\ref{comparisons}.
		
		The governing equations reduce, exactly as in Section~\ref{homogeneus}, to
		\[
		U''+\alpha^2U=0,
		\qquad
		V''+\beta^2V=0,
		\qquad
		\alpha^2=\frac{a\lambda}{B},
		\qquad
		\beta^2=\frac{c\lambda}{D},
		\]
		and the Dirichlet condition \(V(0)=0\) again yields \(V(y)=C\sin(\beta y)\). Note that, in the present case, both wave numbers are known, since \(\lambda\) is prescribed.
		
		Next, we impose the prescribed heat-flux condition at the upper boundary of the film region,
		\[
		B\,\partial_yu(H+\varepsilon,t)
		=
		\varepsilon \op_0 e^{-\lambda t},
		\qquad\text{that is,}\qquad
		B\,U'(H+\varepsilon)
		=
		\varepsilon \op_0 .
		\]
		Since the condition is now non-homogeneous, both trigonometric functions are needed, and we adopt the shifted representation
		\[
		U(y)
		=
		A\cos\!\bigl(\alpha(y-H-\varepsilon)\bigr)
		+
		E\sin\!\bigl(\alpha(y-H-\varepsilon)\bigr),
		\]
		whose derivative is
		\[
		U'(y)
		=
		-\alpha A
		\sin\!\bigl(\alpha(y-H-\varepsilon)\bigr)
		+
		\alpha E
		\cos\!\bigl(\alpha(y-H-\varepsilon)\bigr).
		\]
		Evaluating at \(y=H+\varepsilon\) gives \(B\alpha E=\varepsilon \op_0\), so that
		\[
		U(y)
		=
		A\cos\!\bigl(\alpha(y-H-\varepsilon)\bigr)
		+
		\frac{\varepsilon \op_0}{B\alpha}
		\sin\!\bigl(\alpha(y-H-\varepsilon)\bigr).
		\]
		
		The remaining amplitudes are determined from the transmission conditions at the interface \(y=H\). The continuity of temperature, \(U(H)=V(H)\), yields
		\[
		A\cos(\alpha\varepsilon)
		- C\sin(\beta H)
		=
		\frac{\varepsilon \op_0}{B\alpha}
		\sin(\alpha\varepsilon),
		\]
		whereas the rescaled heat-flux balance condition \(BU'(H)=\varepsilon DV'(H)\), together with
		\[
		U'(H)
		=
		A\alpha\sin(\alpha\varepsilon)
		+
		\frac{\varepsilon \op_0}{B}
		\cos(\alpha\varepsilon),
		\qquad
		V'(H)
		=
		C\beta\cos(\beta H),
		\]
		yields
		\[
		B\alpha A\sin(\alpha\varepsilon)
		-
		\varepsilon DC\beta\cos(\beta H)
		=
		-\varepsilon \op_0\cos(\alpha\varepsilon).
		\]
		
		In contrast to the homogeneous case, the transmission conditions now define a non-homogeneous linear system for the amplitudes \(A\) and \(C\), and no characteristic equation arises from a solvability condition: the decay rate is prescribed through \eqref{eq:flux_nonhom}, while the transmission conditions determine the amplitudes. 	 Up to a sign, the determinant of the system is
		\begin{align}
			\Delta_\varepsilon
			:=
			\varepsilon D\beta\cos(\alpha\varepsilon)\cos(\beta H)
			-
			B\alpha\sin(\alpha\varepsilon)\sin(\beta H),
			\label{eq:det_nonhom}
		\end{align}
		that is, the left-hand side of the characteristic equation \eqref{eq:char_finite_hom_det} of the homogeneous problem. The system is therefore uniquely solvable if and only if the prescribed decay rate \(\lambda\) is not an eigenvalue of the homogeneous problem, a resonance condition that will be taken into account in Section~\ref{comparisons}. Under this assumption, Cramer's rule gives
		\begin{align}
			C_\varepsilon
			=
			\frac{\varepsilon \op_0}{\Delta_\varepsilon},
			\qquad
			A_\varepsilon
			=
			C_\varepsilon\left[
			\cos(\alpha\varepsilon)\sin(\beta H)
			+
			\frac{\varepsilon D\beta}{B\alpha}
			\sin(\alpha\varepsilon)\cos(\beta H)
			\right].
			\label{eq:amplitudes_finite_nonhom}
		\end{align}
		
		The analytical solution corresponding to the prescribed heat flux \eqref{eq:flux_nonhom} therefore takes the form
		\begin{align}
			u_\varepsilon(y,t)
			&=
			\left[
			A_\varepsilon
			\cos\!\bigl(\alpha(y-H-\varepsilon)\bigr)
			+
			\frac{\varepsilon \op_0}{B\alpha}
			\sin\!\bigl(\alpha(y-H-\varepsilon)\bigr)
			\right]
			e^{-\lambda t},
			\qquad
			H<y<H+\varepsilon,
			\label{eq:u_finite_nonhom}
			\\[4pt]
			v_\varepsilon(y,t)
			&=
			C_\varepsilon
			\sin(\beta y)
			e^{-\lambda t},
			\qquad
			0<y<H,
			\label{eq:v_finite_nonhom}
		\end{align}
		with \(A_\varepsilon\) and \(C_\varepsilon\) given by \eqref{eq:amplitudes_finite_nonhom}. The temporal decay of the solution is thus dictated by the imposed boundary forcing, while the transmission conditions determine the corresponding spatial amplitudes.
		
		\paragraph{\emph{\textbf{Thin-Film Limit}}}
		
		We now consider the corresponding non-homogeneous Neumann case in the thin-film limit, obtained from problem~\eqref{eq:limit_bulk_interface_nonhom} by prescribing a non-vanishing heat flux on the interface. Consistently with the finite-thickness film configuration, the prescribed heat flux is again assumed to take the form \eqref{eq:flux_nonhom}, so that the decay rate \(\lambda\) is prescribed by the boundary forcing rather than determined by the problem.
		
		Since the film has collapsed onto the interface \(y=H\), the unknown \(u\) represents an interfacial temperature and therefore depends on time only, and we seek
		\begin{align}
			v(y,t)=V(y)e^{-\lambda t},
			\qquad
			u(t)=Ue^{-\lambda t},
			\label{eq:sep_limit_nonhom}
		\end{align}
		where \(U\) denotes the amplitude associated with the interface temperature. As in Section~\ref{homogeneus}, the bulk equation reduces to \(V''+\beta^2V=0\) with \(\beta^2=c\lambda/D\), and the Dirichlet condition \(V(0)=0\) yields \(V(y)=C\sin(\beta y)\).
		
		The continuity condition at the interface, \(u(t)=v(H,t)\), gives \(U=C_0\sin(\beta H)\), while the interfacial balance relation reads
		\[
		a\,\partial_tu(t)
		=
		-D\,\partial_yv(H,t)
		+
		\op_0e^{-\lambda t}.
		\]
		Substituting \eqref{eq:sep_limit_nonhom} and cancelling the non-vanishing exponential factor, we obtain \(a\lambda U=DV'(H)-\op_0\), and hence, recalling \(V'(H)=C\beta\cos(\beta H)\) and the continuity relation,
		\[
		a\lambda C_0\sin(\beta H)
		=
		DC_0\beta\cos(\beta H)
		-
		\op_0 .
		\]
		The bulk and interface amplitudes are therefore given by
		\begin{align}
			C_0
			=
			\frac{\op_0}
			{D\beta\cos(\beta H)-a\lambda\sin(\beta H)},
			\qquad
			U
			=
			C_0\sin(\beta H),
			\label{eq:C_limit_nonhom}
		\end{align}
		provided that the denominator does not vanish, that is, provided that \(\lambda\) is not an eigenvalue of the homogeneous thin-film limit problem, as follows by comparison with \eqref{eq:char_limit_hom_det}. The analytical solution corresponding to the prescribed heat flux therefore takes the form
		\begin{align}
			v(y,t)
			&=
			\frac{\op_0\,\sin(\beta y)}
			{D\beta\cos(\beta H)-a\lambda\sin(\beta H)}
			e^{-\lambda t},
			\qquad
			0<y<H,
			\label{eq:v_limit_nonhom}
			\\[4pt]
			u(t)
			&=
			\frac{\op_0\,\sin(\beta H)}
			{D\beta\cos(\beta H)-a\lambda\sin(\beta H)}
			e^{-\lambda t}.
			\label{eq:u_limit_nonhom}
		\end{align}
		
		It is worth observing that, at variance with the homogeneous case of Section~\ref{homogeneus}, here the two configurations share the same decay rate \(\lambda\) and hence the same wave number \(\beta\), so that the whole discrepancy between them is carried by the amplitudes. Expanding \eqref{eq:det_nonhom} for small \(\varepsilon\) and recalling that \(B\alpha^2=a\lambda\), one obtains
		\[
		\Delta_\varepsilon
		=
		\varepsilon\bigl[D\beta\cos(\beta H)-a\lambda\sin(\beta H)\bigr]
		+O(\varepsilon^3),
		\]
		so that the bulk amplitude \eqref{eq:amplitudes_finite_nonhom} of the finite-thickness configuration reduces to the one in \eqref{eq:C_limit_nonhom} as the film thickness vanishes, with
		\begin{align}
			C_\varepsilon=C_0+O(\varepsilon^2),
			\qquad
			A_\varepsilon=U+O(\varepsilon^2).
			\label{eq:amplitude_expansion}
		\end{align} This second-order behavior will be used in Section~\ref{comparisons}, together with \eqref{eq:lambda_expansion}, to explain the convergence rates observed there.
		\subsubsection{Comparison and Convergence Analysis}
		\label{comparisons}
		
		The analytical solutions derived in Sections~\ref{homogeneus} and~\ref{nonhomogeneus} allow a direct comparison between the finite-thickness film configuration and the corresponding thin-film limit. Such a comparison provides insight into the dimensional reduction process established theoretically in Section~\ref{sec:thinfilm_analysis} and clarifies how the thermal contribution of the film evolves as the thickness tends to zero. Throughout this section the comparison is understood between the fundamental modes of the two configurations, which we collect here for convenience.
		
		For the homogeneous Neumann condition, the fundamental-mode solutions associated with the finite-thickness film configuration are
		\begin{align}
			u_\varepsilon(y,t)
			&=
			A_1^\varepsilon
			\cos\!\bigl(
			\alpha_1^\varepsilon
			(y-H-\varepsilon)
			\bigr)
			e^{-\lambda_1^\varepsilon t},
			\qquad
			H<y<H+\varepsilon,
			\tag{\ref{eq:u_finite_hom}}
			\\[4pt]
			v_\varepsilon(y,t)
			&=
			A_1^\varepsilon
			\frac{\cos(\alpha_1^\varepsilon\varepsilon)}
			{\sin(\beta_1^\varepsilon H)}
			\sin(\beta_1^\varepsilon y)
			e^{-\lambda_1^\varepsilon t},
			\qquad
			0<y<H,
			\tag{\ref{eq:v_finite_hom}}
		\end{align}
		whereas the corresponding thin-film limit solutions are
		\begin{align}
			u(t)
			&=
			A_1^0 e^{-\lambda_1^0 t},
			\tag{\ref{eq:u_limit_hom}}
			\\[4pt]
			v(y,t)
			&=
			A_1^0
			\frac{\sin(\beta_1^0 y)}
			{\sin(\beta_1^0 H)}
			e^{-\lambda_1^0 t},
			\qquad
			0<y<H.
			\tag{\ref{eq:v_limit_hom}}
		\end{align}
		
		For the non-homogeneous Neumann condition, the finite-thickness film configuration yields
		\begin{align}
			u_\varepsilon(y,t)
			&=
			\left[
			A_\varepsilon
			\cos\!\bigl(
			\alpha(y-H-\varepsilon)
			\bigr)
			+
			\frac{\varepsilon\op_0}{B\alpha}
			\sin\!\bigl(
			\alpha(y-H-\varepsilon)
			\bigr)
			\right]
			e^{-\lambda t},
			\qquad
			H<y<H+\varepsilon,
			\tag{\ref{eq:u_finite_nonhom}}
			\\[4pt]
			v_\varepsilon(y,t)
			&=
			C_\varepsilon\sin(\beta y)e^{-\lambda t},
			\qquad
			0<y<H,
			\tag{\ref{eq:v_finite_nonhom}}
		\end{align}
		with $A_\varepsilon$ and $C_\varepsilon$ given by \eqref{eq:amplitudes_finite_nonhom}, whereas the corresponding thin-film limit is
		\begin{align}
			u(t)
			&=
			\frac{\op_0\sin(\beta H)}
			{D\beta\cos(\beta H)-a\lambda\sin(\beta H)}
			e^{-\lambda t},
			\tag{\ref{eq:u_limit_nonhom}}
			\\[4pt]
			v(y,t)
			&=
			\frac{\op_0\sin(\beta y)}
			{D\beta\cos(\beta H)-a\lambda\sin(\beta H)}
			e^{-\lambda t},
			\qquad
			0<y<H.
			\tag{\ref{eq:v_limit_nonhom}}
		\end{align}
		
		The two cases differ in an essential respect. In the homogeneous case the decay rates $\lambda_1^\varepsilon$ and $\lambda_1^0$ are determined by the characteristic equations \eqref{eq:char_finite_hom_det} and \eqref{eq:char_limit_hom_det}, and therefore differ between the two configurations, together with the associated wave numbers and amplitudes. In the non-homogeneous case, instead, the decay rate $\lambda$ is prescribed by the boundary forcing and is the same in both configurations, so that the whole discrepancy is carried by the amplitudes. The two mechanisms through which the finite-thickness configuration approaches its thin-film limit are thus different, and are quantified by \eqref{eq:lambda_expansion} and \eqref{eq:amplitude_expansion}, respectively.
		
		Several qualitative observations emerge directly from the analytical expressions. In the finite-thickness film configuration, the temperature field within the film is a volumetric eigenfunction possessing a non-trivial spatial structure across the film thickness, characterized by the additional wave number $\alpha$, which is absent from the reduced formulation. Conversely, in the thin-film limit the film no longer possesses an internal spatial coordinate, and its thermal contribution is entirely represented through the interface temperature. The dimensional reduction therefore transforms the volumetric thermal response of the film into an effective interfacial contribution, while preserving the coupling with the bulk region.
		
		Although the film contribution undergoes a substantial structural modification, the bulk temperature field retains the same sinusoidal spatial dependence in both formulations. The differences arise through the interface coupling conditions, which determine the admissible modal decay rates and the associated amplitudes.
		
		It is also worth observing how the volumetric structure of the film disappears in the limit. Since $\alpha\varepsilon=O(\varepsilon)$, expanding \eqref{eq:u_finite_hom} across the film thickness gives
		\[
		u_\varepsilon(y,t)
		=
		A_1^\varepsilon\Bigl(1+O(\varepsilon^2)\Bigr)e^{-\lambda_1^\varepsilon t},
		\qquad
		H<y<H+\varepsilon,
		\]
		so that the temperature profile within the film becomes uniform at second order in the thickness, and the film variable degenerates into a single interfacial value. The same holds for the non-homogeneous case, where the additional contribution proportional to $\sin(\alpha(y-H-\varepsilon))$ in \eqref{eq:u_finite_nonhom} is itself of order $\varepsilon^2$, its amplitude being $O(\varepsilon)$ and its argument $O(\varepsilon)$ as well.
		
		The analytical solutions therefore provide a direct realization of the asymptotic reduction process established in Section~\ref{sec:thinfilm_analysis}: as the film thickness decreases, the volumetric thermal response within the film progressively collapses onto the interface and the finite-thickness solution approaches the reduced bulk-interface formulation.

		\paragraph{\emph{\textbf{Temperature profiles}}}
		
		To illustrate the asymptotic behavior discussed above, the analytical solutions are evaluated for progressively decreasing values of the film thickness and compared with the corresponding thin-film limit solutions.
		
		For the numerical illustrations, the geometric and material parameters are normalized according to
		\(
		H=1,\ a=c=B=D=1,
		\)
		so that the assumption $c=a$ adopted in Section~\ref{mat_sec4} is satisfied. For the homogeneous case, the principal eigenvalues are computed from the characteristic equations \eqref{eq:char_finite_hom_det} and \eqref{eq:char_limit_hom_det}; in the thin-film limit the latter reduces to $\beta\tan(\beta H)=1$, which yields
		$\beta_1^0\simeq0.860334$ and $\lambda_1^0\simeq0.740174$.
		The amplitudes are normalized by setting
		\(
		A_1^\varepsilon=A_1^0=1;
		\)
		this is a choice of normalization, and the fact that it makes the initial data of the two configurations differ is addressed at the end of the present section.
		For the non-homogeneous case, the prescribed heat-flux amplitude is chosen as
		\(
		\op_0=1,
		\)
		the remaining amplitudes being then determined by \eqref{eq:amplitudes_finite_nonhom} and \eqref{eq:C_limit_nonhom}, and the decay rate is prescribed as
		\(
		\lambda=0.25\,\lambda_1^0\simeq0.185.
		\)
		This value is smaller than $\lambda_1^\varepsilon$ for every thickness considered below, so that it is bounded away from the resonances of both configurations, as required in Section~\ref{nonhomogeneus}, and the forced response dominates the homogeneous modes at the time at which the profiles are evaluated.
		
		The temperature profiles are evaluated at the representative time instant \(t=0.15\), for the three film thicknesses
		\(
		\varepsilon=1,\,0.5,\,0.01 .
		\)
		The first two values are deliberately well outside the asymptotic regime, the film being there comparable in size with the bulk region, and are included in order to display the departure from the reduced model; only the third one is representative of a thin film.
		
		Figures~\ref{fig:hom_profiles} and~\ref{fig:nonhom_profiles} show the temperature distributions obtained for the homogeneous and the non-homogeneous Neumann conditions, respectively. In both of them the height $y$ is reported on the vertical axis, consistently with the geometry of Figure~\ref{fig:simplified_domain}.
		
		\begin{figure}[htbp]
			\centering
			\includegraphics[width=0.72\textwidth]{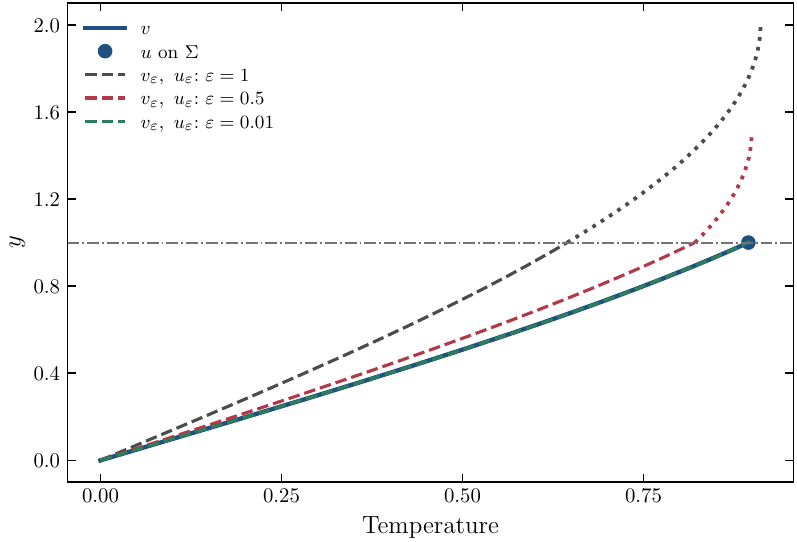}
			\caption{Homogeneous Neumann condition ($\op=0$). Temperature profiles at $t=0.15$ for the thin-film limit, shown as a solid line in the bulk region together with a marker on the interface, and for the finite-thickness film configuration at three values of the film thickness, shown as dashed lines in the bulk region and dotted lines in the film region. The height $y$ is reported on the vertical axis, so that the film occupies the strip above the dash-dotted line marking the interface $y=H$. As $\varepsilon$ decreases the film profile flattens and the bulk profile approaches the reduced solution.}
			\label{fig:hom_profiles}
		\end{figure}
		
		Figure~\ref{fig:hom_profiles} illustrates the homogeneous Neumann case. For the two largest thicknesses, noticeable differences are observed between the finite-thickness and the thin-film solutions, both within the film region and in the bulk domain; at the interface, the finite-thickness temperature falls below the limiting value $u$ by about $28\%$ for $\varepsilon=1$ and by about $8\%$ for $\varepsilon=0.5$. As the film thickness decreases the finite-thickness profiles progressively approach the reduced solution, and for $\varepsilon=0.01$ the two become indistinguishable on the scale of the figure.
		
		The convergence of the bulk solution is accompanied by a progressive flattening of the temperature distribution within the film, in agreement with the expansion of \eqref{eq:u_finite_hom} obtained above: the temperature variation across the film thickness is $A_1^\varepsilon\bigl(1-\cos(\alpha_1^\varepsilon\varepsilon)\bigr)=O(\varepsilon^2)$, so that the volumetric structure of the film is lost at second order. This behavior anticipates the transition toward the interface description characterizing the thin-film limit.
		
		\begin{figure}[htbp]
			\centering
			\includegraphics[width=0.72\textwidth]{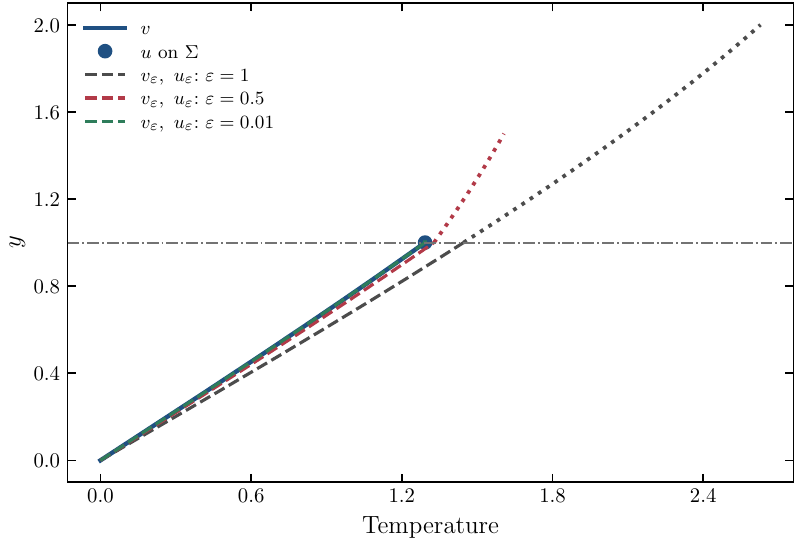}
			\caption{Non-homogeneous Neumann condition ($\op\neq0$). Temperature profiles at $t=0.15$ for the thin-film limit and for the finite-thickness film configuration, with the same line conventions as in Figure~\ref{fig:hom_profiles}. The prescribed heat flux produces a marked temperature gradient across the film, which progressively vanishes as the film thickness decreases.}
			\label{fig:nonhom_profiles}
		\end{figure}
		
		Figure~\ref{fig:nonhom_profiles} presents the corresponding comparison for the non-homogeneous Neumann condition. The prescribed heat flux sustains a temperature increase across the film region, which is therefore no longer isothermal: for the largest thicknesses the film exhibits a pronounced gradient, whereas the bulk temperature stays considerably closer to the corresponding thin-film solution than in the homogeneous case. This is consistent with the fact that here the two configurations share the same decay rate, the discrepancy being carried by the amplitudes alone. As the film thickness decreases, the temperature variation within the film progressively diminishes and the finite-thickness solution converges toward the bulk-interface formulation.
		
		Despite the different thermal responses observed in the two cases, both Figures~\ref{fig:hom_profiles} and~\ref{fig:nonhom_profiles} exhibit the same asymptotic trend: decreasing values of the film thickness systematically reduce the discrepancy between the finite-thickness and the thin-film formulations, indicating that the convergence mechanism is robust with respect to the boundary condition imposed on the film boundary. The next paragraph quantifies this trend.
		\paragraph{\emph{\textbf{Convergence analysis}}}
		
		To quantify the convergence process, we introduce error measures reproducing two of the norms appearing on the left-hand side of the theoretical estimate \eqref{rate2}. Accordingly, the discrepancy in the bulk region is measured in the $L^\infty(0,T;L^2(\Omega_b))$ norm,
		\begin{align}
			E_v(\varepsilon)
			=
			\sup_{t\in(0,T)}
			\left(
			\int_0^H
			\left|
			v_\varepsilon(y,t)-v(y,t)
			\right|^2
			\,\d y
			\right)^{1/2},
			\label{eq:error_bulk}
		\end{align}
		whereas the discrepancy in the film region is measured in the rescaled $L^\infty(0,T;L^2(\Omega_\varepsilon))$ norm,
		\begin{align}
			E_u(\varepsilon)
			=
			\varepsilon^{-1/2}
			\sup_{t\in(0,T)}
			\left(
			\int_H^{H+\varepsilon}
			\left|
			u_\varepsilon(y,t)-u(t)
			\right|^2
			\,\d y
			\right)^{1/2}.
			\label{eq:error_interface}
		\end{align}
		In \eqref{eq:error_interface} the finite-thickness solution is compared with the constant extension of the interfacial temperature along the normal direction, which in the present one-dimensional setting reduces to $u(t)$ itself. The rescaling by $\varepsilon^{-1/2}$ compensates for the vanishing measure of the film region, exactly as in \eqref{rate2}. The supremum in time is taken over the horizon $T=1$ and replaces the evaluation at a single time instant, so that both quantities are directly comparable with the corresponding terms of \eqref{rate2}; the two remaining terms of \eqref{rate2}, involving the $L^2(0,T;H^1)$ norms, are not considered here.
		
		Since the solutions being compared are available in closed form, both error measures are evaluated exactly, up to the root-finding needed to determine $\lambda_1^\varepsilon$ from \eqref{eq:char_finite_hom_det} and the quadrature of the integrals in \eqref{eq:error_bulk}--\eqref{eq:error_interface}. No spatial or temporal discretization of the transmission problem is involved, so that the behavior reported below is not affected by approximation errors.
		
		\begin{figure}[htbp]
			\centering
			\includegraphics[width=0.72\textwidth]{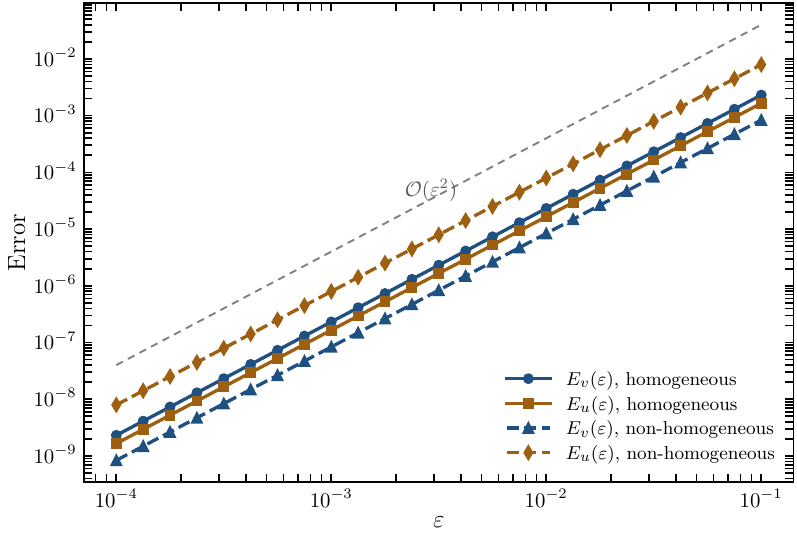}
			\caption{Convergence errors \eqref{eq:error_bulk} and \eqref{eq:error_interface} for the homogeneous and the non-homogeneous Neumann conditions, reported in logarithmic scale as functions of the film thickness. The dashed grey line indicates the reference slope $\mathcal O(\varepsilon^2)$.}
			\label{fig:convergence_errors}
		\end{figure}
		
		The convergence errors are reported in Figure~\ref{fig:convergence_errors}. For both boundary conditions the errors decrease monotonically as the film thickness tends to zero, and the four curves are parallel to the reference slope $\mathcal O(\varepsilon^2)$ over the whole range of thicknesses considered. In the asymptotic regime they are well described by
		\[
		E_v(\varepsilon)\simeq0.231\,\varepsilon^2,
		\qquad
		E_u(\varepsilon)\simeq0.166\,\varepsilon^2
		\]
		for the homogeneous condition, and by
		\[
		E_v(\varepsilon)\simeq0.084\,\varepsilon^2,
		\qquad
		E_u(\varepsilon)\simeq0.794\,\varepsilon^2
		\]
		for the non-homogeneous one, the four constants being independent of $\varepsilon$ to four significant digits over the three decades investigated.
		
		The two boundary conditions therefore differ only through the multiplicative constants: the prescribed heat flux reduces the discrepancy in the bulk region by a factor of about $2.8$, while increasing the discrepancy within the film by a factor of about $4.8$. This is consistent with the profiles of Figure~\ref{fig:nonhom_profiles}, where the boundary forcing acts directly on the film and leaves the bulk response closer to its limiting counterpart, and with the fact that in the non-homogeneous case the two configurations share the same decay rate, so that the bulk fields differ only through their amplitudes.
		
		To estimate the convergence order, the local convergence rates are computed according to
		\begin{align}
			p(\varepsilon_i)
			=
			\frac{
				\log\!\left(
				E(\varepsilon_i)
				/E(\varepsilon_{i+1})
				\right)}
			{
				\log\!\left(
				\varepsilon_i
				/\varepsilon_{i+1}
				\right)
			},
			\label{eq:numerical_rate}
		\end{align}
		and are reported in Figure~\ref{fig:convergence_rates}.
		
		\begin{figure}[htbp]
			\centering
			\includegraphics[width=0.72\textwidth]{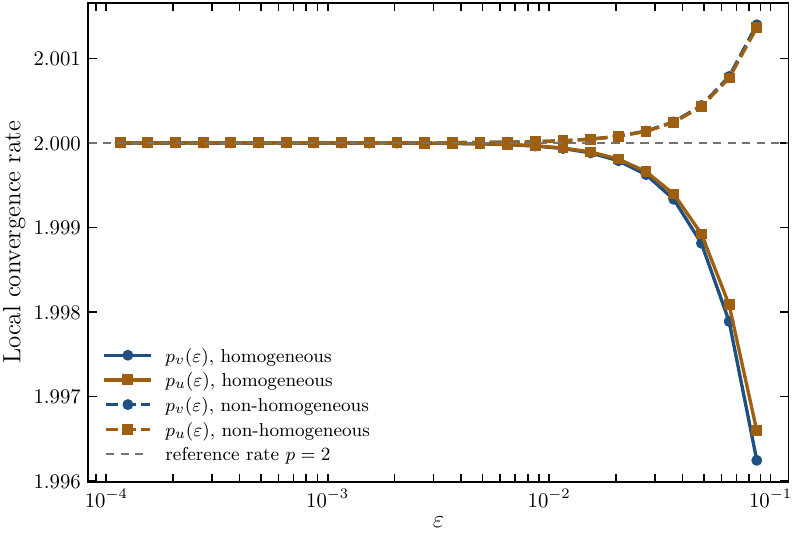}
			\caption{Local convergence rates \eqref{eq:numerical_rate} associated with the error measures \eqref{eq:error_bulk} and \eqref{eq:error_interface}, for the homogeneous and the non-homogeneous Neumann conditions. The dashed horizontal line corresponds to the reference quadratic rate $p=2$. Note the narrow range of the vertical axis: all rates lie within $4\times10^{-3}$ of the reference value.}
			\label{fig:convergence_rates}
		\end{figure}
		
		Figure~\ref{fig:convergence_rates} shows that the local convergence rates remain extremely close to the reference value \(p=2\) throughout the investigated range of film thicknesses, the deviations being confined to the largest values of \(\varepsilon\), where the higher-order terms of the expansions are still visible. The two boundary conditions approach the quadratic rate from opposite sides, the homogeneous case from below and the non-homogeneous one from above, which reflects the opposite sign of the corresponding fourth-order corrections; in both cases the rates converge to $p=2$ as the thickness decreases.
		
		This behavior is precisely the one predicted by the expansions obtained in Sections~\ref{homogeneus} and~\ref{nonhomogeneus}. In the homogeneous case the decay rate satisfies \eqref{eq:lambda_expansion}, and the associated wave numbers and amplitudes are perturbed at the same order, so that the whole fundamental mode differs from its limiting counterpart at order $\varepsilon^2$. In the non-homogeneous case the decay rate is common to the two configurations and the discrepancy is carried by the amplitudes, which satisfy \eqref{eq:amplitude_expansion}. In both cases, therefore,
		\[
		E_v(\varepsilon)=\mathcal O(\varepsilon^2),
		\qquad
		E_u(\varepsilon)=\mathcal O(\varepsilon^2),
		\]
		and the errors reported in Figure~\ref{fig:convergence_errors} confirm this order quantitatively.
		
		\paragraph{\emph{\textbf{Comparison with the theoretical convergence rate}}}
		
		The rates observed above are now compared with the theoretical estimate \eqref{rate2}. The quantities $E_v$ and $E_u$ reproduce, respectively, the third and the first term appearing on the left-hand side of \eqref{rate2}, so that the two convergence analyses refer to the same norms. The right-hand side of \eqref{rate2} consists instead of two contributions, namely a term of order $\varepsilon$ and the initial-datum term $\varepsilon^{-1/2}\norm{u_{0,\varepsilon}-\widetilde{u_0}}_{L^2(\Omega_\varepsilon)}$.
		
		The behavior of the latter must be assessed explicitly, since the amplitudes of the two configurations are prescribed independently and the corresponding initial data therefore do not coincide. Its order is, however, the same as that of the film profile itself: since $\alpha\varepsilon=O(\varepsilon)$, the film temperature differs from its interfacial value by $O(\varepsilon^2)$ uniformly across the thickness, as shown above, so that
		\[
		\varepsilon^{-1/2}
		\norm{u_{0,\varepsilon}-\widetilde{u_0}}_{L^2(\Omega_\varepsilon)}
		=
		\mathcal O(\varepsilon^2),
		\]
		with constants $0.166$ and $0.794$ in the homogeneous and in the non-homogeneous case, respectively. The right-hand side of \eqref{rate2} is therefore of order $\varepsilon$, and the initial datum is not the limiting contribution. The convergence observed in the present simplified setting is thus faster than the one guaranteed by \eqref{rate2}. The two results are consistent, since \eqref{rate2} provides an upper bound valid, in the geometric setting of Section~\ref{mat_sec1}, for general anisotropic coefficients and general data, and is not expected to be optimal for a configuration as particular as the present one, which, as observed in Section~\ref{simply1}, does not fall within that setting.
		
		The origin of the faster rate lies in the way the finite-thickness configuration is perturbed away from its limit, which has been made explicit in Sections~\ref{homogeneus} and~\ref{nonhomogeneus}. In the homogeneous case the characteristic equation \eqref{eq:char_finite_hom} is an $O(\varepsilon^2)$ perturbation of \eqref{eq:char_limit_hom}, so that the principal eigenvalue satisfies \eqref{eq:lambda_expansion}; in the normalized setting considered here this gives
		\[
		\lambda_1^\varepsilon-\lambda_1^0
		\simeq
		-0.1333\,\varepsilon^2,
		\]
		the ratio $(\lambda_1^\varepsilon-\lambda_1^0)/\varepsilon^2$ being constant to seven significant digits for $\varepsilon\le10^{-3}$. The wave numbers and the amplitudes are perturbed at the same order, and consequently so is the whole fundamental mode. In the non-homogeneous case the decay rate is common to the two configurations, and the amplitudes satisfy \eqref{eq:amplitude_expansion}. The quadratic behavior of both error measures follows in either case.
		
		This is further supported by the structure of the proof of Theorem~\ref{thmrates}. In the present flat one-dimensional configuration the principal curvatures of $\Sigma$ vanish, so that $J\equiv1$ and the prescribed flux does not vary across the film thickness; the interfacial temperature is constant along $\Sigma$, so that $\nabla_\Sigma u=0$; and the volumetric sources are neglected, so that $g=0$. Several of the contributions that generate the term of order $\varepsilon$ in \eqref{rate2} therefore vanish identically in this configuration.
		
		Two further remarks are in order concerning the comparison itself. First, the supremum in \eqref{eq:error_interface} is attained at $t=0$, so that $E_u$ coincides with the initial-datum term appearing on the right-hand side of \eqref{rate2}; the comparison between the observed and the guaranteed rate is accordingly most informative for $E_v$. Second, Theorem~\ref{thmrates} assumes the two configurations to share the same initial datum in the bulk region, whereas the single-mode solutions considered here have different bulk eigenfunctions, since $\beta_1^\varepsilon\neq\beta_1^0$. The resulting mismatch is itself of order $\varepsilon^2$ and, by the linearity of the problem together with an energy estimate for \eqref{eq:finite_eps_problem_nonhom} analogous to \eqref{est1}, replacing one bulk initial datum by the other would modify both error measures by $\mathcal O(\varepsilon^2)$ as well, leaving the observed convergence order unaffected.
		
		Overall, the analytical solutions, the temperature profiles and the convergence analysis provide a coherent picture of the thin-film limit formulation. As the film thickness tends to zero, the finite-thickness transmission problem converges toward the reduced bulk-interface problem while preserving the dominant thermal behavior of the coupled film-bulk system, at a rate which is consistent with, and in this particular configuration faster than, the one established in Theorem~\ref{thmrates}.
			\section{Conclusions}
		\label{sec:conclusions}
		
		\mich{The present work investigated a thermal transmission problem in a film-bulk domain, considering both a finite-thickness film configuration and the corresponding thin-film limit. The analysis addressed the asymptotic transition from a volumetric description of the film region to an effective interfacial formulation, obtained as the film thickness tends to zero.}
		
		\mich{A direct passage to the limit was shown to lead to a degenerate formulation, since the measure of the film region vanishes together with its volumetric contribution. To overcome this limitation, suitable scaling assumptions were introduced for the normal heat fluxes, namely for the flux balance across the interface and for the flux prescribed on the external boundary of the film region. Under these assumptions the finite-thickness problem converges to a non-trivial bulk-interface formulation, in which the thermal effect of the collapsing film is retained through an additional balance equation posed on the interface. Such a formulation admits two equivalent readings: as a pair of coupled parabolic problems of different dimensions, or as a single problem posed in the bulk region and supplemented with a dynamic boundary condition on the interface.}
		
		\mich{The mathematical analysis established the well-posedness of both the rescaled finite-thickness problem and the corresponding thin-film limit problem, and proved that the solutions of the former converge to the solution of the latter as the thickness vanishes. The convergence was moreover quantified through an explicit rate in suitable norms, relying only on the limited regularity available for transmission problems. This quantitative result is the main contribution of the present work, and provides, to the best of our knowledge, a first justification of the asymptotic reduction for film-bulk configurations of this type.}
		
		\mich{To complement the theoretical developments, a simplified film-bulk geometry was considered, allowing the derivation of explicit analytical solutions for both homogeneous and non-homogeneous Neumann boundary conditions. The resulting closed-form expressions clarified the structural difference between the two formulations: in the finite-thickness configuration the film temperature is a volumetric field defined throughout the film region, whereas in the thin-film limit the thermal contribution of the film is entirely represented through an interfacial temperature coupled with the bulk solution.}
		
		\mich{The same expressions also made the asymptotic reduction quantitative, and showed that the two configurations approach each other through two distinct mechanisms. In the homogeneous case the characteristic equation of the finite-thickness problem is a second-order perturbation of the one associated with the limit, so that the principal eigenvalue, the wave numbers and the amplitudes are all perturbed at that order. In the non-homogeneous case the decay rate is instead common to the two configurations, and the whole discrepancy is carried by the amplitudes, again at second order. Measuring the discrepancy in the same norms used in the theoretical estimate \eqref{rate2}, the convergence of both the bulk and the interfacial temperatures is therefore quadratic in the film thickness, as confirmed by the evaluation of the corresponding errors.}
		
		\mich{This behavior is consistent with the theoretical analysis. Estimate \eqref{rate2} provides an upper bound of order $\varepsilon$ which holds, in the geometric setting of Section~\ref{mat_sec1}, for general anisotropic coefficients and data, and is therefore not expected to be optimal in the particular configuration considered here, which moreover does not fall within that setting. In the flat, one-dimensional setting of Section~\ref{sec:simplified} the curvature of the interface vanishes, the interfacial temperature is constant along the interface and the volumetric sources are neglected, so that several of the contributions generating the first-order term vanish identically. The initial-datum contribution appearing on the right-hand side of \eqref{rate2} is of order $\varepsilon^2$ as well, and is therefore not the limiting one. The observed quadratic rate is thus compatible with the theoretical estimate, and strictly faster than the bound it provides.}
		
		\mich{Overall, the results provide a mathematically consistent framework for the description of thermal transmission phenomena in film-bulk systems, and clarify how the thermal contribution of a geometrically thin region can be retained in the asymptotic limit through an effective interface formulation.}
		
		\mich{Several directions remain open. The asymptotic analysis was carried out, for simplicity, for a closed interface without boundary and in the absence of external conditions on the bulk boundary; the case of non-empty Dirichlet and Neumann portions, and more generally of an interface with boundary as in the setting of Section~\ref{sec:problem}, is left to future investigation. The convergence rate was established at first order in the film thickness under the structural condition $\mathbf B\bP=\bP\mathbf B$ on the interface, and, without that condition, at the same order up to an additional term measuring the discrepancy between $\bP\mathbf B$ and $\mathbf B\bP$ (cf.\ Remark~\ref{rem:noncommuting}). Sharpening the latter estimate remains open, as does the case of distinct volumetric heat capacities in the two regions. Finally, the numerical approximation of the full multidimensional problem, and the comparison of the resulting discrete solutions with the reduced model, constitute a natural continuation of the present analysis.}

		\section*{Acknowledgments}
		\noindent
		The present work is partially supported by the Italian Ministry of University and Research (MUR) through the Coll\_e\_Innov\_M.A.R.S.\_DIN58\_Carraturo project (“Modeling and Simulation of Architected Cellular Materials”) — project number E19I25001010007  —  as part of the research program at the Department of Civil Engineering and Architecture, University of Pavia. Mihaela Chiappetta and Ferdinando Auricchio gratefully acknowledge this support. 
		Andrea Poiatti gratefully acknowledges support from the Alexander von Humboldt Foundation and
		is a member of Gruppo Nazionale per l’Analisi Matematica, la Probabilità e le loro Applicazioni
		(GNAMPA) of Istituto Nazionale per l’Alta Matematica (INdAM).
		
		\section*{Conflict of interest}
		
		The authors declare that they have no known competing financial interests or personal relationships that could have appeared to influence the work reported in present paper.
		
		\section*{Declaration of Generative AI and AI-assisted technologies in the writing process}
		\noindent
		During the preparation of this work, the authors used AI-assisted language tools exclusively for language editing and proofreading purposes. 
		The authors carefully reviewed and revised the output and take full responsibility for all scientific content, including the methods, results, and conclusions presented in this work.
		\bibliographystyle{siam}
		\bibliography{thinfilm}
	\end{document}